%% file: steepest_lsh_clean.tex
\documentclass[11pt, twoside]{article}

\usepackage[margin=1.2in]{geometry}
\usepackage{color}
\definecolor{mydarkblue}{rgb}{0,0.08,0.45}
\usepackage[colorlinks=true,
    linkcolor=mydarkblue,
    citecolor=mydarkblue,
    filecolor=mydarkblue,
    urlcolor=mydarkblue]{hyperref}   

\usepackage[round]{natbib}

\usepackage[utf8]{inputenc} 
\usepackage[T1]{fontenc}    
\usepackage{url}            
\usepackage{booktabs}       
\usepackage{amsfonts}       
\usepackage{nicefrac}       
\usepackage{microtype}      

\usepackage{graphicx}
\graphicspath{{img/}}
\usepackage{amsmath,amsthm}
\usepackage{bm}
\usepackage{algorithm}
\usepackage{algorithmic}
\usepackage{enumitem}
\usepackage{multicol}

\usepackage{wrapfig}

\newcommand{\ignore}[1]{}

\newcommand{\seb}[1]{}
\newcommand{\sai}[1]{}
\newcommand{\praneeth}[1]{}
\newcommand{\martin}[1]{}
\newcommand{\as}[1]{}
\newcommand{\iseb}[1]{}

\newcommand{\primal}{\mathcal{P}}
\newcommand{\dual}{\mathcal{D}}
\newcommand{\deltav}{\Delta \alphav}
\newcommand{\scal}[2]{{\left\langle #1, #2\right\rangle}}
\newcommand{\sigmoid}[1]{{\dfrac{1}{1 + \exp(-#1)}}}
\newcommand{\defeq}{\mathrel{\overset{\makebox[0pt]{\mbox{\normalfont\tiny\sffamily def}}}{=}}}
\DeclareMathOperator{\gap}{gap}

\newcommand{\A}{\mathcal{A}}
\renewcommand{\P}{\mathcal{P}}
\newcommand{\B}{\mathcal{B}}
\newcommand{\Q}{\mathcal{Q}}

\newcommand{\G}{\mathcal{G}}
\newcommand{\shrinkage}{S}
\newcommand{\good}{{\texttt{good}}}
\newcommand{\bad}{{\texttt{bad}}}
\newcommand{\cross}{{\texttt{cross}}}
\newcommand{\gs}{{\texttt{GS}}}
\newcommand{\steepsub}{\textnormal{\texttt{GS{\tiny-}s}}}
\newcommand{\steeplook}{\textnormal{{\texttt{GS{\tiny-}q}}}}
\newcommand{\steeplookr}{\textnormal{{\texttt{GS{\tiny-}r}}}}
\newcommand{\falconn}{{\texttt{FALCONN}}}
\newcommand{\nmslib}{{\texttt{nmslib}}}

\DeclareMathOperator{\MIPS}{MIPS}

\mathchardef\mhyphen="2D 
\DeclareMathOperator{\SMIPS}{S\mhyphen MIPS}
 
\newcommand{\plus}{+}
\newcommand{\minus}{{-}}

\newcommand{\dhillon}{{\texttt{dhillon}}}
\newcommand{\lsh}{{\texttt{steepest{\tiny-}lsh}}}
\newcommand{\nms}{{\texttt{steepest{\tiny-}nn}}}
\newcommand{\steepest}{{\texttt{steepest}}}
\newcommand{\uniform}{{\texttt{uniform}}}

\newcommand{\rcv}{{\texttt{rcv1}}}
\newcommand{\makeregression}{{\texttt{make{\tiny-}regression}}}
\newcommand{\sector}{{\texttt{sector}}}
\newcommand{\wa}{{\texttt{w1a}}}

\newcommand{\ijcnn}{{\texttt{ijcnn1}}}

\newtheorem{definition}{Definition}
\newtheorem{theorem}{Theorem}

\newtheorem{lemma}{Lemma}

\newtheorem{remark}[definition]{Remark}

\DeclareMathOperator{\prox}{prox}

\DeclareMathOperator*{\E}{\mathbb{E}}
\DeclareMathOperator*{\argmin}{\arg\min}
\DeclareMathOperator*{\argmax}{\arg\max}

\DeclareMathOperator*{\sign}{sign}

\newcommand\norm[1]{\left\lVert#1\right\rVert}
\newcommand\floor[1]{\left\lfloor#1\right\rfloor}
\newcommand\ceil[1]{\left\lceil#1\right\rceil}
\newcommand\abs[1]{\left|#1\right|}
\newcommand\encase[1]{\left[ #1 \right]}
\newcommand\encaser[1]{\left( #1 \right)}
\newcommand\encasecurly[1]{\left\{ #1 \right\}}
\newcommand{\positive}[1]{\encase{#1}_+}
\newcommand{\where}{\ | \ }

\newcommand{\R}{\mathbb{R}}
\newcommand{\real}{\mathbb{R}}

\newcommand{\bv}{\bm{b}}
\newcommand{\cv}{\bm{c}}

\newcommand{\av}{\bm{a}}
\newcommand{\vv}{\bm{v}}
\newcommand{\gv}{\bm{g}}
\newcommand{\wv}{\bm{w}}
\newcommand{\xv}{\bm{x}}
\newcommand{\yv}{\bm{y}}
\newcommand{\sv}{\bm{s}}

\newcommand{\qv}{\bm{q}}
\newcommand{\pv}{\bm{p}}
\newcommand{\zv}{\bm{z}}
\newcommand{\alphav}{{\boldsymbol \alpha}}

\newcommand{\zetav}{{\boldsymbol \zeta}}

\newcommand{\Av}{\bm{A}}

\newcommand{\ind}[1]{\mathbf{1}_{\encasecurly{#1}}} 
\newcommand{\0}{\mathbf{0}} 
\newcommand{\unit}{\mathbf{e}} 
\newcommand{\one}{\mathbf{1}}

\def\gradient[#1]{{\dfrac{\partial}{\partial #1}}}

\allowdisplaybreaks

\title{Efficient Greedy Coordinate Descent for Composite Problems}
\author{
	Sai Praneeth Karimireddy\footnote{Equal contribution.}\\
	\texttt{sai.karimireddy@epfl.ch}
	\and
	Anastasia Koloskova$^*$\\
	\texttt{anastasia.koloskova@epfl.ch}
	\and
	Sebastian U. Stich\\
	\texttt{sebastian.stich@epfl.ch}
	\and
	Martin Jaggi\\
	\texttt{martin.jaggi@epfl.ch}
}
\begin{document}

\maketitle

\begin{abstract}
Coordinate descent with random coordinate selection is the current state of the art for many large scale optimization problems.
However, greedy selection of the steepest coordinate on smooth problems can yield convergence rates independent of the dimension $n$, and requiring upto $n$ times fewer iterations.

In this paper, we consider greedy updates that are based on subgradients for a class of non-smooth composite problems, which includes $L1$-regularized problems, SVMs and related applications. For these problems we provide (i) the first linear rates of convergence independent of $n$, and show that our greedy update rule provides speedups similar to those obtained in the smooth case. This was previously conjectured to be true for a stronger greedy coordinate selection strategy.

Furthermore, we show that (ii) our new selection rule can be mapped to instances of maximum inner product search, allowing to leverage standard nearest neighbor algorithms to speed up the implementation.
We demonstrate the validity of the approach through extensive numerical experiments.
\end{abstract}

\setlength{\abovedisplayskip}{6pt}
\setlength{\belowdisplayskip}{6pt}

\section{Introduction}\label{sec:intro}
In recent years, there has been increased interest in coordinate descent (CD) methods
due to their simplicity, low cost per iteration, and efficiency~\citep{wright_coordinate_2015}.
Algorithms based on coordinate descent are the state of the art for many optimization
problems~\citep{nesterov_efficiency_2012,shalev-shwartz_stochastic_2013,lin_accelerated_2014,shalev-shwartz_accelerated_2013-1,shalev-shwartz_accelerated_2016,richtarik_parallel_2016,fercoq_accelerated_2015,nesterov_efficiency_2017}.
Most of the CD methods draw their coordinates from a fixed distribution---for instance from the uniform distribution as in uniform coordinate descent~(UCD).
However, it is clear that significant improvements can be achieved by choosing more
\emph{important} coordinates more frequently~\citep{nesterov_efficiency_2012,nesterov_efficiency_2017,stich_approximate_2017,
StichSafeAdaptiveImportance2017,perekrestenko_faster_2017}. In particular, we could greedily choose the `best' coordinate at each iteration i.e. the steepest coordinate descent (SCD).

\paragraph{SCD for composite problems.}
Consider the smooth quadtratic function $f(\alphav) \defeq \frac{1}{2}\norm{A\alphav - \bv}_2^2$. There are three natural notions of the `best' coordinate.\footnote{Following standard notation (cf.~\citep{nutini_coordinate_2015}) we call them the Gauss-Southwell (GS) rules.} One could choose (i) {\steepsub}: the \emph{steepest} coordinate direction based on (sub)-gradients,
(ii)~{\steeplookr}: the coordinate which allows us to take the largest step, and
(iii)~{\steeplook}: the coordinate that allows us to minimize the function value the most. For our example (and in general for smooth functions), the three rules are equivalent. When we add an additional non-smooth function to $f$, such as $g(\alphav) = \lambda \norm{\alphav}_1$, however, the three notions are no more equivalent. The performance of greedy coordinate descent in this composite setting is not well understood, and is the focus of this work.

\paragraph{Iteration complexity of SCD.}
If the objective $f$ decomposes into $n$ identical separable problems, then clearly SCD is identical to UCD. In all but such extreme cases, \citet{nutini_coordinate_2015} give a refined analysis of SCD for smooth functions and show that it outperforms UCD.
This lead to a renewed interest in greedy methods (e.g. \citep{karimi2016linear, you2016asynchronous, dunner2017efficient, song2017accelerated, nutini2017let, stich_approximate_2017, locatello2018matching, lu2018accelerating}). However, for the composite case the analysis in~\citep{nutini_coordinate_2015} of SCD methods for any of the three rules mentioned earlier falls back to that of UCD. Thus they fail to demonstrate the advantage of greedy methods for the composite case. In fact it is claimed that the rate of the {\steepsub} greedy rule may even be worse than that of UCD. In this work we provide a refined analysis of SCD for a certain class of composite problems, and show that all three strategies ({\steepsub}, {\steeplookr}, and {\steeplook}) converge on composite problems at a rate similar to SCD in the smooth case. Thus for these problems too, greedy coordinate algorithms are provably faster than UCD other than in extreme cases.

\paragraph{Efficiency of SCD.}
A na{\"i}ve implementation of SCD 
would require computing the full gradient at a cost roughly~$n$ times more than just computing one coordinate of the gradient as required by UCD.
This seems to negate any potential gain of SCD over UCD. The working principle behind \emph{approximate SCD} methods is to trade-off exactness of the greedy direction against the time spent to decide the steepest direction (e.g. \citep{stich_approximate_2017}).
For smooth problems, \citet{DhillonNearestNeighborbased2011} show that \emph{approximate nearest neighbor search} algorithms can be used to provide in \emph{sublinear time} an approximate steepest descent direction. We build upon these ideas and extend the framework to non-smooth composite problems, thereby capturing a significantly larger class of input problems. In particular we show how to efficiently map the {\steepsub} rule to an instance of maximum inner product search~($\MIPS$).

\paragraph{Contributions.}
We analyze and advocate the use of the {\steepsub} greedy rule to compute the update direction for composite problems.
Our main contributions are:\vspace{-2mm}
\begin{itemize}[leftmargin=15pt,itemsep=1pt]
 \item[i)] We show that on a class of composite problems, greedy coordinate methods achieve convergence rates which are very similar to those obtained for smooth functions, thereby extending the applicability of SCD.
  This class of problems covers several important applications such as SVMs (in its dual formulation), Lasso regression, $L1$-regularized logistic regression among others. With this we establish that greedy methods significantly outperform UCD also on composite problems, except in extreme cases (cf. Remark~\ref{rem:comparision-with-UCD}).
  \item[ii)] 
  We show that both the {\steepsub} as well as the {\steeplookr} rules achieve convergence rates which are (other than in extreme cases) faster than UCD. This sidesteps the negative results by \citet{nutini_coordinate_2015} for these methods through a more fine-grained analysis. We also study the effect of approximate greedy directions on the convergence.
 \item[iii)] Algorithmically, we show how to precisely map the {\steepsub} direction computation as a special instance of a maximum inner product search problem (MIPS).
 Many standard nearest neighbor algorithms such as e.g. Locality Sensitive Hashing (LSH) can therefore be used to efficiently run SCD on composite optimization problems.
 \item[iv)] We perform extensive numerical experiments 
 to study the advantages and limitations of our steepest descent combined with a current state-of-the-art MIPS implementation \citep{boytsov2013engineering}.
\end{itemize}

\paragraph{Related Literature.}
Coordinate descent, being one of the earliest known optimization methods, has a rich history (e.g. \citep{bickley1941relaxation, warga1963minimizing, bertsekas1989parallel, bertsekas1991some}). A significant renewal in interest followed the work of \citet{nesterov_efficiency_2012}, who provided a simple analysis of the convergence of UCD. In practice, many solvers (e.g. \citep{ndiaye2015gap,massias2018celer}) combine UCD with active set heuristics where attention is restricted to a subset of \emph{active} coordinates. These methods are orthogonal to, and hence can be combined with, the greedy rules studied here. Greedy coordinate methods can also be viewed as an `extreme' version of adaptive importance sampling \citep{stich_approximate_2017,perekrestenko_faster_2017}. However unlike greedy methods, even in the smooth case, there are no easily characterized function classes for which the adaptive sampling schemes or the active set methods are provably faster than UCD. The work closest to ours, other than the already discussed \citet{nutini_coordinate_2015}, would be that of \cite{DhillonNearestNeighborbased2011}. The latter show a sublinear $O(1/t)$ convergence rate for {\steeplookr} on composite problems. They also propose a practical variant for $L1$-regularized problems which essentially ignores the regularizer and is hence not guaranteed to converge.

\section{Setup}\label{sec:setup}

We consider composite optimization problems of the structure
\begin{align} \label{eq:general-nonsmooth}  \min_{\alphav \in \R^n} \Big[F(\alphav) :=  f(\alphav) + \sum_{i = 1}^n g_i(\alpha_i) \Big],
\end{align}
where $n$ is the number of coordinates, 
$f \colon \R^n \to \R$ is convex and smooth, and the $g_i \colon \R \to \R$, $i \in [n]$ are
convex and possibly non-smooth.
In this exposition, we further restrict the function $g(\alphav) := \sum_{i = 1}^n g_i(\alpha_i)$ to either enforce a \emph{box constraint} or an \emph{$L1$ regularizer}. This comprises many important problem classes, for instance dual SVM or Lasso regression, see Appendix~\ref{subsec:applications}.

We further assume that the smooth component $f$ is coordinate-wise $L$ smooth: for any $\alphav$, $\gamma$ and $i$,
  \begin{equation}\label{eq:coord-smoothness}
    f(\alphav + \gamma \unit_i) \leq f(\alphav)
    + {\nabla_i f(\alphav)}{\gamma}
    + \frac{L \gamma^2}{2}\,.
  \end{equation}
Sometimes we will assume that $f$ is in addition also strongly convex with respect to the $\norm{\cdot}_p$ norm, $p \in \{1,2\}$, that is,
\begin{equation}
	    f(\alphav + \deltav) \geq f(\alphav)
    + \scal{\nabla f(\alphav)}{\deltav}
    + \frac{\mu_p}{2}\norm{\deltav}_p^2
\end{equation}
  for any $\alphav$ and $\alphav + \deltav$ in the domain of $F$.
  In general it holds $\mu_1 \in [\mu_2/n , \mu_2]$. See \citet{nutini_coordinate_2015} for a detailed comparison of the two constants.

\section{SCD for Non-Smooth Problems} \label{subsec:approx-steepest}
Here we briefly recall the definitions of the {\steepsub}, {\steeplookr} and {\steeplook} coordinate selection rules and introduce the approximate {\steepsub} rule that we will consider in detail.
\begin{align}
\steepsub(\alphav) &:= \argmax_j \Big[ \min_{s\in \partial g_j} \abs{\nabla_j
    f(\alphav) + s } \Big] \label{eq:prox-steepest} \,,\\
\steeplookr(\alphav) &:= \argmax_{j \in [n]} \abs{\gamma_j}\,,\\
\steeplook(\alphav) &:= \argmax_{j \in [n]} \abs{\chi_j(\alphav)}\,,
\end{align}
for an iterate $\alphav \in \R^n$, $\nabla_j f(\alphav) := \langle \nabla f(\alphav),\unit_j \rangle $ for standard unit vector $\unit_j$. Here $\chi_j(\alphav)$ and $\gamma_j$ are defined as the minimum value and minimizer respectively of
\begin{equation*}
  \min_{\gamma} \encase{\gamma\nabla_{j} f(\alphav)
  + \frac{L\gamma^2}{2} + g_{j}(\alpha_{j} + \gamma) - g_{j}(\alpha_{j})}\,.
\end{equation*}
We relax the requirement for an exact steepest selection, and define an approximate {\steepsub} rule.
\begin{definition}[$\Theta$-approximate {\steepsub}]\label{def:approx-steepest}
  For given $\alphav$, the coordinate $j$ is considered
  to be a $\Theta$-approximate steepest direction for $\Theta \in (0,1]$ if
  $$
  \min_{s\in \partial g_j} \abs{\nabla_j
  f(\alphav) + s } \geq \Theta\max_i \encase{\min_{s\in \partial g_i}
  \abs{\nabla_i f(\alphav) + s }}\,.
  $$
\end{definition}

\subsection{SCD for $L1$-regularized problems}
\label{sec:method}
\label{subsec:l1-method}
We now discuss the {\steepsub} rule for the concrete example of $L1$ problems, and collect some observations that we will use later to define the mapping to the $\MIPS$ instance. A similar discussion is included for box constrained problems in Appendix~\ref{sec:box-method}.

Consider $L1$-regularized problems of the form
	\begin{equation}\label{eq:l1_prob}
	   \min_{\alphav \in \R^n} \big[ F(\alphav) :=
     f(\alphav) + \lambda\norm{\alphav}_1 \big] \,.
	\end{equation}
The {\steepsub} steepest rule \eqref{eq:prox-steepest}
  and update rules can be simplified for such functions.
  Let $\sign(x)$ denote the sign function, and define $\shrinkage_{\lambda}(x)$
  as the shrinkage operator
  $$
    \shrinkage_\lambda(x) :=
    \begin{cases}
      x - \sign(x)\lambda, & \text{if } \abs{x} \geq \lambda \\
      0 & \text{otherwise}\,.
    \end{cases}
  $$
  Further, for any $\alphav$, let us define $\sv(\alphav)$ as
  \begin{equation}\label{eq:def-s-vector}
    s(\alphav)_i :=
    \begin{cases}
      \shrinkage_{\lambda}(\nabla_i f(\alphav)), & \text{if } \alpha_i = 0\\
      \nabla_i f(\alphav)  + \sign(\alpha_i)\lambda  & \text{otherwise} \,.
    \end{cases}
  \end{equation}

  \begin{lemma}\label{lem:l1-steepest}
    For any $\alphav$, the {\steepsub} rule is equivalent to
    \begin{equation}\label{eq:l1-steepest}
      \max_i \encase{\min_{s\in \partial g_i} \abs{\nabla_i
      f(\alphav) + \sv }} \equiv \max_{i}\abs{s(\alphav)_i}\,.
    \end{equation}
  \end{lemma}
  Our analysis of {\steepsub} rule requires bounding the number of `{\bad}' steps (to be detailed in Section \ref{sec:convergence-main}). For this, we will slightly modify the update of the coordinate descent method. Note that we still always follow the {\steepsub} direction, but will sometimes not perform the standard proximal coordinate update along this direction.
  To update the $i_t$-th coordinate, we either rely on the standard proximal step on the coordinate,
  \begin{equation}\label{eq:l1-intermediate-update}
    \alpha_{i_t}^{+} := \shrinkage_{\frac{\lambda}{L}}\encaser{\alpha_{i_t}^{(t)} -
    \frac{1}{L}\nabla_{i_t}f(\alphav^{(t)})}\,.
  \end{equation}
  or we perform line-search
  \begin{equation}\label{eq:l1-line-search}
    \alpha_{i_{t}}^{+} :=
          \argmin_{\gamma}F\big(\alphav^{(t)} + (\gamma - \alpha^{(t)}_{i_t})\unit_{i_t}\big)
  \end{equation}
  Finally, the $i_t$-th coordinate is updated as
  \begin{equation}\label{eq:l1-update}
	\alpha^{(t+1)}_i := \begin{cases}
	\alpha^{+}_i, & \text{if } \alpha^{+}_i \alpha^{(t)}_i \geq 0\\
	0, & \text{otherwise}
	\end{cases}.
	\end{equation}
  Our modification or `post-processing' step \eqref{eq:l1-update} ensures that the coordinate $\alpha_i$ can never `cross' the origin. This small change will later on help us bound the precise number of steps needed in our convergence rates (Sec.~\ref{sec:convergence-main}).
	The details are summarized in Algorithm \ref{alg:L1}.

\begin{algorithm}[h!]
   \caption{$L1$ Steepest Coordinate Descent}
   \label{alg:L1}
\begin{algorithmic}[1]
   \STATE {\bfseries Initialize:} $\alphav_0 := \0 \in \R^n$.
   \FOR{$t=0,1,\dots,\text{until convergence}$}
      \STATE Select coordinate $i_t$ as in {\steepsub}, {\steeplookr}, or {\steeplook}.
        \label{stp:l1-steepest}
      \STATE Find $\alpha_{i_t}^{+}$ via gradient \eqref{eq:l1-intermediate-update} or line-search \eqref{eq:l1-line-search}.
      \STATE Compute $\alpha^{(t+1)}_{i_t}$ as in \eqref{eq:l1-update}.
   \ENDFOR
\end{algorithmic}
\end{algorithm}

\section{Convergence Rates}\label{sec:convergence-main}

In this section, we present our main convergence results. 
\praneeth{This paragraph can probably be improved.}
We illustrate the novelty of our results in the important $L1$-regularized case: For strongly convex functions $f$, we provide the first linear rates of convergence independent of $n$ for greedy coordinate methods over $L1$-regularized problems, matching the rates in the smooth case.
In particular, for {\steepsub} this was conjectured to be impossible~\cite[Section H.5, H.6]{nutini_coordinate_2015} (see Remark \ref{rem:comparision-with-UCD}). We also show the sublinear convergence of the three rules in the non-strongly convex setting. Similar rates also hold for box-constrained problems.

\subsection{Linear convergence for strongly convex $f$}

\begin{theorem}
  \label{thm:informal-primal}
  	Consider an $L1$-regularized optimization problem \eqref{eq:l1_prob}, with $f$ being  coordinate-wise $L$ smooth, and $\mu_1$ strongly convex with respect to the $L1$ norm. After $t$ steps of Algorithm \ref{alg:L1} where the coordinate~$i_t$ is chosen using either the {\steepsub} , {\steeplookr}, or {\steeplook} rule,
	\[
		F(\alphav^{(t)}) - F(\alphav^{\star}) \leq \encaser{1 - \frac{\mu_1}{L}}^{\ceil{t/2}}\encaser{F(\alphav^{(0)}) - F(\alphav^{\star})}\,.
	\]
\end{theorem}

\begin{remark}
	The linear convergence rate of Theorem~\ref{thm:informal-primal} also holds for the $\Theta$-approximate {\steepsub} rule as in Definition \ref{def:approx-steepest}. In this case the $\mu_1$ will be multiplied by $\Theta^2$.
\end{remark}

\begin{remark}
	All our linear convergence rates can be extended to objective functions which only satisfy the weaker condition of proximal-PL strong convexity \citep{karimi2016linear}. 
\end{remark}

\begin{remark}\label{rem:comparision-with-UCD}
The standard analysis (e.g. in \citet{nesterov_efficiency_2012}) of UCD gives a convergence rate of
\[
	\E\big[F(\alphav^{(t)})\big] - F(\alphav^{\star}) \leq \encaser{1 - \frac{\mu_2}{nL}}^{t}\encaser{F(\alphav^{(0)}) - F(\alphav^{\star})}\,.
\]
Here $\mu_2$ is the strong convexity constant with respect to the $L2$ norm, which satisfies $\mu_1 \in [\mu_2/n,\mu_2]$.  The left boundary $\mu_1 = \mu_2/n$ marks the worst-case for SCD, resulting in convergence slower than UCD. It is shown in \citet{nutini_coordinate_2015} that this occurs only in extreme examples (e.g. when $f$ consists of $n$ identical separable functions). For all other situations when $\mu_1 \geq 2\mu_2/n$, our result shows that SCD is faster.
\end{remark}

\begin{remark}\label{rem:diff-rules}
Our analysis in terms of $\mu_1$ works for all three selection rules {\steepsub}, {\steeplookr}, or {\steeplook} rules. In \cite[Section H5, H6]{nutini_coordinate_2015} it was conjectured (but not proven) that this linear convergence rate holds for {\steeplook}, but that it cannot hold for {\steepsub} or {\steeplookr}. Example functions were constructed where it was shown that the single step progress of {\steepsub} or {\steeplookr} is much smaller than $1 - \mu_2/(nL)$. However these example steps were all {\bad} steps, as we will define in the following proof sketch, whose number we show can be bounded. 
\end{remark}

We state an analogous linear rate for the box-constrained case too, but refer to Appendix~\ref{sec:box-method} for the detailed algorithm and proof.
\begin{theorem}\label{thm:informal-box}
Suppose that $f$ is coordinate-wise $L$ smooth, and $\mu_1$ strongly convex with respect to the $L1$ norm, for problem \eqref{eq:general-nonsmooth} with $g$ encoding a box-constraint. After $t$ steps of Algorithm \ref{alg:box-steepest} (the box analogon of Algorithm \ref{alg:L1}) where the coordinate~$i_t$ is chosen using the {\steepsub} , {\steeplookr}, or {\steeplook} rule, then
	\begin{equation*}
	f(\alphav^{(t)}) - f(\alphav^{\star}) \leq\\ \encaser{1 - \max\encaser{\frac{1}{2n},\frac{\mu_1}{L}}}^{\ceil{t/2}}\encaser{f(\alphav^{(0)}) - f(\alphav^{\star})}\,.
	\end{equation*}
\end{theorem}
While the proof shares ideas with the $L1$-case, there are significant differences, e.g. the division of the steps into three categories: i) {\good} steps which give a $(1 - \mu_1/L)$ progress, ii) {\bad} steps which may not give much progress but are bounded in number, and a third iii)~{\cross} steps which give a $(1 - 1/n)$ progress.
\begin{remark}
For the box case, the greedy methods converge faster than UCD if $\mu_1 \geq 2\mu_2 /n$, as before, and if $\mu_2/L \leq 1/4$. Typically, $\mu_2/L$ is much smaller than 1 and so the second condition is almost always satisfied. Hence we can expect greedy to be much faster in the box case, just as in the unconstrained smooth case. 
It remains unclear if the $1/n$ term truly affects the rate of convergence. For example, in the separated quadratic case considered in \cite[Sec. 4.1]{nutini_coordinate_2015}, $\mu_1/L \leq 1/n$ and so we can ignore the $1/n$ term in the rate (see Remark \ref{rem:box-is-n-there} in the Appendix).
\end{remark}

\paragraph{Proof sketch.}
While the full proofs are given in the appendix, we here give a proof sketch of the convergence of Algorithm \ref{alg:L1} for $L1$-regularized problems in the strongly convex case, as in Theorem \ref{thm:informal-primal}. 

The key idea of our technique is to partition the iterates into two sets: {\good} and {\bad} steps depending on whether they make (provably) sufficient progress. Then we show that the modification to the update we made in \eqref{eq:l1-update} ensures that we do not have too many {\bad} steps. 
Since Algorithm \ref{alg:L1} is a descent method, we can focus only on the {\good} steps and describe its convergence. The ``contradiction'' to the convergence of {\steepsub} provided in~\cite[Section H.5, H.6]{nutini_coordinate_2015} are in fact instances of {\bad} steps.

\begin{figure}\vspace{0em}
  \begin{center}
      \hfill
  	\includegraphics[width=0.45\textwidth]{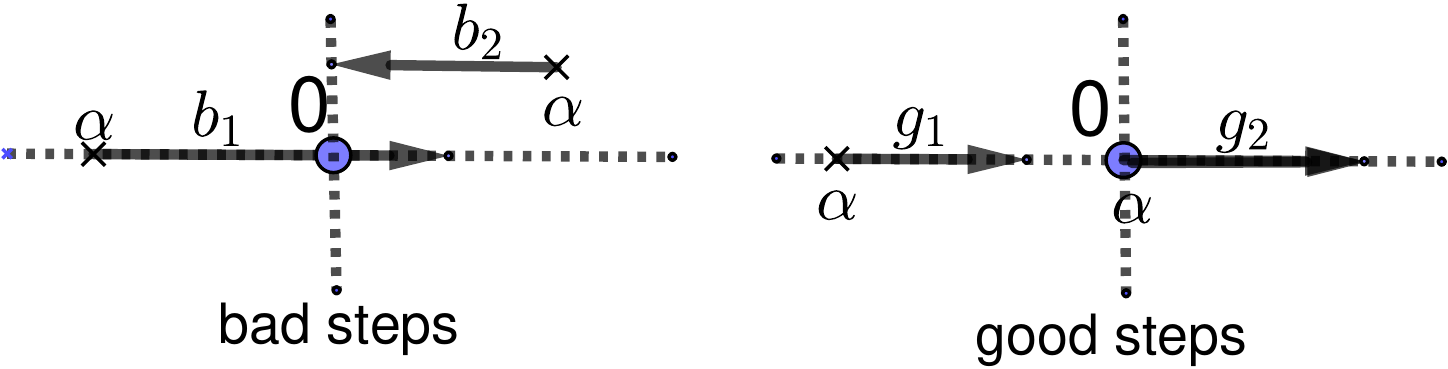}\hfill\null
  	\caption{ The arrows represent proximal coordinate updates $\shrinkage_\frac{\lambda}{L}(\alpha_i -\frac{1}{L}\nabla_i f(\alphav))$ from different starting points $\alphav$. Updates which `cross' ($b_1$) or `end at' ($b_2$) the origin are {\bad}, whereas the rest ($g_1$, $g_2$) are {\good}.}\vspace{-3mm}
  	\label{fig:good-bad-main}
  \end{center}
\end{figure}

The definitions of {\good} and {\bad} steps are explained in Fig.~\ref{fig:good-bad-main} (and formally in Def.~\ref{def:good-steps-L1}). The core technical lemma below shows that in a {\good} step, the update along the {\steepsub} direction has an alternative characterization. For the sake of simplicity, let us assume that $\Theta = 1$ and that we use the exact {\steepsub} coordinate.

\begin{lemma}\label{lem:characterize-informal}
  Suppose that iteration $t$ of Algorithm \ref{alg:L1} updates coordinate $i$ and that it was a {\good} step. Then
  \begin{equation*}
  F(\alphav^{(t+1)}) - F(\alphav^{(t)}) \leq  \min_{\wv \in \R^n}\Big\{\scal{\nabla f(\alphav^{(t)})}{\wv - \alphav^{(t)}} \\+ \frac{L}{2}\|\wv - \alphav^{(t)}\|_1^2 + \lambda(\norm{\wv}_1 - \|\alphav^{(t)}\|_1)\Big\}\,.
  \end{equation*}
\end{lemma}
\emph{Proof sketch.}
  We will only examine the case when $\alphav > 0$ here for the sake of simplicity. Combining this with the assumption that iteration $t$ was a {\good} step gives that both $\alpha_i > 0$, $\alpha_i^+ > 0$, and $\alpha_i^+ = \alpha_i - \frac{1}{L}(\nabla_j f(\alphav) + \lambda)$. Further if $\alphav > 0$, the {\steepsub} rule simplifies to $\argmax_{j \in [n]}\abs{\nabla_j f(\alphav) + \lambda}$.
  
  Since $f$ is coordinate-wise smooth \eqref{eq:coord-smoothness},
	\begin{align*}
		F(\alphav^+) - F(\alphav) &\leq 
		(\nabla_j f(\alphav))(\alpha_i^+ - \alpha_i) + \frac{L}{2}(\alpha_i^+ - \alpha_i)^2 + \lambda(\abs{\alpha_i^+} - \abs{\alpha_i}) \\
		&= -\frac{1}{L}(\nabla_j f(\alphav))(\nabla_j f(\alphav) + \lambda) + \frac{L}{2L^2}(\nabla_j f(\alphav) + \lambda)^2 - \frac{\lambda}{L}(\nabla_j f(\alphav) + \lambda)\\
		&= -\frac{1}{2L}(\nabla_j f(\alphav) + \lambda)^2\,.
	\end{align*}  
	But the {\steepsub} rule exactly maximizes the last quantity. Thus we can continue:
\begin{align*}
		F(\alphav^+) - F(\alphav) &\leq
		 -\frac{1}{2L}\norm{\nabla f(\alphav) + \lambda \one}_{\infty}^2\\
		 &= \min_{\wv \in \R^n}\encasecurly{\scal{\nabla f(\alphav) + \lambda \one}{\wv - \alphav} + \frac{L}{2}\norm{\wv - \alphav}_1^2}\\
		 &= \min_{\wv \in \R^n}\Big\{\scal{\nabla f(\alphav)}{\wv - \alphav} + \frac{L}{2}\norm{\wv - \alphav}_1^2 + \lambda(\scal{ \one}{\wv} - \scal{ \one}{\alphav})\Big\}\,.
	\end{align*}
Recall that $\alphav > 0$ and so $\norm{\alphav}_1 = \scal{\alphav}{\one}$. Further for any $x \in \R$, $\abs{x} \geq x$ and so $\scal{ \one}{\wv} \leq \norm{\wv}_1$. This means that
\[
	\lambda(\scal{ \one}{\wv} - \scal{ \one}{\alphav}) \leq \lambda(\norm{\wv}_1 - \norm{\alphav}_1)\,.
\]
Plugging this into our previous equation gives us the lemma. See Lemma \ref{lem:good-step-progress-L1} for the full proof.
\qed

If $\lambda = 0$ (i.e. $F$ is smooth), Lemma \ref{lem:characterize-informal} reduces to the `refined analysis' of \citet{nutini_coordinate_2015}. We can now state the rate obtained in the strongly convex case.

\emph{Proof sketch for Theorem \ref{thm:informal-primal}.}
Notice that if $\alpha_{i_t} = 0$, the step is necessarily {\good} by definition (see Fig. \ref{fig:good-bad-main}). Since we start at the origin $\0$, the first time each coordinate is picked is a {\good} step. Further, if some step $t$ is {\bad}, this implies that $\alpha_{i_t}^{+}$ `crosses' the origin. In this case our modified update rule \eqref{eq:l1-update} sets the coordinate $\alpha_{i_t}$ to 0. The next time coordinate $i_t$ is picked, the step is sure to be {\good}. Thus in $t$ steps, we have at least $\ceil{t/2}$ {\good} steps.

As per Lemma~\ref{lem:characterize-informal}, every {\good} step corresponds to optimizing the upper bound with the $L1$-squared regularizer. We can finish the proof:
\begin{align*}
	F(\alphav^{(t+1)}) - F(\alphav^{(t)}) &\leq  \min_{\wv \in \R^n}\Big\{\scal{\nabla f(\alphav^{(t)})}{\wv - \alphav^{(t)}} + \frac{L}{2}\|\wv - \alphav^{(t)}\|_1^2 + \lambda(\norm{\wv}_1 - \norm{\alphav^{(t)}}_1)\Big\}\\
	&\overset{(a)}{\leq}  \frac{\mu_1}{L}\min_{\wv \in \R^n}\Big\{\scal{\nabla f(\alphav^{(t)})}{\wv - \alphav^{(t)}} + \frac{\mu_1}{2}\|\wv - \alphav^{(t)}\|_1^2 + \lambda(\norm{\wv}_1 - \norm{\alphav^{(t)}}_1)\Big\}\\
	&\overset{(b)}{\leq} \frac{\mu_1}{L}\encaser{F(\alphav^{\star}) - F(\alphav^{(t)})} \,.
\end{align*}
Inequality $(a)$ follows from \citet[Lemma 9]{karimireddy2018adaptive}, and $(b)$ from strong convexity of $f$. Rearranging the terms above gives us the required linear rate of convergence.
\qed

\subsection{Sublinear convergence for general convex~$f$}
A sublinear convergence rate independent of $n$ for SCD can be obtained when $f$ is not strongly convex.
\begin{theorem}\label{thm:informal-general-convex}
	Suppose that $F$ is coordinate-wise $L$ smooth and convex, for $g$ being an $L1$-regularizer or a box-constraint. Also let $\Q^\star$ be the set of minima of $F$ with a minimum value $F^\star$. After $t$ steps of Algorithm \ref{alg:L1} or Algorithm \ref{alg:box-steepest} respectively, 
	where the coordinate~$i_t$ is chosen using the {\steepsub}, {\steeplookr}, or {\steeplook} rule,
\[
		F(\alphav^{(t)}) - F^\star \leq \mathcal{O}\Big(\frac{LD^2}{t}\Big)\,,
\]
where $D$ is the $L1$-diameter of the level set. For the set of minima $\Q^\star$,
 \[D = \max_{\wv \in \R^n}\min_{\alphav^\star\in\Q^\star} \big\{\norm{\wv - \alphav^\star}_1 \,\big|\, F(\wv) \leq F(\alphav^{(0)})\big\}\,.\]
\end{theorem}
While a similar convergence rate was known for the {\steeplookr} rule \citep{DhillonNearestNeighborbased2011}, we here establish it for all three rules---even for the approximate {\steepsub}.

\section{Maximum Inner Product Search}\label{subsec:approx-mips}
We now shift the focus from the theoretical rates to the actual implementation.
A very important observation---as pointed out by~\citet{DhillonNearestNeighborbased2011}---is that finding the steepest descent direction is closely related to a geometric problem.
As an example consider the function $f(\alphav)\defeq \frac{1}{2} \norm{A\alphav -\bv}^2$ for a data matrix $A \in \R^{d \times n}$. The gradient takes the form $\nabla f(\alphav) = A^\top \qv$ for $\qv = (A\alphav-\bv)$ and thus finding steepest coordinate direction is equal 
to finding the datapoint with the largest (in absolute value) inner product $\langle{\Av_i,\qv}\rangle$ with the query vector $\qv$, which a priori requires the evaluation of all $n$ scalar products. However, when we have to perform multiple similar queries (such as over the iterations of SCD), it is possible to pre-process the dataset $A$ to speed up the query time. Note that we do not require the columns $\Av_i$ to be normalized.

For the more general set of problems we consider here, we need the following slightly stronger primitive.

\begin{definition}[$\SMIPS$]\label{def:snns}
Given a set of $m$, $d$-dimensional points $\pv_1,\dots,\pv_m \in \R^d$,
the \emph{Subset Maximum Inner Product Search} or $\SMIPS$ problem is to pre-process the set $\P$ such that for any query vector $\qv$ and any subset of the points $\B \subseteq [m]$, 
 the best point $j \in \B$, i.e.
\begin{equation*}
  \SMIPS_{\P}(\qv;\B) := \argmax_{j \in \B}\encasecurly{\scal{\pv_j}{\qv}}\,,
\end{equation*}
can be computed with $o(m)$ scalar product evaluations.
\end{definition}

State-of-the-art algorithms relax the exactness assumption and compute an approximate solution in time equivalent to a \emph{sublinear} number of scalar product evaluations, i.e. $o(n)$ (e.g. \citep{CharikarSimilarityEstimationTechniques2002, lv2007multi, Shrivastava2014ww, Neyshabur2015uz, AndoniPracticalOptimalLSH2015}).
We consciously refrain from stating more precise running times, as these will depend on the actual choice of the algorithm and the parameters chosen by the user. Our approach in this paper is transparent to the actual choice of $\SMIPS$ algorithm, we only show how SCD steps can be \emph{exactly cast} as such instances. By employing an arbitrary solver one thus gets a sublinear time approximate SCD update. An important caveat is that in subsequent queries, we will \emph{adaptively} change the subset $\B$ based on the solution to the previous query. Hence the known theoretical guarantees shown for LSH do not directly apply, though the practical performance does not seem to be affected by this (see Appendix Fig.~\ref{fig:adaptivity}, \ref{fig:adaptivity-lsh}).
Practical details of efficiently solving $\SMIPS$ are provided in Section \ref{sec:experiments}.

\section{Mapping {\steepsub} to $\MIPS$}\label{sec:efficient}

We now move to our next 
 contribution and show how the {\steepsub} rule can be \emph{efficiently} implemented. We aim to cast the problem of computing the {\steepsub} update as an instance of $\MIPS$ (Maximum Inner Product Search), for which very fast query algorithms exist. 
In contrast, the {\steeplookr} and {\steeplook} rules do not allow such a mapping.
In this section, we will only consider objective functions of the following special structure:
\begin{equation}\label{eq:problem-structure}
		\min_{\alphav \in \real^n}\Big\{F(\alphav) \defeq \underbrace{l(A\alphav) + \cv^\top\alphav}_{f(\alphav)} + \sum_{i=1}^n g_i(\alpha_i)\Big\}\,.
\end{equation}
The usual problems such as Lasso, dual SVM, or logistic regression, etc. have such a structure (see Appendix~\ref{subsec:applications}).

\paragraph{Difficulty of the Greedy Rules.}
This section will serve to strongly motivate our choice of using the {\steepsub} rule over the {\steeplookr} or {\steeplook}. Let us pause to examine the three greedy selection rules and compare their relative difficulty.
As a warm-up, consider again the smooth function $f(\alphav) \defeq \frac{1}{2} \norm{A\alphav -\bv}^2$ for a data matrix $A \in \R^{d \times n}$ as introduced above in Section~\ref{subsec:approx-mips}. We have observed that the steepest coordinate direction is equal to
\vspace{-2pt}
\begin{align}\label{eq:example-simple-steepest}
 \argmax_{j \in [n]}\abs{\nabla_j f(\alphav)} \, \equiv \, \argmax_{j \in [n]} \max_{s \in \{-1,1\}}{\scal{s\Av_j}{\vv}}\,.
\end{align}
\vspace{-2pt}
The formulation on the right is an instance of $\MIPS$ over the $2n$ vectors $\pm \Av_i$.
Now consider a non-smooth problem of the form $F(\alphav) \defeq \frac{1}{2} \norm{A\alphav -\bv}^2 + \lambda \norm{\alphav}_1$. For simplicity, let us assume $\alphav > 0$ and $L=1$. In this case, the subgradient is $A^\top \vv + \lambda\one$ and the {\steepsub} rule is
\vspace{-2pt}
\begin{equation}\label{eq:example-prox-gss}
 \argmax_{j \in [n]}\min_{s \in \delta \abs{\alpha_j}}\abs{\nabla_i f(\alphav)  +s} \equiv \, \argmax_{j \in [n]} \max_{s \in \{-1,1\}}{\scal{s\Av_j}{\vv} + s\lambda}\,.
\end{equation}
\vspace{-2pt}
The rule \eqref{eq:example-prox-gss} is clearly not much harder than \eqref{eq:example-simple-steepest}, and can be cast as a $\MIPS$ problem with minor modifications (see details in Sec. \ref{subsec:efficient-L1}). 

Let $\alpha_j^+$ denote the proximal coordinate update along the $j$-th coordinate. In our case, $\alpha_j^+ = \shrinkage_{\lambda}(\alpha_j - \scal{\Av_j}{\vv})$. The {\steeplookr} rule can now be `simplified' as:
\vspace{-2pt}
\begin{equation}\label{eq:example-prox-gsr}
 \argmax_{j \in [n]}\abs{\alpha_j^+ - \alpha_j} \, \equiv   \argmax_{j \in [n]} 
 \left\{\!\begin{aligned}
 		\alpha_j,  \quad \mbox{if } \abs{\alpha_j - \scal{\Av_j}{\vv}} \leq \lambda\\
 		\sign(\alpha_j - \scal{\Av_j}{\vv}), \, \mbox{otherwise}\,.
\end{aligned}\right\}
\end{equation}
\vspace{-2pt}
It does not seem easy to cast \eqref{eq:example-prox-gsr} as a $\MIPS$ instance. It is even less likely that the {\steeplook} rule which reads
\[
	\argmin_{j \in [n]}\encasecurly{\nabla_j f(\alphav)(\alpha_j^+ \! - \! \alpha_j) \! + \! \tfrac{1}{2}(\alpha_j^+ \! - \! \alpha_j)^2 \! + \! \lambda(\abs{\alpha_j^+} \! - \! \alpha_j)}
\]
can be mapped as to $\MIPS$. This highlights the simplicity and usefulness of the {\steepsub} rule.

\subsection{Mapping $L1$-Regularized Problems}\label{subsec:efficient-L1}
Here we focus on problems of the form \eqref{eq:problem-structure} where $g(\alphav) = \lambda \norm{\alphav}_1$. 
Again, we have $\nabla f(\alphav) = A^\top \nabla l(\vv) + \cv$ where $\vv = A\alphav$. 

For simplicity, let $\alphav \neq 0$. Then the {\steepsub} rule in~\eqref{eq:l1-steepest} is
\begin{align}
	\argmax_{j \in [n]}\abs{s(\alphav)_j} &= \argmax_{j \in [n]}\abs{\scal{\Av_j}{\nabla l(\vv)} \! + \! c_j \! + \! \sign(\alpha_j)\lambda} \nonumber\\
&\hspace*{-0.5in}= 	\argmax_{j \in [n]}\max_{s \in \pm 1}s\encase{\scal{\Av_j}{\nabla l(\vv)} +  c_j + \sign(\alpha_j)\lambda}\label{eq:gss-example-case}\,.
\end{align}
We want to map the problem of the above form to a $\SMIPS$ instance. Define for some $\beta > 0$, vectors
\begin{equation}\label{eq:A-tilde-L1}
  \tilde{\Av}_j^\pm := \begin{pmatrix} \pm\beta,  & \beta c_j, & \Av_j \end{pmatrix}^\top \,,
\end{equation}
and form a query vector $\qv$ as
\begin{equation}\label{eq:query-L1-steepest}
\qv := \begin{pmatrix} \frac{\lambda}{\beta}, & \frac{1}{\beta}, & \nabla l(\vv) \end{pmatrix}^\top \,.
\end{equation}
A simple computation shows that the problem in \eqref{eq:gss-example-case} is equivalent to
\[
	\argmax_{j \in [n]} \, \max_{s \in \pm 1}\scal{s \tilde\Av_j^{\sign(\alpha_j)}}{\qv}\,.
\]
Thus by searching over a subset of vectors in $\{\pm \tilde\Av_j^\pm\}$, we can compute the {\steepsub} direction. Dealing with the case where $\alpha_j =0$ goes through similar arguments, and the details are outlined in Appendix~\ref{sec:proofs-mapping}. Here we only state the resulting mapping.

The constant $\beta$ in \eqref{eq:A-tilde-L1} and \eqref{eq:query-L1-steepest} is chosen to ensure that the entry is of the same order of magnitude on average as the rest of the coordinates of $\Av_i$. The need for $\beta$ only arises out of the performance concerns about the underlying algorithm to solve the $\SMIPS$ instance. For example, $\beta$ has no effect if we use exact search.

Formally, define the set $\P := \{\pm \tilde\Av_j^\pm: j \in [n]\}$.
Then at any iteration $t$ with current iterate $\alphav^{(t)}$, we also define $\B_t$ as $\B_t = \B_t^{1} \cup \B_t^{2} \cup \B_t^{3} \cup \B_t^{4}$, where \vspace{-1mm}
\begin{equation}
\begin{aligned}\label{eq:L1-Bt}
    \B_t^{1} &= \encasecurly{\tilde{\Av}_j^\plus: \alpha_j^{(t)} > 0}\,, &  \quad
    \B_t^{2} &= \encasecurly{-\tilde{\Av}_j^\plus: \alpha_j^{(t)} \geq 0}\,,\\
    \B_t^{3} &= \encasecurly{\tilde{\Av}_j^\minus: \alpha_j^{(t)} \leq 0}\,, & 
    \B_t^{4} &= \encasecurly{-\tilde{\Av}_j^\minus: \alpha_j^{(t)} < 0}\,.
\end{aligned}
\end{equation}

\begin{lemma}\label{lem:L1-steepest-efficient}
  At any iteration $t$, for $\P$ and $\B_t$ as defined in \eqref{eq:L1-Bt}, the query vector $\qv_t$ as in \eqref{eq:query-L1-steepest}, and $s(\alphav)$ as in~\eqref{eq:l1-steepest} then the following are equivalent for $f(\alphav)$ is of the form $l(A\alphav) + \cv^\top\alphav$:
  $$
    \SMIPS_{\P}(\qv_t;\B_t) = \argmax_{i}\abs{s(\alphav)_i}\,.\vspace{-1mm}
  $$
  \vspace{-3mm}
\end{lemma}
The sets $\B_t$ and $\B_{t+1}$ differ in at most four points since $\alphav^{(t)}$  and $\alphav^{(t+1)}$  differ only in a single coordinate. This makes it computationally very efficient to incrementally maintain $\B_{t+1}$ and $\alphav^{(t+1)}$ for $L1$-regularized problems.

\subsection{Mapping Box-Constrained Problems}
\vspace{-1mm}
Using similar ideas, we demonstrate how to efficiently map problems of the form \eqref{eq:problem-structure} where $g$ enforces box constraints, such as for the dual SVM. 
The detailed approach is provided in Appendix~\ref{subsec:efficient-box}.

\section{Experimental Results}
\label{sec:experiments}
\vspace{-2mm}
Our experiments focus on the standard tasks of Lasso regression, as well as SVM training (on the dual objective). We refer the reader to Appendix \ref{subsec:applications} for definitions. Lasso regression is performed on the {\rcv} dataset while SVM is performed on {\wa} and the {\ijcnn} datasets. All columns of the dataset (features for Lasso, datapoints for SVM) are normalized to unit length,
allowing us to use the standard cosine-similarity algorithms {\nmslib} \citep{boytsov2013engineering} to efficiently solve the $\SMIPS$ instances. Note however that our framework is applicable without any normalization, if using a general $\MIPS$ solver instead.

We use the \texttt{hnsw} algorithm of the {\nmslib} library with the default hyper-parameter value $M$ and other parameters as in Table \ref{tab:datasets}, selected by grid-search.\footnote{A short overview of how to set these hyper-parameters can be found at \url{https://github.com/nmslib/nmslib/blob/master/python_bindings/parameters.md}.} 
More details such as the meaning of these parameters can be found in the {\nmslib} manual \cite[pp.~61]{naidan2015non}.
We exclude the time required for pre-processing of the datasets since it is amortized over the multiple experiments run on the same dataset (say for hyper-parameter tuning etc.). All our experiments are run on an Intel Xeon CPU E5-2680 v3 (2.50GHz, 30 MB cache) with 48 cores and 256GB RAM. 
\begin{table}[]
\centering
\caption{Datasets and hyper-parameters: Lasso is run on {\rcv}, and SVM on {\wa} and {\ijcnn}. ($\mathbf{d}$, $\mathbf{n}$) is dataset size, the constant $\mathbf{\beta}$ from \eqref{eq:A-tilde-L1}, \eqref{eq:query-L1-steepest} is set to $50/\sqrt{n}$, {\nmslib} hyper-parameter $M$ is set as a default, $efC = 100$.}
\label{tab:datasets}\vspace*{0.1in}
\resizebox{0.5\linewidth}{!}{\small
\setlength{\tabcolsep}{2pt}
\begin{tabular}{lrrrrrr}
\hline
\textbf{Dataset} & {$\mathbf{n}$} & {$\mathbf{d}$} & {$\mathbb{\rho}$} & {efS} & {post} \\ \hline
{\rcv}, $\lambda = 1$         &  47,236          & 15,564  & 19\%      & 100 &  2 \\
{\rcv}, $\lambda = 10$         &  47,236          & 15,564   & 3\%       & 400 & 2\\
{\wa}          &  2,477           & 300           &     		 & 100 & 0 \\
{\ijcnn}       & 49,990         &  22           &      & 50 & 0\\ \hline
\end{tabular}
}\vspace*{0.1in}
\end{table}

First we compare the practical algorithm ({\dhillon}) of \citet{DhillonNearestNeighborbased2011}, which disregards the regularization part in choosing the next coordinate, and Algorithm~\ref{alg:L1} with {\steepsub} rule ({\steepest}) for Lasso regression. Note that {\dhillon} is not guaranteed to converge. To compare the selection rules without biasing on the choice of the library, we perform exact search to answer the $\MIPS$ queries.
As seen from Fig.~\ref{fig:dhillon-lambda}, {\steepest} significantly outperforms {\dhillon}.
In fact {\dhillon} stagnates (though it does not diverge), once the error $f(\alphav)$ becomes small and the $L1$ regularization term starts playing a significant role. Increasing the regularization $\lambda$ further worsens its performance.
This is understandable since the rule used by {\dhillon} ignores the $L1$ regularizer.

\begin{figure}
	\centering
	\includegraphics[width=0.4\linewidth]{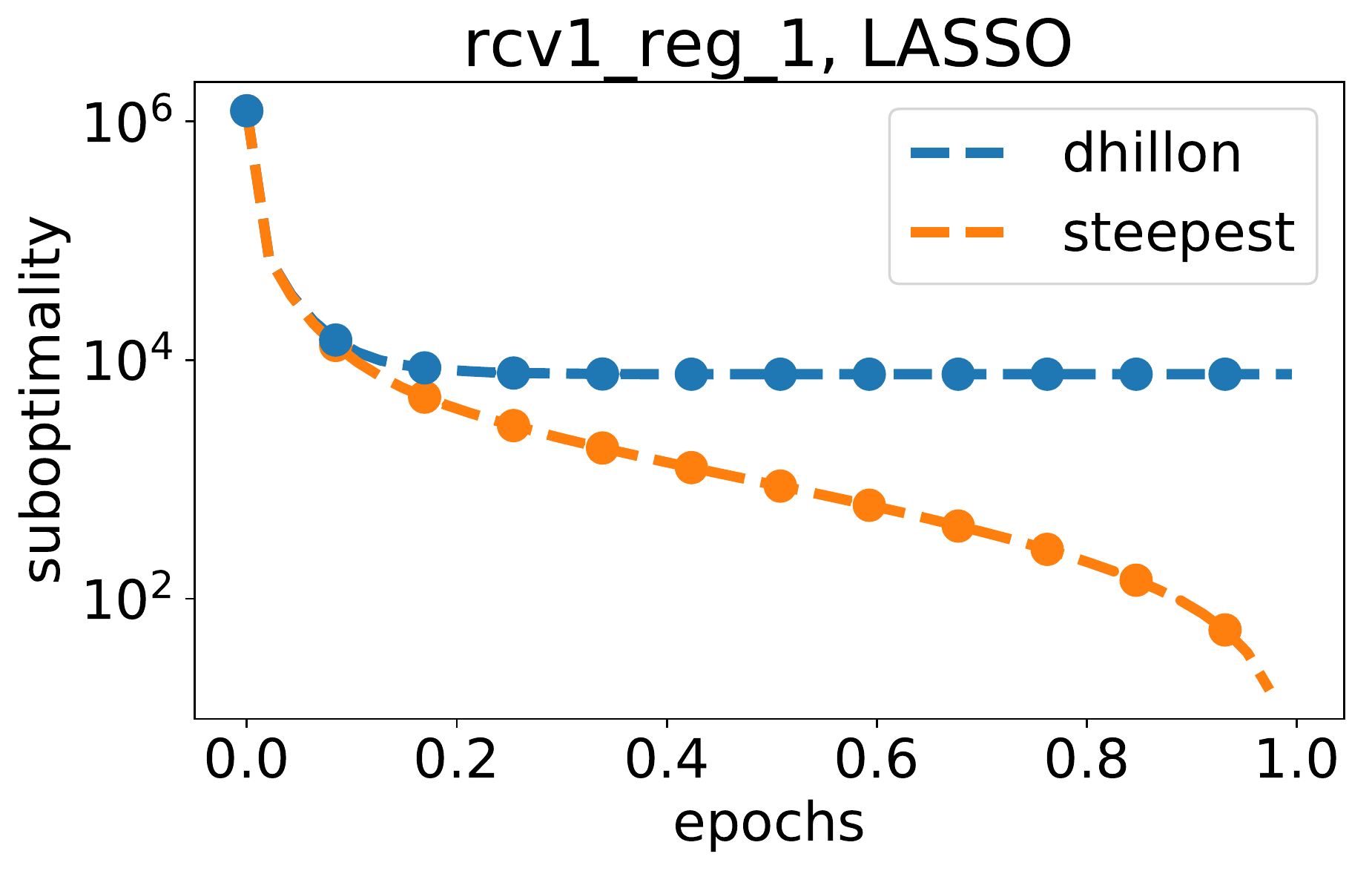}
	\includegraphics[width=0.4\linewidth]{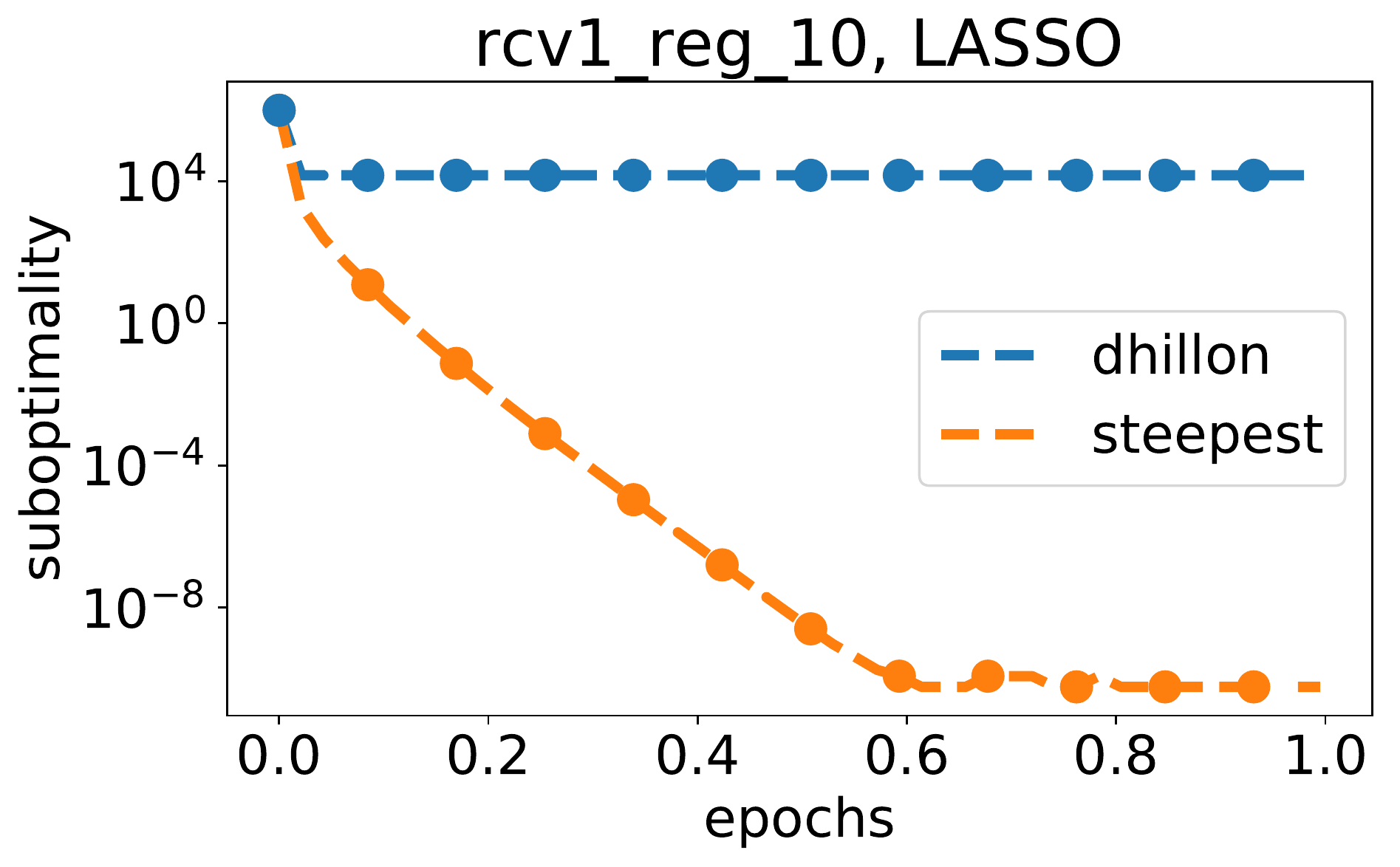}
	\caption{ Evaluating {\dhillon}: {\steepest} which is based on the {\steepsub} rule outperforms {\dhillon} which quickly stagnates. Increasing the regularization, it stagnates in even fewer iterations.}
	\label{fig:dhillon-lambda}
\end{figure}

Next we compare our {\steepest} strategy (Algorithms~\ref{alg:L1} and \ref{alg:box-steepest} using the {\steepsub} rule), and the corresponding nearest-neighbor-based approximate versions 
({\nms}) against {\uniform}, which picks coordinates uniformly at random.
 In all these experiments, $\lambda \in \{1,10\}$ for Lasso and at $1/n$ for SVM. 
 Fig.~\ref{fig:iterations} shows the clearly superior performance in terms of iterations of the {\steepest} strategy as well as {\nms} over {\uniform} for both the Lasso as well as SVM problems. However, towards the end of the optimization i.e. in high accuracy regimes, {\nms} fails to find directions substantially better than {\uniform}. This is because towards the end, all components of the gradient $\nabla f(\alphav)$ become small, meaning that the query vector is nearly orthogonal to all points---a setting in which the employed nearest neighbor library {\nmslib} performs poorly~\citep{boytsov2013engineering}.

\begin{figure}
	\centering
	\includegraphics[width=0.4\linewidth]{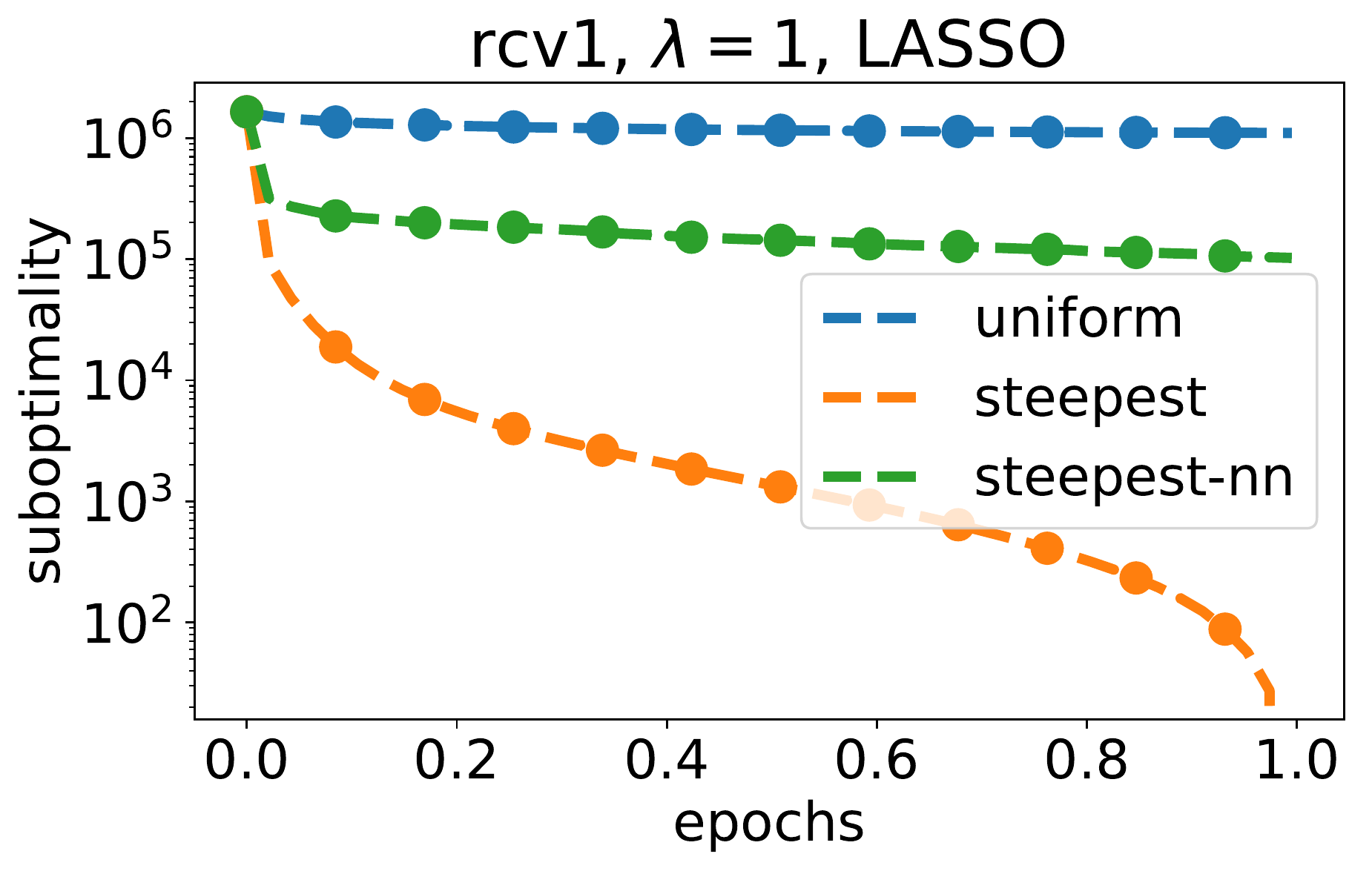}
	\includegraphics[width=0.4\linewidth]{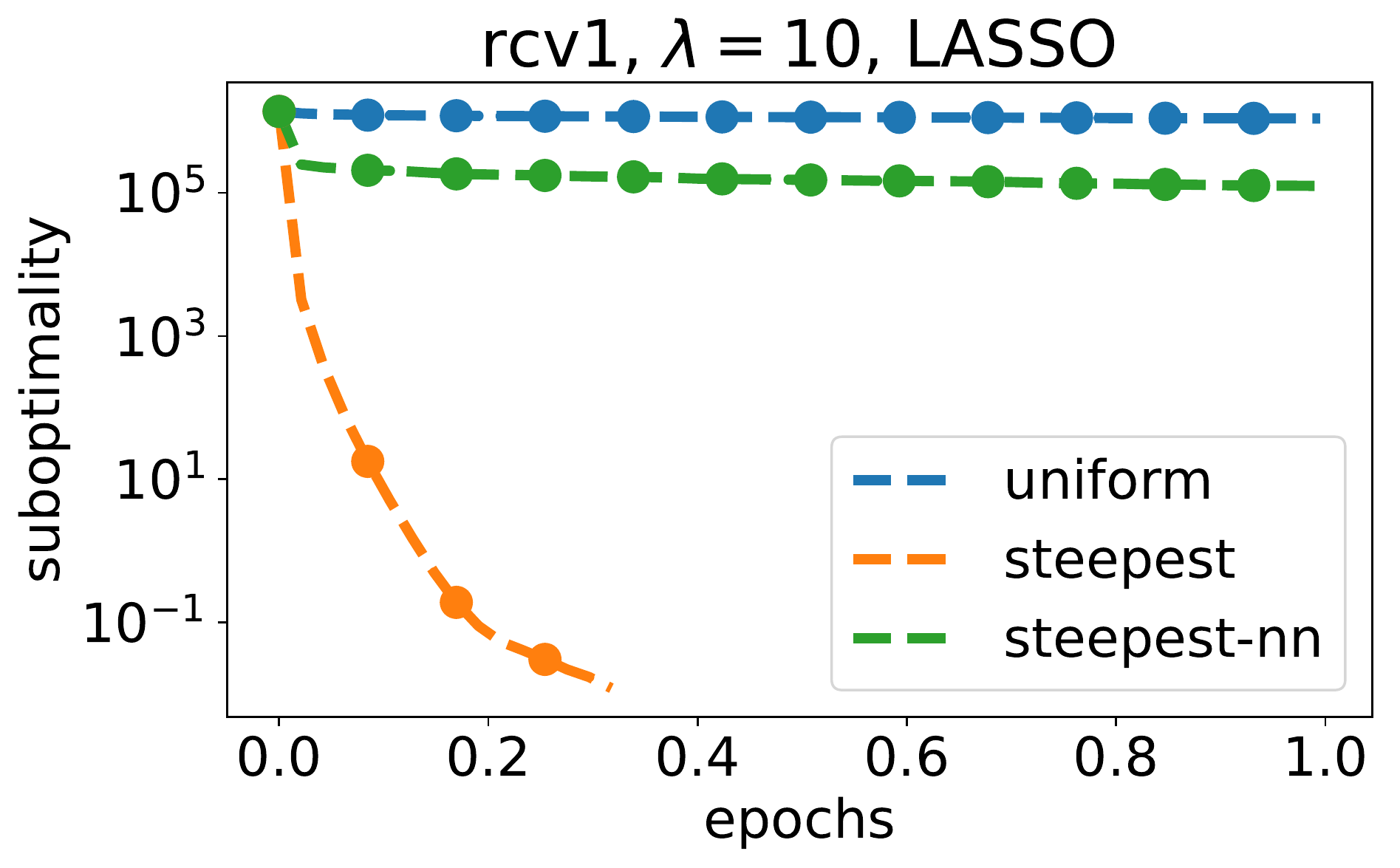}\\
	\includegraphics[width=0.4\linewidth]{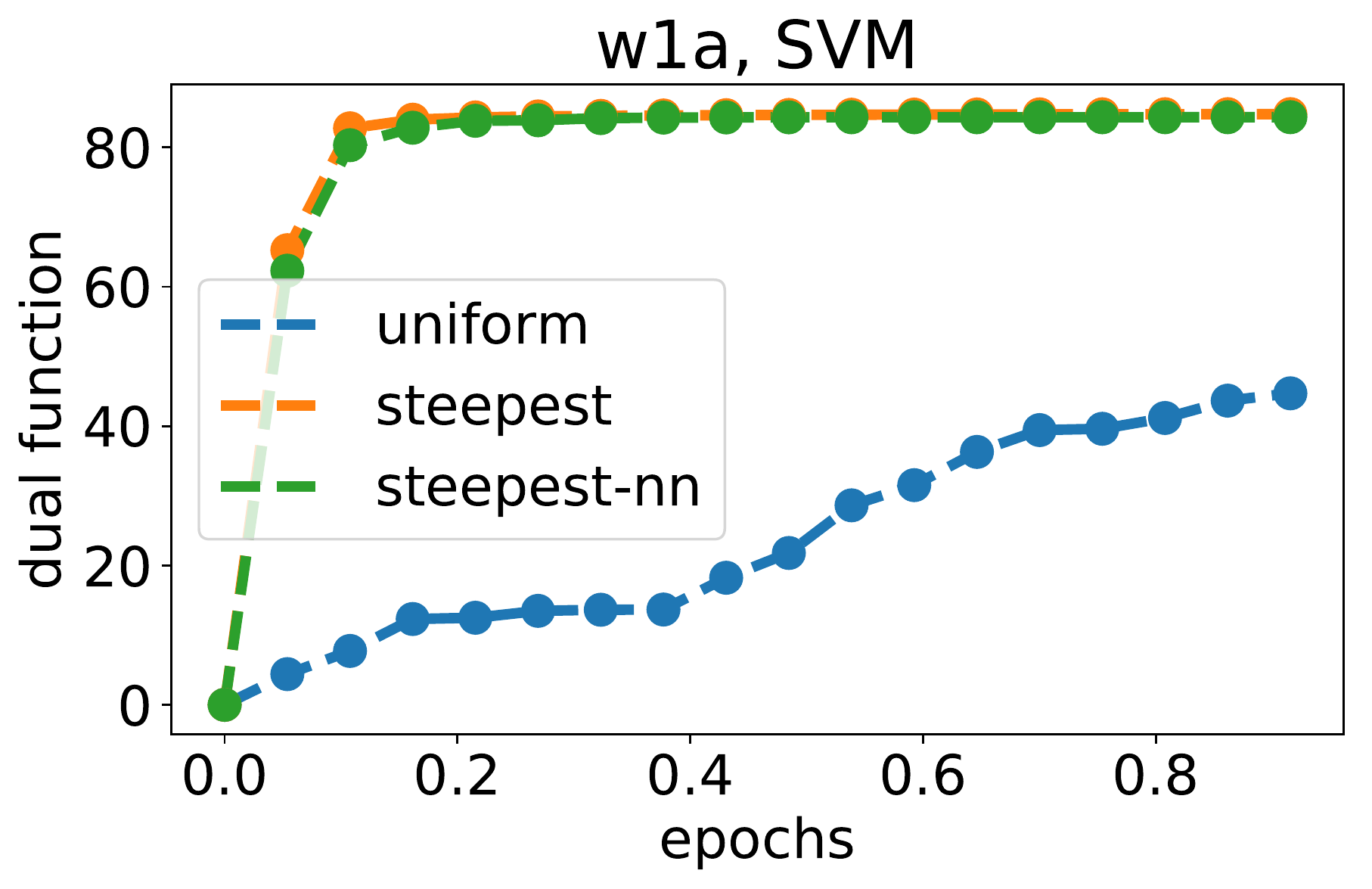}
	\includegraphics[width=0.4\linewidth]{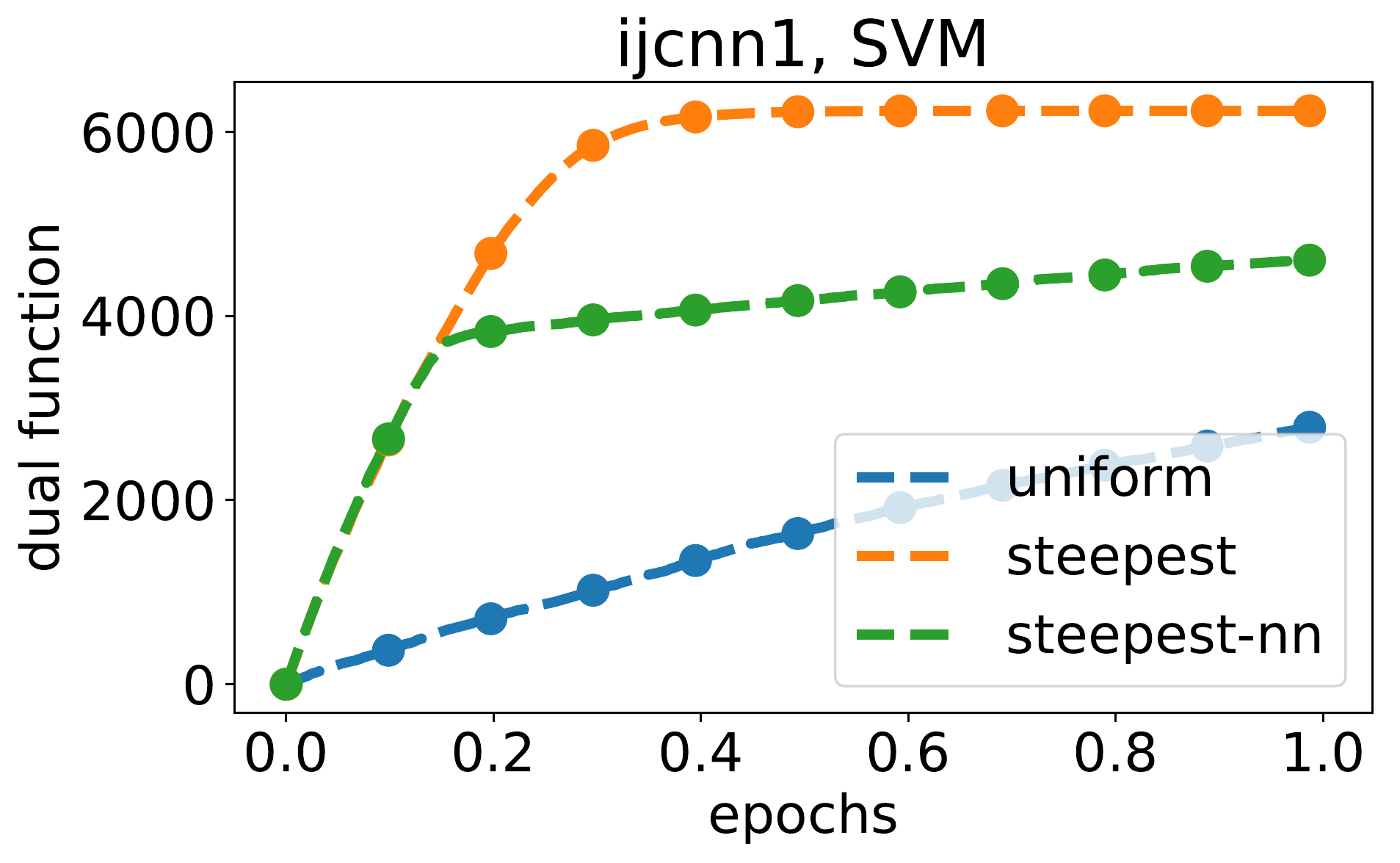}
	\caption{ {\steepest} as well {\nms} significantly outperform {\uniform} in number of iterations.}
	\label{fig:iterations}
\end{figure}

Fig.~\ref{fig:time} compares the wall-time performance of the {\steepest}, {\nms} and {\uniform} strategies. This includes all the overhead of finding the descent direction. In all instances, the {\nms} algorithm is competitive with {\uniform} at the start, compensating for the increased time per iteration by increased progress per iteration. However towards the end {\nms} gets comparable progress per iteration at a significantly larger cost, making its performance worse. 
With increasing sparsity of the solution (see Table \ref{tab:datasets} for sparsity levels), exact {\steepest} rule starts to outperform {\uniform} and {\nms}. 

Wall-time experiments (Fig.~\ref{fig:time}) show that {\nms} always shows a significant performance gain in the important early phase of optimization, but in the later phase loses out to {\uniform} due to the query cost and poor performance of {\nmslib}. In practice, the recommended implementation is to use {\nms} algorithm in the early optimization regime, and switch to uniform once the iteration cost outweighs the gain. In the Appendix (Fig. \ref{fig:adaptivity}) we further investigate the poor quality of the solution provided by {\nmslib}.
\begin{figure}
	\centering
	\includegraphics[width=0.4\linewidth]{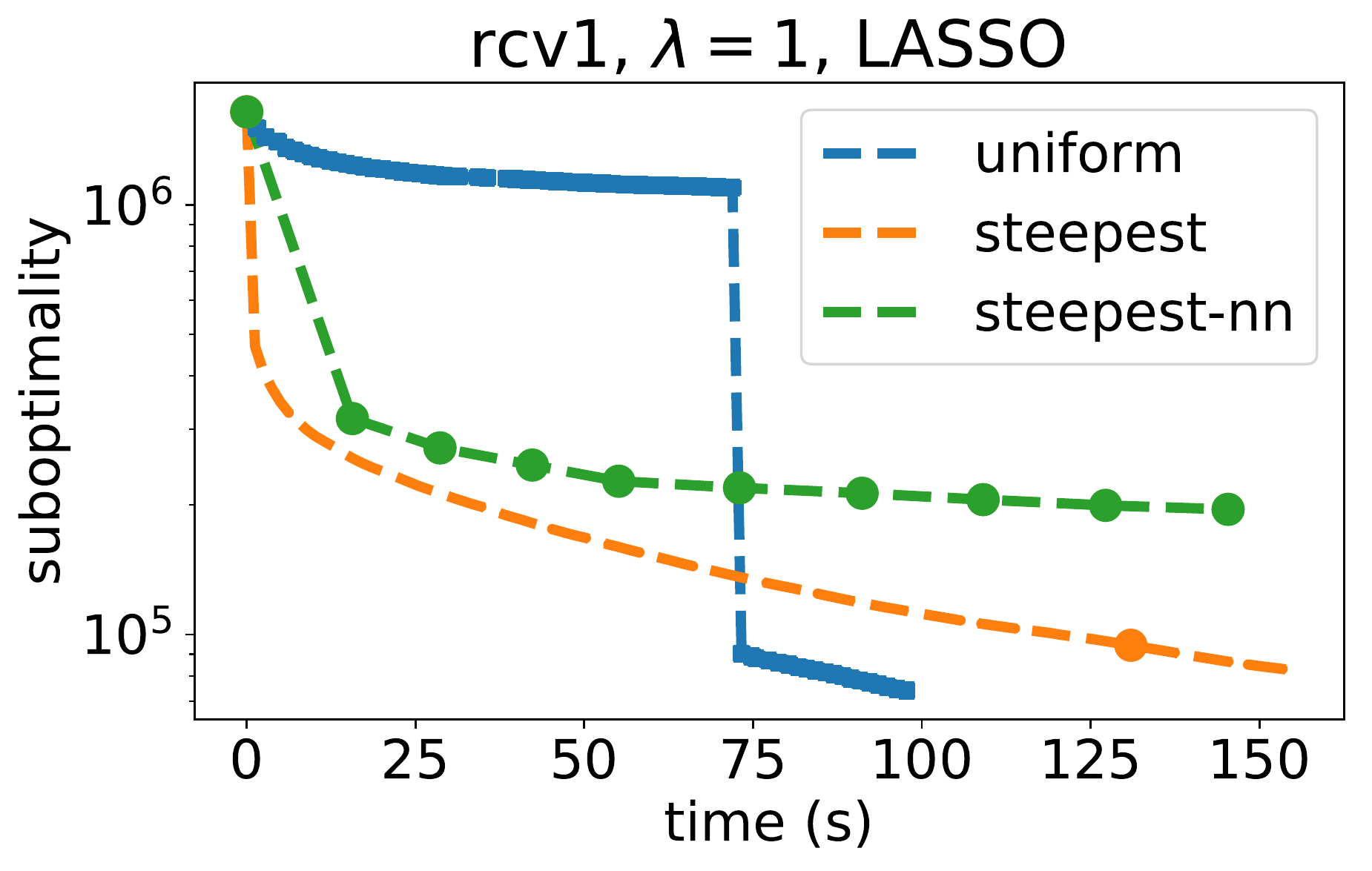}
	\includegraphics[width=0.4\linewidth]{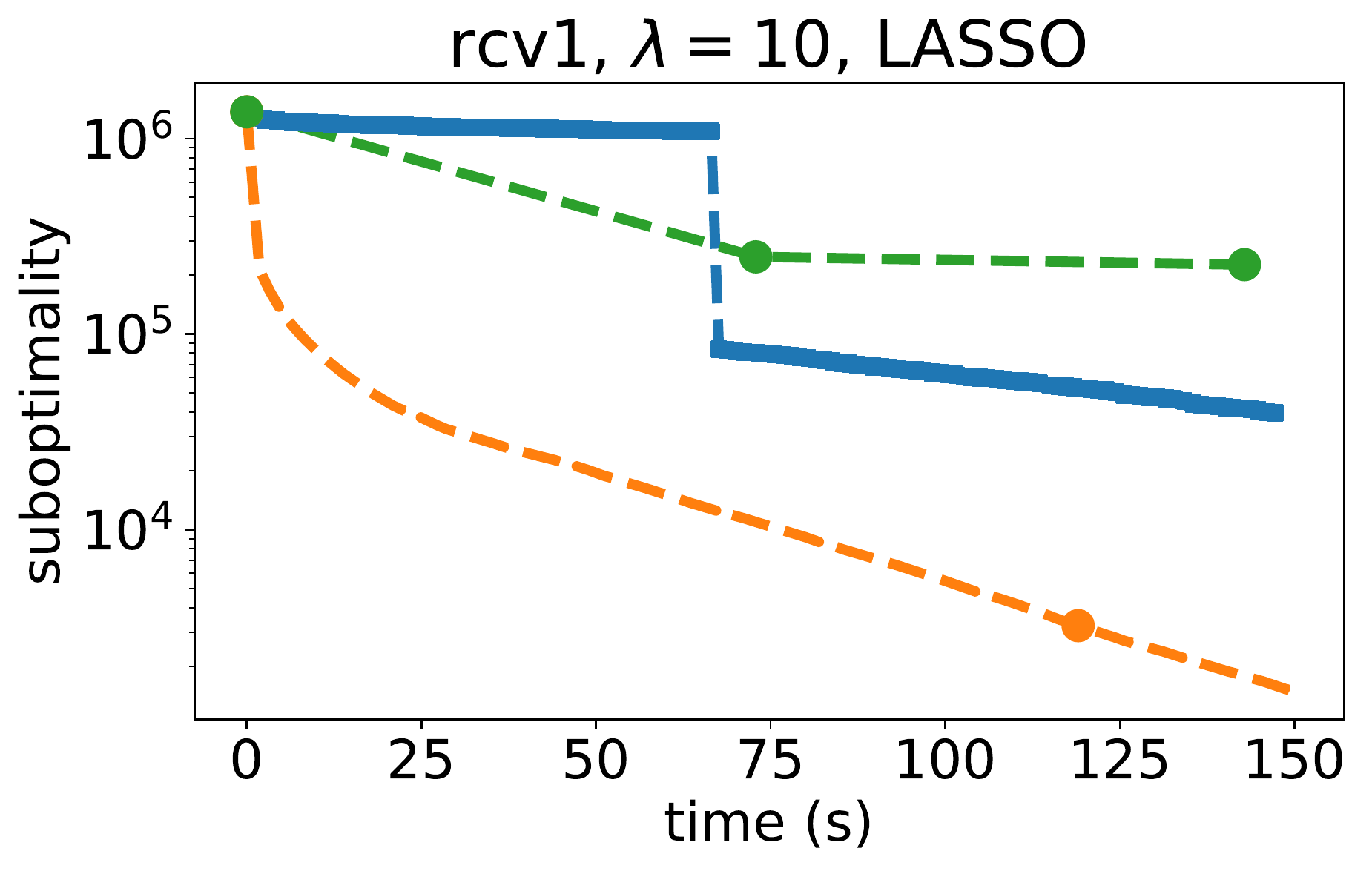}\\
	\includegraphics[width=0.4\linewidth]{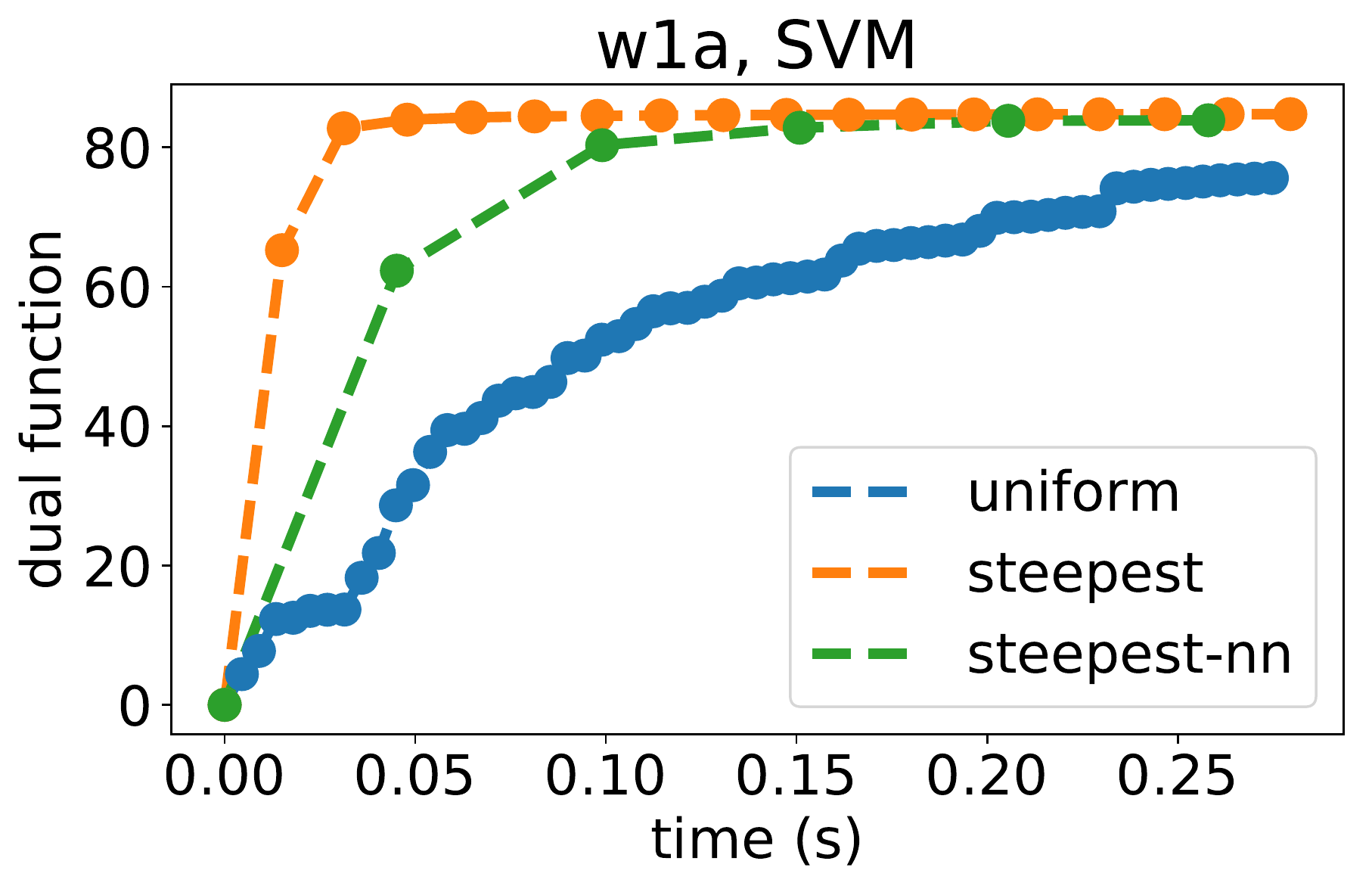}
	\includegraphics[width=0.4\linewidth]{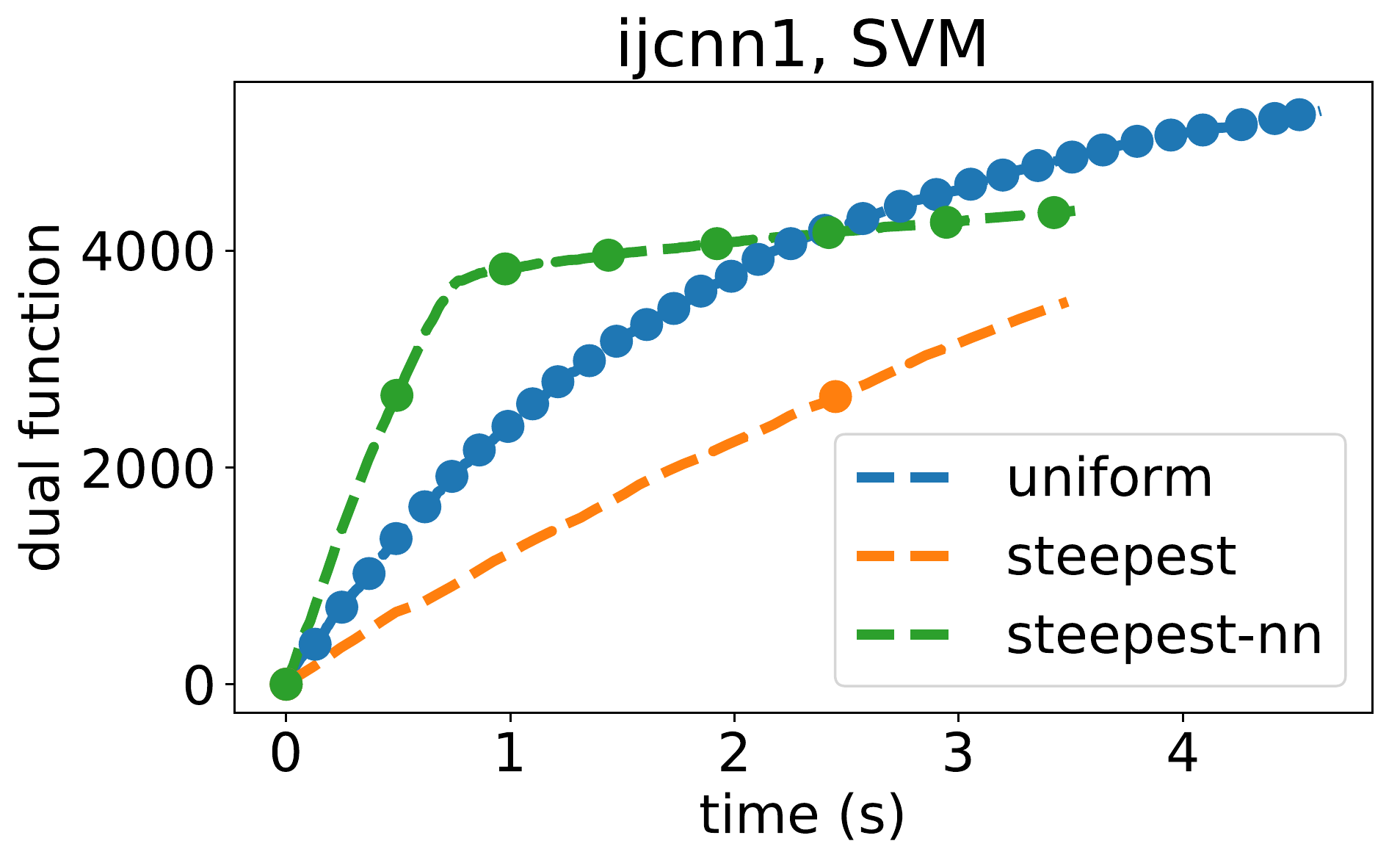}
	\caption{ {\nms} is very competitive and sometimes outperforms {\uniform} even in terms of wall time especially towards the beginning. However eventually the performance of {\uniform} is better than {\nms}. This is because as the norm of the gradient becomes small, the used {\nmslib} algorithm performs poorly.}
	\label{fig:time}
\end{figure}

Repeating our experiments with other datasets, or using {\falconn} \citep{AndoniPracticalOptimalLSH2015}, another popular library for $\MIPS$, yielded comparable results, provided in Appendix~\ref{sec:figures}.

\section{Concluding Remarks}\label{sec:conclusion}
In this work we have proposed a $\Theta$-approximate {\steepsub} selection rule for coordinate descent, and showed its convergence for several important classes of problems for the first time, furthering our understanding of steepest descent on non-smooth problems.
We have also described a new primitive, the Subset Maximum Inner Product Search ($\SMIPS$), and casted the {\steepsub} selection rule as an instance of $\SMIPS$. This enabled the use of fast sublinear algorithms designed for this problem to efficiently compute a $\Theta$-approximate {\steepsub} direction.

We obtained strong empirical evidence for the superiority of the {\steepsub} rule over randomly picking coordinates on several real world datasets, validating our theory. Further, we showed that for Lasso regression, our algorithm consistently outperforms the practical algorithm presented in \citep{DhillonNearestNeighborbased2011}. Finally, we perform extensive numerical experiments showcasing the strengths and weaknesses of a current state-of-the-art libraries for computing a $\Theta$-approximate {\steepsub} direction. As $n$ grows, the cost per iteration for {\nmslib} remains comparable to that of UCD, while the progress made per iteration increases. This means that as problem sizes grow, {\steepsub} implemented via $\SMIPS$ becomes an increasingly attractive approach. Further, we also show that when the norm of the gradient becomes small, current state-of-the-art methods struggle to find directions substantially better than uniform. Alleviating this, and leveraging some of the very active development of recent alternatives to LSH as subroutines for our method is a promising direction for future work. In a different direction, since the {\steepsub} rule, as opposed to {\steeplook} or the {\steeplookr}, uses only local subgradient information, it might be amenable to gradient approximation schemes typically used in zeroth-order algorithms (e.g. \citep{wibisono2012finite}).

\paragraph*{Acknowledgements.} We thank Ludwig Schmidt for numerous discussions on using {\falconn}, and for his useful advice on setting its hyperparameters. We also thank Vyacheslav Alipov for insights about {\nmslib}, Hadi Daneshmand for algorithmic insights, and Mikkel Thorup for discussions on using hashing schemes in practice. The feedback from many anonymous reviewers has also helped significantly improve the presentation.

\bibliography{CD-LSH-bibtex}
\bibliographystyle{apalike}

\newpage
\appendix
\setlength{\belowdisplayskip}{5pt} \setlength{\belowdisplayshortskip}{3pt}
\setlength{\abovedisplayskip}{5pt} \setlength{\abovedisplayshortskip}{3pt}
\part*{Appendix}

\section{Setup and Notation}
\label{sec:basicstuff}
In this section we go over some of the definitions and notations which we had skipped over previously. We will also describe the class of functions we tackle and applications which fit into this framework.

\subsection{Function Classes}
\label{sec:defsmoothandconvex}

\begin{definition}[coordinate-wise $L$-smoothness]\label{def:smoothness}
  A function $f:\R^n\rightarrow \R$ is \emph{coordinate-wise $L$-smooth} if
	$$
    f(\alphav + \gamma \unit_i) \leq f(\alphav) + \gamma\,\nabla_i f(\alphav) + \frac{L \gamma^2}{2},
  $$
  for any $\alphav \in \R^n$, $\gamma \in \R$, $i \in [n]$, and ${\unit}_{i}$ is a coordinate basis vector. 
  \end{definition}

We also define strong convexity of the function $f$.
\begin{definition}[$\mu$-strong convexity]\label{def:strong-convex}
  A function $f:\R^n\rightarrow \R$ is \emph{$\mu$-strongly convex} with respect to some norm $\norm{\cdot}$ if
  $$
    f(\alphav + \deltav) \geq f(\alphav)
    + \scal{\nabla f(\alphav)}{\deltav}
    + \frac{\mu}{2}\norm{\deltav}^2
  $$
  for any $\alphav$ and $\alphav + \deltav$ in the domain of $f$ (note that $f$ does not necessarily need to be defined on the entire space $\R^n$).
\end{definition}
We will frequently denote by $\mu_2$ the strong convexity constant corresponding to the usual Euclidean norm, and by $\mu_1$ the strong convexity constant corresponding the $L1$ norm. In general it holds : $\mu_1 \in [\mu_2/n , \mu_2]$. See \citep{nutini_coordinate_2015} for a detailed comparision of the two constants.

\begin{theorem}\label{lem:L1-coord-smoothnes-relation}
  Suppose that the function is twice-differentiable and 
	  \[
  		L \geq \max_{i \in [n]}[\nabla^2 f(\alphav)]_{i,i}\,,
	  \]
	 for any $\alphav$ in the domain of $f$ i.e. the maximum diagonal element of the Hessian is bounded by $L$.
		Then $f(\alphav)$ is coordinate-wise $L$-smooth. Additionally, for any $\deltav \in \R^n$
	  	\[
  			f(\alphav + \deltav) \leq f(\alphav) + \scal{\nabla f(\alphav)}{\deltav} + \frac{L}{2}\norm{\deltav}_1^2\,.
	  	\]
\end{theorem}
\begin{proof}
  	By Taylor's expansion and the intermediate value theorem, we have that for any $\alphav, \deltav \in \R^n$, there exists a $\alpha \in [0,1]$ such that for $\vv = \alphav + \alpha \deltav$,
  	\begin{equation}\label{eq:taylor-equality}
		  		f(\alphav + \deltav) = f(\alphav) + \scal{\nabla f(\alphav)}{\deltav} + \frac{1}{2}\deltav^\top \nabla^2 f(\vv)\deltav\,.  	
  	\end{equation}
  	Now if $\deltav = \gamma \unit_i$ for some $\gamma \geq 0$ and coordinate $i$, the equation \eqref{eq:taylor-equality} becomes
  	\[
  		f(\alphav + \gamma \unit_i) = f(\alphav) + \gamma\,\nabla_i f(\alphav) + \frac{\gamma^2}{2}[\nabla^2 f(\vv)]_{i,i}\,.
  	\]
  	The first claim now follows since $L$ was defined such that $L \geq \nabla^2 f(\vv)_{i,i}$. For the second claim, consider the following optimization problem over $\widehat\deltav$,
  	\begin{equation}\label{eq:optimization-L1-norm}
  		\max_{\norm{\widehat\deltav}_1 \leq 1}\encasecurly{\widehat\deltav^\top \nabla^2 f(\vv)\widehat\deltav \defeq \Q(\widehat\deltav)} \,.
  	\end{equation}
  	We claim that the maximum is achieved for $\widehat\deltav = \unit_i$ for some $i \in [n]$. For now assume this is true. Then we would have that
  	\begin{align*}
  		\deltav^\top \nabla^2 f(\vv)\deltav &= \norm{\deltav}_1^2 \widehat\deltav^\top \nabla^2 f(\vv)\widehat\deltav\\
  		&\leq \norm{\deltav}_1^2 \unit_i^\top\nabla^2 f(\vv)\unit_i\\
  		&=\norm{\deltav}_1^2 [\nabla^2 f(\vv)]_{i,i}\\
  		&\leq \norm{\deltav}_1^2 L\,.
  	\end{align*}
  	Using this result in equation \eqref{eq:taylor-equality} would prove our second claim of the theorem. Thus we need to study the optimum of \eqref{eq:optimization-L1-norm}. Since $f$ is a convex function, $\Q$ is also a convex function. We can now appeal to Lemma \ref{lem:max-convex} for the convex set $\norm{\xv}_1 \leq 1$. The corners of the set exactly correspond to the unit directional vectors $\unit_i$. With this we finish the proof of our theorem.
\end{proof}
\begin{remark}
	This result states that if we define smoothness with respect to the $L1$ norm, the resulting smoothness constant is same as the coordinate-wise smoothness constant. This is surprising since for a general convex function $g$, using the update rule 
	\[
		\alphav^{(t+1)} = \argmin_{\sv}\scal{\nabla f(\alphav^{(t)})}{\sv - \alphav^{(t)}} + \frac{L}{2}\norm{\sv - \alphav^{(t)}}_1^2 + g(\sv)\,,
	\]
does not necessarily yield a coordinate update. We believe this observation (though not crucial to the current work) was not known before.
\end{remark}
 Let us prove an elementary lemma about maximizing convex functions over convex sets.
 \begin{lemma}[Maximum of a constrained convex function]\label{lem:max-convex} 
 	For any convex function $\Q(\xv)$, the maximum over a compact convex set $B$ is achieved at a `corner'. Here a `corner' is defined to be a point $\xv \in B$ such that there do not exist two points $\yv \in B$ and $\zv \in B$, $\yv \neq \xv, \zv \neq \xv$ such that for some $\gamma \in [0,1]$, $(1 - \gamma)\yv +\gamma\zv = \xv$. 
 \end{lemma}
 \begin{proof}
  	Suppose that the maximum is achieved at a point $\xv$ which is not a `corner'. Then let $\yv,\zv \in B$ be two points such that for $\gamma \in [0,1]$, we have $(1 - \gamma)\yv +\gamma\zv = \xv$. Since the function is convex,
  	\begin{equation}
  		\max(\Q(\yv), \Q(\zv)) \geq (1 - \gamma)\Q(\yv) + \gamma \Q(\zv) \geq \Q((1 - \gamma)\yv +\gamma\zv) = \Q(\xv)\,.\vspace{-2mm}
  	\end{equation}
 \end{proof}
We also assume that the proximal term $g(\alphav) = \sum_{i=1}^n g_i(\alpha_i)$ is such that $g_i(\alpha_i)$ is either $|\alpha_i|$ or enforces a box-constraint.

\subsection{Proximal Coordinate Descent}
\label{sec:proximal}
As argued in the introduction, coordinate descent is the method of choice for large scale problems of the form~\eqref{eq:general-nonsmooth}.
We denote the iterates by $\alphav^{(t)} \in \R^n$, and a single coordinate of this vector by a subscript $\alpha^{(t)}_i$, for $i \in [n]$.
CD methods only change one coordinate of the iterates in each iteration. That is, when coordinate $i$ is updated at iteration $t$, we have $\alpha^{(t+1)}_j = \alpha^{(t)}_j$ for $j \neq i$, and
\begin{equation}\label{eq:coordinate-update}
\alpha^{(t+1)}_i := \text{prox}_{\frac{1}{L} g_{i}} \encase{\alpha^{(t)}_i - \dfrac{1}{L} \nabla_{i} f(\alphav^{(t)})},
\end{equation}
\[
  \text{where }\prox_{\gamma g_{i}}[y] := \argmin_{\xv \in \R }  \tfrac{1}{2} (x - y)^2 + \gamma g_i(x).
 \]

Combining the smoothness condition, and the definition of the proximal update \eqref{eq:coordinate-update}, we get the progress made $\chi_j(\alphav)$ is 
\begin{equation}\label{eq:coordinate-progress}
	\chi_j(\alphav) \defeq \min_{\gamma} \encase{\gamma\nabla_{j} f(\alphav^{(t)})
 		  + \frac{L\gamma^2}{2} + g_{j}(\alpha^{(t)}_{j} + \gamma) - g_{j}(\alpha^{(t)}_{j})} \geq F(\alphav) - \min_{\gamma}F(\alphav + \gamma\unit_i)\,.
\end{equation}

\subsection{Applications}\label{subsec:applications}
There is a number of relevant problems in machine learning which are the form
\[
\min_{\alphav \in \real^n}\encasecurly{F(\alphav) \defeq l(A\alphav) + \cv^\top\alphav + \sum_{i=1}^n g_i(\alpha_i)}\,,
\]
where the non-smooth term $g(\alphav)$ either enforces a box constraint, or is an $L1$-regularizer.
This class covers several important problems such as SVMs, Lasso regression, logistic regression and elastic net regularized problems. We use SVMs and Lasso regression as running examples for our methods.

\paragraph{SVM.} The loss function for training SVMs with $\lambda$ as the  regularization parameter can be written as
\begin{equation}\label{eq:primal-svm}
  \min_{\wv}\encase{\primal(\wv) \defeq
  \frac{1}{n}\sum_{i=1}^n \positive{1 - b_i \wv ^\top \av_i}
  + \frac{\lambda}{2}\norm{\wv}_2^2} \,,
\end{equation}
where $\{(\av_i, b_i)\}_{i=1}^n$ for $\av_i\in\R^d$ and $b_i\in\{\pm1\}$ the
training data. We can define the corresponding \emph{dual} problem for
\eqref{eq:primal-svm} as
\begin{equation}\label{eq:svm-loss}
  \max_{\alphav \in [0,1]^n}\encase{\dual(\alphav) \defeq \dfrac{1}{n} \sum_{i = 1}^n \alpha_{i} - \frac{1}{2\lambda n^2} \norm{A\alphav}_2^2 }\,,
\end{equation}
where $A \in \R^{d \times n}$ for the data matrix of the columns $b_i\av_i$
\citep{shalev-shwartz_stochastic_2013}. We can map this to~\eqref{eq:general-nonsmooth}
with $l(\alphav) := -\dual(\alphav)$, with $\cv:= -\frac1n \boldsymbol{\vec{1}}$, and $g_i(\alpha_i) := \ind{\alpha_i \in [0,1]}$ the box indicator function, i.e.
$$
  F(\alphav) \defeq \frac{1}{2\lambda n^2} \norm{A\alphav}_2^2
  - \dfrac{1}{n} \sum_{i = 1}^n \alpha_{i}
  +\sum_{i=1}^n \ind{\alpha_i \in [0,1]} \,.
$$
It is straight-forward to see that the function $f$ is coordinate-wise $L$-smooth for
$
L = \frac{1}{\lambda n^2} \max_{i\in[n]}\norm{\Av_i}^2
$.

We map the dual variable $\alphav$ back to the primal variable as $\wv(\alphav) = \frac{1}{\lambda n^2} A \alphav$ and the duality gap defined as
$$
  {\gap}(\alphav) \defeq \primal(\wv(\alphav)) - \dual(\alphav)\,.
$$

\paragraph{Logistic regression.} Here we consider the $L1$-regularized
logistic regression loss which is of the form
\begin{equation} \label{eq:log-loss}
  \min_{\alphav}\encase{F(\alphav) \defeq \sum_{i=1}^{d}\sigmoid{b_i \alphav^\top \av_i}
  + \lambda \norm{\alphav}_1} \,,
\end{equation}
where $\{(\av_i, b_i)\}_{i=1}^d$ for $\av_i\in\R^n$ and $b_i\in\{\pm1\}$ is the
training data. The data matrix  $A \in \R^{d \times n}$ is composed with $\av_i$
as the \emph{rows}. Denote $\Av_i$ to be the $i$th \emph{column} of $A$.
As in the Lasso regression case, the regularizer $\lambda \norm{\alphav}_1$ is $g(\alphav)$ and the sigmoid loss corresponds to $l(\alphav)$ which is coordinate-wise
$L$-smooth for
$$
L = \frac{1}{4}\max_{i\in[n]}\norm{\Av_i}^2 \,.
$$

\paragraph{Lasso regression.} The objective function of Lasso regression (i.e. $L1$-regularized least-squares) can directly be mapped to formulation \eqref{eq:general-nonsmooth} as
\begin{equation} \label{ref:lasso-loss}
  \min_{\alphav}\encase{F(\alphav) \defeq \frac{1}{2}\norm{A \alphav - \bv}_2^2
  + \lambda \norm{\alphav}_1} \,.
\end{equation}
Here $l(\alphav) = \frac{1}{2}\norm{A \alphav - \bv}_2^2$,
$g(\alphav) = \lambda \norm{\alphav}_1$ and $f$ is coordinate-wise $L$-smooth for
$
L = \max_{i\in[n]}\norm{\Av_i}^2 \,.
$

\paragraph{Elastic net regression.} The loss function used for Elastic net
can similarly be easily mapped to \eqref{eq:general-nonsmooth} as
\begin{equation} \label{ref:elasticnet-loss}
  \min_{\alphav}\encase{F(\alphav) \defeq \frac{1}{2}\norm{A \alphav - \bv}_2^2
  +\frac{\lambda_2}{2} \norm{\alphav}_2^2 + \lambda_1 \norm{\alphav}_1} \,.
\end{equation}
Here $l(\alphav) = \frac{1}{2}\norm{A \alphav - \bv}_2^2
+\frac{\lambda_2}{2} \norm{\alphav}_2^2 $,
$g(\alphav) = \lambda \norm{\alphav}_1$. The function $f$ is coordinate-wise $L$-smooth for
$$
L = \max_{i\in[n]}\norm{\Av_i}^2 + \lambda_2 \,,
$$
as well as $\lambda_2$-strongly convex. Similarly, $l(\alphav)$ could
also include an $L2$-regularization term for logistic regression.

Note that in the SVM case $n$ represents the number of data points with $d$ features
whereas the roles are reversed in the regression problems.
In all the above cases, $g$ was either a box-constraint (in the SVM) case or
was an $L1$-regularizer.

\section{Algorithms for Box-Constrained Problems}\label{sec:box-method}
In this section, we give an outline of the {\steepsub} strategy as it applies to box constrained problems as well as derive a reduction to the $\MIPS$ framework.
Assume that the minimization problem at hand is of the form
\footnote{If the constraints are of the form
$\alpha_i \in [a_i, b_i]$ for $i \in [n]$,
we simply rescale and shift the coordinates.}
	\begin{equation}\label{eq:smooth_constraint_prob}
	\min_{\alphav \in [0,1]^n} f(\alphav)\,.
	\end{equation}

The proximal coordinate update \eqref{eq:coordinate-update} for updating
coordinate~$i_t$ simplifies to
\begin{equation}\label{eq:box-update}
  \alpha_{i_t} := \min\encaser{1,\positive{\alpha^{(t)}_{i_t} -
  \dfrac{1}{L}\nabla_{i_t}f(\alphav^{(t)})}}.
\end{equation}

At any iteration $t$, let us define an \emph{active set} of coordinates
$\A_t \subseteq [n]$ as follows
\begin{equation}\label{eq:active-set}
  \A_t := \encasecurly{i \in [n]\ \text{ s.t. } \
    \begin{split}
      \alpha_i^{(t)} &\in (0,1), \text{ or}\\
      \alpha_i^{(t)} &= 0 \text{ and } \nabla_i f(\alphav^{(t)}) < 0, \text{ or}\\
      \alpha_i^{(t)} &= 1 \text{ and } \nabla_i f(\alphav^{(t)}) > 0 \,.
    \end{split}}
\end{equation}
Then, the following lemma shows that we can also simplify the rule
\eqref{eq:prox-steepest} for selecting the steepest direction.

\begin{lemma}\label{lem:box-steepest}
  At any iteration $t$, the {\steepsub} rule is equivalent to
  \begin{equation}\label{eq:box-steepest}
    \max_i \encase{\min_{s\in \partial g_i} \abs{\nabla_i
    f(\alphav^{(t)}) + \sv }} \equiv \max_{i \in \A_t}\abs{\nabla_i
    f(\alphav^{(t)})}\,,
  \end{equation}
  assuming that the right side evaluates to 0 if $\A_t$ is empty.
\end{lemma}
While we will see a formal proof later in Section \ref{sec:proofs-mapping}, let us quickly get some intuition for why the above works. When the $i$th coordinate is on the border i.e. $\alpha_i \in \{0,1\}$, there is only one valid direction to move---inside the domain. The set $\A_t$ just maintains the coordinates for which the negative gradient direction leads to a valid non-zero step.

We summarize all the details in Algorithm \ref{alg:box-steepest}. Note that we
have not specified \emph{how} to perform steps \ref{stp:box-steepest} (coordinate selection) and
\ref{stp:box-active} (updating the active set). This we will do next in Section \ref{subsec:efficient-box}.
Our algorithm also supports performing an optional explicit line search to update the coordinate
at step \ref{stp:box-linesearch}. 

\begin{algorithm}[h!]
   \caption{Box Steepest Coordinate Descent (Box-SCD)}
   \label{alg:box-steepest}
\begin{algorithmic}[1]
   \STATE {\bfseries Initialize:} $\alphav_0 \leftarrow {\bf 0}  \in \R^n$,
   and $\A_0 \leftarrow \encasecurly{i \where \nabla_i f(\alphav^{(t)}) < 0}$.
   \FOR{$t=0,1,\dots,\text{until convergence}$}
      \STATE Select coordinate $i_t$ as in {\steepsub}, {\steeplookr}, or {\steeplook}.
        \eqref{eq:box-steepest}. \label{stp:box-steepest} 
      \STATE Find $\alpha_{i_{t}}$ according to \eqref{eq:box-update}.
      \STATE \textit{(Optional line search:)}\\
        $\quad \alpha_{i_{t}} \leftarrow
          \min_{\gamma \in [0,1]}f(\alphav^{(t)} + (\gamma - \alpha^{(t)}_{i_t})\unit_{i_t})$.
          \label{stp:box-linesearch}
      \STATE $\alpha^{(t+1)}_{i_{t}} \leftarrow \alpha_{i_{t}}$.
      \STATE Update $\A_{t+1}$ according to \eqref{eq:active-set}.
      \label{stp:box-active}
   \ENDFOR
\end{algorithmic}
\end{algorithm}

\subsection{Mapping Box Constrained Problems}\label{subsec:efficient-box}
For this part, just as in the $L1$-regularization case in Section \ref{sec:efficient}, we will only look at functions having the special structure of \eqref{eq:problem-structure}, which for $g(\alphav)$ being the box indicator function becomes the constrained problem
\[
	\min_{\alphav \in [0,1]^n}\encasecurly{f(\alphav) \defeq l(A\alphav) + \cv^\top\alphav = \sum_{i=1}^d l_i(\Av_i^\top\alphav) + \cv^\top\alphav}\,.
\]
For this case the selection rule in equation \eqref{eq:box-steepest} in Algorithm \ref{alg:box-steepest} becomes
\[
	\max_{i \in \A_t}\abs{\nabla_i f(\alphav)} = \max_{i \in \A_t}\abs{\Av_i^\top\nabla l(A\alphav) + c_i\alpha_i}
\]
We see that the above is nearly in the form suited for $\MIPS$ oracle already, which as we recall from Def. \ref{def:snns}, is
\begin{equation*}
  \SMIPS_{\P}(\qv;\B) := \argmax_{j \in \B}\encasecurly{\scal{\pv_j}{\qv}}\,.
\end{equation*}
The bulk of our work is now to efficiently maintain the active set $\A_t$ as per \eqref{eq:active-set}. We will make two changes to the formulation above to make it amenable to the $\SMIPS$ solvers: i) get around the extra $c_i\alpha_i$ term, and ii) get around the $\abs{\cdot}$.

First to tackle the extra $c_i\alpha_i$ term, we modify the columns of the matrix $A$ as follows
\begin{equation}\label{eq:A-tilde-box}
  \tilde{\Av}_i := \begin{pmatrix} \beta \\ \Av_i \end{pmatrix}\,.
\end{equation}
Now define the query vector $\qv_t$ as
\begin{equation}\label{eq:query-box-steepest}
  \qv_t := \begin{pmatrix} \frac{c_i}{\beta} \\ \nabla l(A \alphav^{(t)}) \end{pmatrix}\,.
\end{equation}
It is easy to see that 
\[
	\tilde{\Av}_i^\top\qv_t =  \Av_i^\top\nabla l(A\alphav) + c_i\alpha_i\,.
\]
The constant $\beta$ is chosen to ensure that the entry is of the same order of magnitude on average as the rest of the coordinates of $\Av_i$ and can in general be set to $\frac{1}{\sqrt{n}}$. The $\beta$ affects the distribution of the data points, as well as the input query vector. Depending on the technique used to solve the $\SMIPS$ instance, tuning $\beta$ might be useful for the practical performance of the algorithm.

Now to get rid of the absolute value, we use the fact that $\abs{x} = \max(x,-x)$. Define the set $\P = \{\pm \tilde{\Av}_i\}$. Then,
\[
	\max_{i \in [n]}\abs{\tilde{\Av}_i^\top\qv_t} = \max_{\av \in \P}\av^\top\qv_t\,.
\]
To remove the coordinate $j$ from the search space for the equation on the left, it is enough to remove the vector $\argmax_{\av \in \pm \tilde{\Av}_i}\av^\top\qv_t$ from $\P$. In this manner, by controlling the search space $\B_t \subseteq \P$, we can effectively restrict the coordinates over which we search. Formally, let us define
\begin{align}\label{eq:P-box-def}
  \begin{split}
    \P &:= \P^\plus \cup \P^\minus \,\text{, where}\\
    \P^\pm &:= \encasecurly{\pm\tilde{\Av}_1,\dots,\pm\tilde{\Av}_n} \,.
  \end{split}
\end{align}
The active set $\A$ contains only the `valid' directions. Hence, if $\alpha_j =0$ then $j \in \A_t$ only if $\nabla_j f(\alphav_t) = \tilde{\Av}_j^\top\qv_t \leq 0$. This can be accomplished by removing $\tilde{\Av}_j$ and keeping only $-\tilde{\Av}_j$ in $\P$. In general, at any iteration $t$ with current iterate $\alphav^{(t)}$, we define $\B_t := \B_t^\plus \cup \B_t^\minus$ where
\begin{align}\label{eq:box-Bt}
  \begin{split}
    \B_t^\plus &:= \encasecurly{\tilde{\Av}_j: \alpha_j^{(t)} > 0}\,,\\
    \B_t^\minus &:= \encasecurly{-\tilde{\Av}_j: \alpha_j^{(t)} < 1}\,.
  \end{split}
\end{align}
\begin{lemma}\label{lem:box-steepest-efficient}
  At any iteration $t$, for $\P$ and $\B_t$ as defined in \eqref{eq:P-box-def}, \eqref{eq:box-Bt}, the query vector $\qv_t$ as in \eqref{eq:query-box-steepest}, and $\A_t$ as in \eqref{eq:active-set} then the following are equivalent for $f(\alphav)$ is of the form $l(A\alphav) + \cv^\top\alphav$:
  $$
    \SMIPS_{\P}(\qv_t;\B_t) = \argmax_{j \in \A_t}\abs{\nabla_j f(\alphav^{(t)})}\,.
  $$
\end{lemma}
 Finally note that the since the vector $\alphav^{(t+1)}$ and $\alphav^{(t)}$ differ in a single coordinate (say $i_t$), the set $\B_{t+1}$ and $\B_{t}$ differs in at most two points and can be updated in $O(1)$ time.

 \section{Theoretical Analysis for $L1$-regularized Problems}\label{sec:convergence}
 In this section we discuss our main theoretical contributions. We give formal proofs of the rates of convergence of our algorithms and demonstrate the theoretical superiority of the {\steepsub} rule over uniform coordinate descent.
 Recall the definitions of smoothness and strong convexity of $f$.
  \[
    f(\alphav + \gamma \unit_i) \leq f(\alphav) + \gamma\,\nabla_i f(\alphav) + \frac{L \gamma^2}{2},
  \]
  where ${\unit}_{i}$ for $i \in [n]$ is a coordinate basis vector
and that for any $\deltav \in \R^n$,
  \[
   f(\alphav + \deltav) - \encaser{ f(\alphav) + \scal{\nabla f(\alphav)}{\deltav}} \geq \frac{\mu_1}{2}\norm{\deltav}_1^2\,.
  \]
  In our proofs we will only focus on the {\steepsub} rule, but they are true for the {\steeplookr} or {\steeplook} rules too.

 \subsection{Good and Bad Steps}
 	The key to the convergence analysis is dividing the steps into {\good} steps, which make sufficient progress, and {\bad} steps which do not.
 
\begin{figure}\vspace{0em}
  \begin{center}
      \hfill
  	\includegraphics[width=0.45\textwidth]{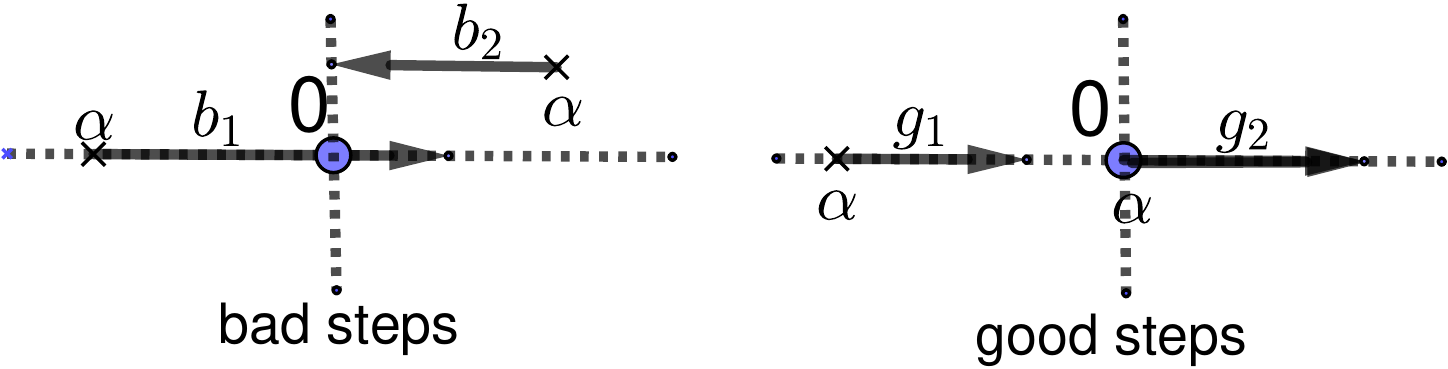}\hfill\null
  	\caption{Characterizing steps for the $L1$ case: The arrows represent proximal coordinate updates $\shrinkage_\frac{\lambda}{L}(\alpha_i -\frac{1}{L}\nabla_i f(\alphav))$ from different starting points $\alphav$. Updates which `cross' ($b_1$) or `end at' ($b_2$) the origin are {\bad} whereas the rest ($g_1$, $g_2$) are {\good}.}\vspace{-1em}
	\label{fig:good-bad-L1}
  	\end{center}
\end{figure}

 	\begin{definition}[{\good} and {\bad} steps for $L1$ case]\label{def:good-steps-L1}
 		At any step $t$, let $i_t$ be the coordinate to be updated. Then if either i) $\alpha_i \cdot \encaser{\shrinkage_\frac{\lambda}{L}(\alpha_i -\frac{1}{L}\nabla_i f(\alphav))} > 0$, or ii) $\alpha_i = 0$, then the step is deemed a {\good} step (see Figure \ref{fig:good-bad-L1}). The steps where $\alpha_i \cdot \encaser{\shrinkage_\frac{\lambda}{L}(\alpha_i -\frac{1}{L}\nabla_i f(\alphav))} \leq 0$ are called {\bad} steps.
 	\end{definition}

 	In general there is no guarantee that all steps are {\good}. However our algorithms exploit the structure of the proximal function $g$ to ensure that at least a constant fraction of total steps are {\good}. Recall that the algorithm \ref{alg:L1} performs the following update step to update the $i$th coordinate  as in \eqref{eq:l1-intermediate-update} and \eqref{eq:l1-update}:
 	\begin{align*}
	    \alpha_{i}^{+} &= \shrinkage_{\frac{\lambda}{L}}\encaser{\alpha_{i}^{(t)} -
    	\frac{1}{L}\nabla_{i}f(\alpha_{i}^{(t)})}\,\text{ and}\\
		\alpha^{(t+1)}_i &= 
			\begin{cases}
				\alpha^{+}_i, & \text{if } \alpha^{+}_i \alpha^{(t)}_i \geq 0\\
				0, & \text{otherwise}
			\end{cases}.
	\end{align*}
	From here on we will stop indicating the iteration number $t$ when obvious.
 	\begin{lemma}[Counting {\good} steps]\label{lem:num-good-L1}
 		After $t$ iterations of Algorithm \ref{alg:L1}, out of the $t$ steps, at least $\ceil{t/2}$ steps are {\good}.
 	\end{lemma}
 	\begin{proof}
 		The definition \ref{def:good-steps-L1} says that the only bad steps are those for which i) $\alpha_i^{(t)} \neq 0$ and ii) $\alpha_i\alpha_i^{+} \leq 0$. However by our modification to the update rule \eqref{eq:l1-update} ensures that in this case, $\alpha_i^{(t+1)} = 0$. This ensures that the next time coordinate $i$ is picked, we are guaranteed a {\good} step. Since we start at the origin, we have our lemma.
 	\end{proof}
 	\subsection{Progress Made in {\good} and {\bad} Steps}\label{subsec:characterizing-progress-L1}
 	Below is the core technical lemma in the convergence proof. We show that if the step was {\good}, then the update chosen by the {\steepsub} rule actually corresponds to optimizing an upper-bound on the function with the usual $L2$-squared regularizer replaced by an $L1$-squared regularizer. We of course also have only an approximate {\steepsub} coordinate. 	
 	 
 	 Recall that we had defined
 	 \[
 		\chi_j(\alphav) \defeq \min_{\gamma} \encase{\gamma\nabla_{j} f(\alphav)
 		  + \frac{L\gamma^2}{2} + \lambda(\norm{\alpha_{j} + \gamma}_1 - \norm{\alpha_{j}}_1)}\,.
 	\]
\begin{lemma}[{\good} steps make a lot of progress]\label{lem:good-step-progress-L1}
 	Suppose that the $\Theta$-approximate {\steepsub} rule (recall Definition \ref{def:approx-steepest}) chose to update the $i$th coordinate, and further suppose that it was a {\good} step. Then the following holds:
 	\[
 		\chi_i(\alphav) \leq \Theta^2 \min_{\deltav \in \R^n}\encase{\scal{\nabla f(\alphav)}{\deltav} + \frac{L}{2}\norm{\deltav}_1^2 + \lambda(\abs{\alphav+\deltav} - \abs{\alphav})}\,.
 	\]
\end{lemma}
\begin{proof}
 	Recall from Section \ref{sec:method} that the {\steepsub} rule for the $L1$ case was to find a coordinate $i$ such that
 	\[
 		\abs{s(\alphav)_j} \geq \Theta \argmax_{j}\abs{s(\alphav)_j}\,.
 	\]
 	The vector $s(\alphav)$ (by expanding the definition of the shrinkage operator) can equivalently be defined as
 	\[
 		s(\alphav)_j :=
    \begin{cases}
      0, & \text{if }	\alpha_j = 0, \abs{\nabla_j f(\alphav)} \leq \lambda\\
	  \nabla_j f(\alphav) - \lambda, & \text{if } \alpha_j = 0, \nabla_j f(\alphav) > \lambda\\
      \nabla_j f(\alphav) + \lambda, & \text{if } \alpha_j = 0, \nabla_j f(\alphav) < - \lambda\\
      \nabla_j f(\alphav)  + \sign(\alpha_j)\lambda  & \text{otherwise}\,.
    \end{cases}
 	\]
 	Let us define the folowing vector $\zetav \in [-1,1]^n$ as
 	\[
 		\zeta_j := (s(\alphav)_j - \nabla_j f(\alphav))/\lambda \,. 
 	\]
	Examining the four cases listed above shows indeed that $\zeta_j \in [-1,1]$.
	We will use the fact that for any $\beta \in [-1,1]$, $\abs{x} \geq \beta x$. We have that
	\[
		\abs{\alphav+\deltav} \geq \scal{\zetav}{\alphav+\deltav} = \scal{\zetav}{\alphav} + \scal{\zetav}{\deltav}\,.
	\]
 	This implies that for any vector $\deltav \in \R^n$,
 	\begin{align*}
 		\encasecurly{\P_{\abs{}}(\alphav, \deltav)\defeq \scal{\nabla f(\alphav)}{\deltav} + \frac{L}{2}\norm{\deltav}_1^2 + \lambda(\norm{\alphav+\deltav}_1 - \norm{\alphav}_1)}&\\
 		&\hspace*{-1.7in}\geq \scal{\nabla f(\alphav)}{\deltav} + \frac{L}{2}\norm{\deltav}_1^2 + \lambda(\scal{\zetav}{\alphav + \deltav} - \norm{\alphav}_1)\\
 		&\hspace*{-1.7in}= \scal{\nabla f(\alphav) + \lambda \zetav}{\deltav} + \frac{L}{2}\norm{\deltav}_1^2 + \lambda(\scal{\zetav}{\alphav} - \norm{\alphav}_1)\\
 		&\hspace*{-1.7in}= \scal{s(\alphav)}{\deltav} + \frac{L}{2}\norm{\deltav}_1^2 + \lambda(\scal{\zetav}{\alphav} - \norm{\alphav}_1)\,.
 	\end{align*}
 	In the last step we used the definition of $\zetav$.
 	 Let us examine the expression $(\zeta_j \alpha_j - \abs{\alpha_j})$. If $\alpha_j =0$, each of the terms in the expression is 0 and so it evaluates to 0. If $\alpha_j \neq 0$, then by the definition, $\zeta_j = \sign(\alpha_j)$ and $\zeta_j \alpha_j = \abs{\alpha_j}$. In this case too, the expression is 0. Thus,
 	 \begin{align*}
 	 	\P_{\abs{}}(\alphav, \deltav) &\geq \scal{s(\alphav)}{\deltav} + \frac{L}{2}\norm{\deltav}_1^2 + \lambda(\scal{\zetav}{\alphav} - \norm{\alphav}_1)\\
 	 	&= \encasecurly{ \scal{s(\alphav)}{\deltav} + \frac{L}{2}\norm{\deltav}_1^2 \defeq \P_{\zetav}(\alphav, \deltav)}\,.
 	 \end{align*}
 	 Minimizing $ \P_{\zetav}(\alphav, \deltav)$ over $\deltav$ gives us that
 	\[
 		\min_{\deltav \in \R^n} \scal{s(\alphav)}{\deltav} + \frac{L}{2}\norm{\deltav}_1^2 = -\frac{1}{2L}\max_{j \in [n]}\encaser{s(\alphav)_j}^2 \,.
 	\]
 	This exactly corresponds to the {\steepsub} rule. Since $i$ was a $\Theta$-approximate {\steepsub} direction,
	\[
		\min_{\deltav \in \R^n} \scal{s(\alphav)}{\deltav} + \frac{L}{2}\norm{\deltav}_1^2 \geq -\frac{1}{2L\Theta^2}\encaser{s(\alphav)_i}^2
	\] 	
 	  Further since this is a {\good} step, by Lemma \ref{lem:good-step-character-L1}, we can replace the right side of the above equation with $\chi_i(\alphav)$. Putting these observations together we have that for any $\alphav \in \R^n$
 	\begin{align*}
 	\min_{\deltav \in \R^n}\P_{\abs{}}(\alphav, \deltav) &\geq \min_{\deltav \in \R^n}\P_{\zetav}(\alphav, \deltav)\\
 	&\geq \frac{1}{\Theta^2}\chi_i(\alphav)\,.
	\end{align*}
	Rearranging the terms gives the proof of the lemma. 
\end{proof}

\begin{lemma}[Characterizing a {\good} step]\label{lem:good-step-character-L1}
 		Suppose that we update the $i$th coordinate and that it was a {\good} step. Then 
 		\[
 			\chi_i(\alphav) = -\frac{1}{2L}(\alpha_i^+ - \alpha_i)^2 = -\frac{1}{2L}(s(\alphav)_i)^2\,.
 		\]
\end{lemma}
\begin{proof}
By the definition of the coordinate proximal update,
\begin{align*}
	\chi_i(\alphav) &= \min_{\gamma \in \R}\encase{\gamma\nabla_{i} f(\alphav)
 		  + \frac{L\gamma^2}{2} + \lambda(\abs{\alpha_{i} + \gamma} - \abs{\alpha_{i}})}\\
 		  &= (\alpha_i^{+} - \alpha_i)\nabla_{j} f(\alphav)
 		  + \frac{L}{2}(\alpha_i^{+} - \alpha_i)^2 + \lambda(\abs{\alpha_{i}^{+}} - \abs{\alpha_{i}})\,.
\end{align*}
Now first suppose that $\alpha_i \neq 0$. Since this is a good step, $\sign(\alpha_i^+) = \sign(\alpha_i)$. Without loss of generality, let us assume that $\alpha_i > 0$. Then $(\alpha_i^+ = \alpha_i - (\nabla_i f(\alphav) + \lambda)/L)$. Using this in the above expression,
\begin{align*}
	\chi_i(\alphav) &= (\alpha_i^{+} - \alpha_i)\nabla_{i} f(\alphav)
 		  + \frac{L}{2}(\alpha_i^{+} - \alpha_i)^2 + \lambda(\abs{\alpha_{i}^{+}} - \abs{\alpha_{i}})\\
 		  &= - \frac{1}{L}(\nabla_i f(\alphav) + \lambda)(\nabla_i f(\alphav) + \frac{L}{2 L^2}\encaser{\nabla_i f(\alphav) + \lambda}^2 - \frac{\lambda}{L}(\nabla_i f(\alphav) + \lambda)\\
 		  &=- \frac{1}{L}(\nabla_i f(\alphav) + \lambda)(\nabla_i f(\alphav) + \lambda) + \frac{1}{2L}\encaser{\nabla_i f(\alphav) + \lambda}^2 \\
 		  &= -\frac{1}{2L}(\nabla_i f(\alphav) + \lambda)^2\\
 		  &= -\frac{1}{2L}(s(\alphav)_i)^2\,.
\end{align*}
The proof for when $\alpha_i < 0$ is identical. Now let us see the case when $\alpha_i = 0$. Since this is not a {\bad} step, we have that $\alpha_i^+ \neq 0$ meaning that $\abs{\nabla_i f(\alphav)} > \lambda$. Without loss of generality, assume that $\alpha_i^+ > 0$---the other case is identical. Then $(\alpha_i^+ = -\frac{1}{L}\encaser{\nabla_i f(\alphav) + \lambda})$ and $\encaser{s(\alphav)_i = \nabla_i f(\alphav) + \lambda}$. Doing the same computations as before,
\begin{align*}
	\chi_i(\alphav) &= (\alpha_i^{+} - \alpha_i)\nabla_{i} f(\alphav)
 		  + \frac{L}{2}(\alpha_i^{+} - \alpha_i)^2 + \lambda(\abs{\alpha_{i}^{+}} - \abs{\alpha_{i}})\\
 		  &= - \frac{1}{L}(\nabla_i f(\alphav) + \lambda)(\nabla_i f(\alphav) + \frac{L}{2 L^2}\encaser{\nabla_i f(\alphav) + \lambda}^2 - \frac{\lambda}{L}(\nabla_i f(\alphav) + \lambda)\\
 		  &=- \frac{1}{L}(\nabla_i f(\alphav) + \lambda)(\nabla_i f(\alphav) + \lambda) + \frac{1}{2L}\encaser{\nabla_i f(\alphav) + \lambda}^2 \\
 		  &= -\frac{1}{2L}(\nabla_i f(\alphav) + \lambda)^2\\
 		  &= -\frac{1}{2L}(\alpha_i^+ - \alpha_i)^2\\
 		  &= -\frac{1}{2L}(s(\alphav)_i)^2\,.
\end{align*}
Such simple calculations shows that indeed the lemma holds in all cases.
\end{proof}
\begin{remark}
	Lemma \ref{lem:good-step-character-L1} shows that for a {\good} step, the {\steepsub}, {\steeplookr}, and {\steeplook} rules coincide. Thus even though we explicitly write the analysis for the {\steepsub} rule, it also holds for the other updates. Note that while the three rules coincide for a {\good} step, they can be quite different during the {\bad} steps. Further the computations required for the three rules is not the same since we a priori do not know if a step is going to be {\good} or {\bad}.
\end{remark}
We have characterized the update made in a {\good} step and now let us look at the {\bad} steps.
 \begin{lemma}[{\bad} steps are not too bad]\label{lem:bad-progress-L1}
 		In any step of Algorithm \ref{alg:L1} including the {\bad} steps, the objective value never increases.
 	\end{lemma}
 	\begin{proof}
 		If the step was {\good}, since the update \eqref{eq:l1-intermediate-update} is just a proximal coordinate update, it is guaranteed to not increase the function value.
 		The modification we make when the step is {\bad} \eqref{eq:l1-update} makes the lemma slightly less obvious. Without loss of generality, by symmetry of $L1$ about the origin, let us assume that $\alpha_i > 0$. Since the step was {\bad}, this implies that $\alpha_i^{+} = \alpha_i - \frac{1}{L}(\nabla_i f(\alphav) + \lambda) < 0$ for \eqref{eq:l1-update}. Then 
 		\begin{align*}
 			F(\alphav^{t+1}) - F(\alphav^{(t)}) &\leq -\nabla_i f(\alphav) \alpha_i + \frac{L}{2}\alpha_i^2 - \lambda\abs{\alpha_i}\\
 				&= \alpha_i\encaser{-\nabla_i f(\alphav) + \frac{L}{2}\alpha_i - \lambda}\\
 				&\leq L{\alpha_i}\encaser{-\frac{1}{L}\nabla_i f(\alphav) + \alpha_i/2 - \frac{1}{L}\lambda}\\
 				&= L{\alpha_i}\encaser{\encaser{\alpha_i - \frac{1}{L}[\nabla_i f(\alphav)+\lambda]} - \alpha_i/2}\\
 				&\leq  L{\alpha_i}\encaser{0 + 0}\,. \qedhere
 		\end{align*}
 	\end{proof}

\subsection{Convergence in the Strongly Convex Case}\label{subsec:strong-convergence-rate-L1}
\begin{theorem}\label{thm:strong-convex-L1-appendix}
	After $t$ steps of Algorithm \ref{alg:L1} where in each step the coordinate was selected using the $\Theta$-approximate {\steepsub} rule, let $\G_t \subseteq [t]$ indicate the {\good} steps. Assume that the function was $L$-coordinate-wise smooth and $\mu_1$ strongly convex with respect to the $L1$ norm. The size $\abs{\G_t} \geq \ceil{t/2}$ and
	\[
		F(\alphav^{(t+1)}) - F(\alphav^\star) \leq \encaser{1 - \frac{\Theta^2 \mu_1}{L}}^{\abs{\G_t}}\encaser{F(\alphav^{(0)}) - F(\alphav^{\star})}\,.
	\]
\end{theorem} 
 
 We have almost everything we need in place to prove our convergence rates for the strongly convex, and the general convex case.
Recall that the function $f$ is $\mu_1$ strongly convex with respect to the $L1$ norm. This implies that
\[
	F(\alphav^\star) \geq f(\alphav) + \scal{\nabla f(\alphav)}{\alphav^\star - \alphav} + \frac{\mu_1}{2}\norm{\alphav^\star - \alphav}_1^2 + g(\alphav^\star)\,.
\]	
 We will need one additional Lemma from \citep{karimireddy2018adaptive} to relate the upper bound we minimize in Lemma \ref{lem:good-step-progress-L1}.
 \begin{lemma}[Relating different regularizing constants \citep{karimireddy2018adaptive}]\label{lem:change-mu-L1}
 For any vectors $\gv, \alphav \in \R^n$, and constants $L \geq \mu > 0$,
 	\begin{multline*}
 		\min_{\wv \in \R^n}\encasecurly{\lambda(\norm{\wv}_1 - \norm{\alphav}_1) + \scal{\gv}{\wv - \alphav} + \frac{L}{2}\norm{\wv - \alphav}_1^2} \leq \\ \frac{\mu}{L} \min_{\wv \in R^n}\encasecurly{ \lambda(\norm{\wv}_1 - \norm{\alphav}_1) + \scal{\gv}{\wv - \alphav} + \frac{\mu}{2}\norm{\wv - \alphav}_1^2}\,.
 	\end{multline*}
 \end{lemma}	
 	\begin{proof}
 		The function $\lambda(\abs{\wv} - \abs{\alphav}) + \scal{\gv}{\wv - \alphav}$ is convex and is 0 when $\wv = \alphav$, and $\norm{\wv - \alphav}_1$ is a convex positive function. Thus we can apply Lemma 9 from \citep{karimireddy2018adaptive}.
 	\end{proof}
 
 The proof of the theorem now easily follows. 

\paragraph*{Proof of Theorem \ref{thm:strong-convex-L1-appendix}.}
First, by Lemma \ref{lem:num-good-L1}, we know that there are at least as many {\good} steps as there are {\bad} steps. This means that $\abs{\G_t} \geq \ceil{t/2}$ for any $t$.
 Now suppose that $t$ was a {\good} step and updated the $i$th coordinate. Now by the progress made in a proximal update \eqref{eq:coordinate-progress} and Lemma \ref{lem:good-step-progress-L1}, 
 \begin{align*}
  F(\alphav^{(t+1)}) &\leq F(\alphav^{(t)}) +  \chi_i(\alphav^{(t)})\\
  &= F(\alphav^{(t)}) + \Theta^2\min_{\wv \in \R^n}\encasecurly{\scal{\nabla f(\alphav^{(t)})}{\wv - \alphav^{(t)}} + \frac{L}{2}\norm{\wv - \alphav^{(t)}}_1^2 + \lambda(\norm{\wv}_1 - \norm{\alphav^{(t)}}_1)}\\
  &\leq F(\alphav^{(t)}) + \frac{\Theta^2\mu_1}{L}\min_{\wv \in \R^n}\encasecurly{\scal{\nabla f(\alphav^{(t)})}{\wv - \alphav^{(t)}} + \frac{\mu_1}{2}\norm{\wv - \alphav^{(t)}}_1^2 + \lambda(\norm{\wv}_1 - \norm{\alphav^{(t)}}_1)}
\end{align*} 	
In the last step we used Lemma \ref{lem:change-mu-L1}. We will now use our definition of strong convexity. We have shown that
\begin{align*}
	F(\alphav^{(t+1)}) - F(\alphav^\star) &\leq F(\alphav^{(t)}) - F(\alphav^\star)\\&\hspace*{0.2in} + \frac{\Theta^2\mu_1}{L}\min_{\wv \in \R^n}\encasecurly{\scal{\nabla f(\alphav^{(t)})}{\wv - \alphav^{(t)}} + \frac{\Theta^2\mu_1}{2}\norm{\wv - \alphav^{(t)}}_1^2 + \lambda(\norm{\wv}_1 - \norm{\alphav^{(t)}}_1)}\\
	&\leq  F(\alphav^{(t)}) - F(\alphav^\star) +\\&\hspace*{0.2in} \frac{\Theta^2\mu_1}{L}\encaser{\scal{\nabla f(\alphav^{(t)})}{\alphav^\star - \alphav^{(t)}} + \frac{\Theta^2\mu_1}{2}\norm{\alphav^\star - \alphav^{(t)}}_1^2 + \lambda(\norm{\alphav^{\star}}_1 - \norm{\alphav^{(t)}}_1)} \\
	&\leq F(\alphav^{(t)}) - F(\alphav^\star) + \frac{\Theta^2\mu_1}{L}\encaser{f(\alphav^\star) - f(\alphav^{(t)}) + \lambda(\norm{\alphav^{\star}}_1 - \norm{\alphav^{(t)}}_1)}\\
	&= (1 - \frac{\Theta^2\mu_1}{L})\encaser{F(\alphav^{(t)}) - F(\alphav^\star)}\,.
\end{align*}
We have shown that we make significant progress every {\good} step. Combining this with Lemma \ref{lem:bad-progress-L1} which shows that the function value does not increase in a {\bad} step finishes the proof.
\qed
 	
 \begin{remark}
 	We show that the {\steepsub} rule has convergence rates indepedent of $n$ for $L1$-regularized problems. As long as the update is kept the same as in Algorithm \ref{alg:L1} (with the modification), our proof also works for the {\steeplookr} and the {\steeplook} rules. This answers the conjecture posed by \citep{nutini_coordinate_2015} in the affirmative, at least for $L1$-regularized problems.
 \end{remark}	
 	
\subsection{Convergence in the General Convex Case}\label{subsec:general-convergence-rate-L1}
\begin{theorem}\label{thm:general-convex-L1-appendix}
		After $t$ steps of Algorithm \ref{alg:L1} where in each step the coordinate was selected using the $\Theta$-approximate {\steepsub} rule, let $\G_t \subseteq [t]$ indicate the {\good} steps. Assume that the function was $L$-coordinate-wise smooth. Then the size $\abs{\G_t} \geq \ceil{t/2}$ and
	\[
		F(\alphav^{(t+1)}) - F(\alphav^\star) \leq \frac{L D^2}{2\Theta^2 \abs{\G_t}}\,,
	\]
	where $D$ is the $L1$-diameter of the level set. For the set of minima $\Q^\star$,
	\[
		D = \max_{\wv \in \R^n}\min_{\alphav^\star \in \Q^\star}\encasecurly{\norm{\wv - \alphav^\star}_1 \Big| F(\wv) \leq F(\alphav^{(0)})}\,.
	\]
\end{theorem}
\begin{proof}
	We start exactly as in the strongly convex case. First, by Lemma \ref{lem:num-good-L1}, we know that there are at least as many {\good} steps as there are {\bad} steps. This means that $\abs{\G_t} \geq \ceil{t/2}$ for any $t$.
 Now suppose that $t$ was a {\good} step and updated the $i$th coordinate. Now by the progress made in a proximal update \eqref{eq:coordinate-progress} and Lemma \ref{lem:good-step-progress-L1},
 \begin{align*}
  F(\alphav^{(t+1)}) &\leq F(\alphav^{(t)}) +  \chi_i(\alphav^{(t)})\\
  &\leq f(\alphav^{(t)}) + \Theta^2\min_{\wv \in \R^n}\encasecurly{\scal{\nabla f(\alphav^{(t)})}{\wv - \alphav^{(t)}} + \frac{L}{2}\norm{\wv - \alphav^{(t)}}_1^2 + \lambda \norm{\wv}_1}\\
  &\leq f(\alphav^{(t)}) + \Theta^2\min_{\wv = (1-\gamma)\alphav^{(t)} + \gamma \alphav^\star} \encasecurly{\scal{\nabla f(\alphav^{(t)})}{\wv - \alphav^{(t)}} + \frac{L}{2}\norm{\wv - \alphav^{(t)}}_1^2 + \lambda \norm{\wv}_1}\\
  &= f(\alphav^{(t)}) + \Theta^2\min_{\gamma \in \R} \encasecurly{\gamma\scal{\nabla f(\alphav^{(t)})}{\alphav^\star - \alphav^{(t)}} + \frac{L\gamma^2}{2}\norm{\alphav^\star - \alphav^{(t)}}_1^2 + \lambda \norm{(1-\gamma)\alphav^{(t)} + \gamma \alphav^\star}_1}\\
  &\leq f(\alphav^{(t)}) + \Theta^2\min_{\gamma \in \R} \encasecurly{\gamma (f(\alphav^\star) - f(\alphav^{(t)})) + \frac{L\gamma^2}{2}\norm{\alphav^\star - \alphav^{(t)}}_1^2 + \lambda (1-\gamma)\norm{\alphav^{(t)}}_1 + \lambda\gamma\norm{\alphav^{\star}}_1}\,. 
  \end{align*}
  In the last step we used the convexity of $f$ and of $\abs{\cdot}$. Now denote the suboptimality $h_t = F(\alphav^{(t)}) - F(\alphav^\star)$. We have shown that
  \begin{align*}
  	h_{t+1} &= F(\alphav^{(t+1)}) - F(\alphav^\star)\\
  	&\leq f(\alphav^{(t)}) - F(\alphav^\star) \\&\hspace*{0.4in}+ \Theta^2\min_{\gamma \in \R} \encasecurly{\gamma (f(\alphav^\star) - f(\alphav^{(t)})) + \frac{L\gamma^2}{2}\norm{\alphav^\star - \alphav^{(t)}}_1^2 + \lambda (1-\gamma)\norm{\alphav^{(t)}}_1 + \lambda\gamma\norm{\alphav^{\star}}_1}\\
  	&= f(\alphav^{(t)}) + \lambda \norm{\alphav^{(t)}}_1 - F(\alphav^\star) \\&\hspace*{0.4in}+ \Theta^2\min_{\gamma \in \R} \encasecurly{\gamma (f(\alphav^\star) + \lambda \norm{\alphav^{\star}}_1 - f(\alphav^{(t)}) - + \lambda \norm{\alphav^{(t)}}_1) + \frac{L\gamma^2}{2}\norm{\alphav^\star - \alphav^{(t)}}_1^2 }\\
  	&= h_t + \Theta^2\min_{\gamma \in \R} -\gamma h_t + \gamma^2\frac{L}{2}\norm{\alphav^\star - \alphav^{(t)}}_1^2\\
  	&\leq h_t + \Theta^2\min_{\gamma \in \R} -\gamma h_t + \gamma^2\frac{L D^2}{2}\\
  	&= h_t - \frac{2\Theta^2 h_t^2}{L D^2}\,.
  \end{align*}
  Now we know that the function value is non-increasing both during the {\good} and {\bad} steps from Lemma \ref{lem:bad-progress-L1}. Hence $h_{t}^2 \geq h_t h_{t+1}$. Using this inequality and dividing the entire equation by $h_t h_{t+1}$ gives
  \[
  	\frac{1}{h_{t}} \leq \begin{cases}
  		\frac{1}{h_{t+1}} - \frac{2\Theta^2}{L D^2} & \mbox{ if } t \in \G_t\\
  		\frac{1}{h_{t+1}} & \mbox{ o.w.}
		\end{cases}  		
  		\,. 
  \]
  Summing this up gives that
  \[
  	\frac{1}{h_{t+1}} \geq \frac{1}{h_0} + \frac{\abs{\G_{t+1}}2\Theta^2}{L D^2} \geq \frac{\abs{\G_{t+1}}2\Theta^2}{L D^2}\,.
  \]
  Inverting the above equation gives the required theorem.
\end{proof}

 \begin{remark}
 	We show that the {\steepsub} rule has convergence rates independent of $n$ for $L1$-regularized problems. As long as the update \eqref{eq:l1-update} is kept the same as in Algorithm \ref{alg:L1}, our proof also works for the {\steeplookr} and the {\steeplook} rules. Previously only the {\steeplookr} was analyzed for this case by \cite{DhillonNearestNeighborbased2011}. 
 	A careful reading of their proof gives the rate 
 	\[
 		h_t \leq \frac{8 LD^2}{t}\,.
 	\]
 	Thus we improve the rate of convergence by a factor of 8 even for {\steeplookr}, and show a convergence result for the first time for {\steepsub}. 
 \end{remark}

\section{Theoretical Analysis for Box-constrained Problems}\label{sec:convergence-box}	
 	In this section we examine the algorithm for the box-constrained case. We give formal proofs of the rates of convergence and demonstrate the theoretical superiority of the {\steepsub} rule over uniform coordinate descent.

 Recall the definition of coordinate-wise smoothness of $f$:
  \[
    f(\alphav + \gamma \unit_i) \leq f(\alphav) + \gamma\,\nabla_i f(\alphav) + \frac{L \gamma^2}{2},
  \]
  where ${\unit}_{i}$ for $i \in [n]$ is a coordinate basis vector.
Using strong convexity, for any $\deltav \in \R^n$ gives:
  \[
   f(\alphav + \deltav)  \geq  f(\alphav) + \scal{\nabla f(\alphav)}{\deltav} + \frac{\mu_1}{2}\norm{\deltav}_1^2\,.
  \]

\subsection{The Good, the Bad, and the Cross Steps}
 	Unlike in the $L1$ case, we need to divide the steps into three kinds. Differentiating between the three kinds is key to our analysis.
\begin{definition}[{\good}, {\bad}, and {\cross} steps]
	At any step $t$, let $i$ be the coordinate being updated. Then if $(\alpha_i - \frac{1}{L}\nabla_i f(\alphav)) \in (0,1)$ it is called a {\good} step. If $(\alpha_i - \frac{1}{L}\nabla_i f(\alphav)) \notin (0,1)$, and $\alpha_i \in (0,1)$ the step is considered {\bad}. Finally if $(\alpha_i - \frac{1}{L}\nabla_i f(\alphav)) \notin (0,1)$ and $\alpha_i \in \{0,1\}$, we have a {\cross} step. See Figure \ref{fig:good-bad-cross-box} for illustration.
\end{definition}

\begin{figure}
	\centering
	\def\svgwidth{0.6\columnwidth}
	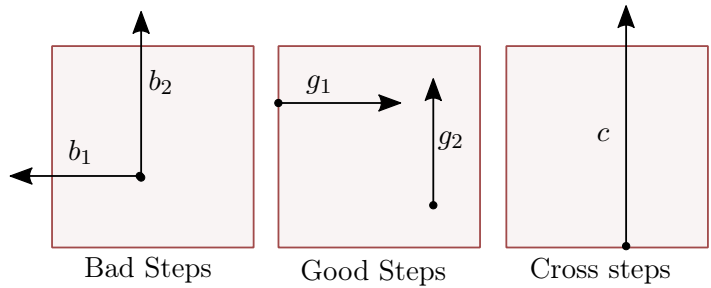
 		\caption{Characterizing steps in box-constrained case (for $n=2$): The {\bad} steps correspond to those which start in the interior and the update $-\frac{1}{L}\nabla_i f(\alphav)$ is interrupted by the boundary of the box constraint (as in steps $b_1$ and $b_2$). The {\good} steps are those which end in the interior of the box ($g_1$ and $g_2$), and the {\cross} steps such as $c$ are those which both start and end at the boundary.}
 		\label{fig:good-bad-cross-box}
\end{figure}
	We would like to bound the number of {\bad} steps. This we can do thanks to the structure of the box constraint.
	\begin{lemma}[Counting {\bad} steps]\label{lem:count-bad-steps-box}
	After $t$ iterations of Algorithm \ref{alg:box-steepest}, we have atmost $\floor{t/2}$ {\bad} steps.
	\end{lemma}
	\begin{proof}
		Suppose we are updating the $i$th coordinate. As is clear from the definition (and Fig. \ref{fig:good-bad-cross-box}), a {\bad} step occurs when we start in the interioir ($\alpha_i \in (0,1)$) and attempt to move outside. But in this case, our update ensures that $\alpha_i^{(t+1)} \in \{0,1\}$. Thus the next time coordinate $i$ is picked, it cannot be a {\bad} step. Since we start at the origin $\0^n$, in the first $t$ steps we can have atmost $\floor{t/2}$ {\bad} steps.
	\end{proof}
 	
\subsection{Progress Made in One Step} 	
This section is the core technical part of the proof. The key to our rate is realizing that the three kinds of steps have to be dealt with separately, and proposing a novel analysis of the {\cross} step.
Recall that we had defined
 	 \[
 		\chi_j(\alphav) \defeq \min_{\gamma + \alpha_j \in [0,1]} \encasecurly{\gamma\nabla_{j} f(\alphav)
 		  + \frac{L\gamma^2}{2}}\,.
 	\] 	
Let us also recall the update step used in Algorithm \ref{alg:box-steepest}. 
The update used is
\begin{equation*}
  \alpha_{i}^+ = \min\encaser{1,\positive{\alpha_{i} -
  \dfrac{1}{L}\nabla_{i}f(\alphav)}}.
\end{equation*}
The {\steepsub} rule used to choose the coordinate $i \in \A$ such that 
\[
	\abs{\nabla_i f(\alphav)} \geq \Theta \max_{j \in \A}\abs{\nabla_j f(\alphav)} \,,
\]
where the active set $\A \subseteq [n]$ consists of the coordinates which have feasible update directions:
\[
	j \in \A \mbox{ if } \exists \gamma > 0 \mbox{ such that } \encaser{\alpha_j - \gamma\nabla_j f(\alphav)} \in [0,1]\,.
\]

Before we begin, we should verify if the algorithm is even feasible---make sure that $\A$ is never empty.
\begin{lemma}[$\A$ is not empty]
	If $\A = \emptyset$, then $\alphav$ is the optimum. 
\end{lemma}
\begin{proof}
	If $\A = \emptyset$, it means that none of the negative gradient directions are feasible i.e. for any $j \in [n]$ and $v \in [0,1]$, $(\nabla_j f(\alphav)(v - \alpha_j) \geq 0)$. This means that for any vector $\vv \in [0,1]^n$,
	\[
		\scal{\nabla f(\alphav)}{\vv - \alphav} = \sum_{j \in [n]}\nabla_j f(\alphav)(v_j - \alpha_j) \geq 0\,.
	\]
	This implies that $\alphav$ is the optimum.
\end{proof}

Now let us first look at the {\good} steps. 	
\begin{lemma}[{\good} steps make a lot of progress]\label{lem:good-step-progress-box} 	
 	Suppose that the $\Theta$-approximate {\steepsub} rule (recall Def. \ref{def:approx-steepest}) chose to update the $i$th coordinate, and further suppose that it was a {\good} step. Then,
 	\[
 		\chi_i(\alphav) \leq \Theta^2 \min_{\deltav \in [0,1]^n - \alphav}\encasecurly{\scal{\nabla f(\alphav)}{\deltav} + \frac{L}{2}\norm{\deltav}_1^2}\,.
 	\]
 	\end{lemma}
 	\begin{proof}
 		Define the right hand side above to be $\P(\deltav)$ defined for $\deltav \in \real^n$ such that $\deltav + \alphav \in [0,1]^n$:
 		\[
 			\P(\deltav) \defeq \scal{\nabla f(\alphav)}{\deltav} + \frac{L}{2}\norm{\deltav}_1^2\,.
 		\]
 		Suppose we are given some $\vv$ in the domain. Let us construct a vector $\tilde\vv \in \R^n$ such that 
 		\[
 			\tilde v_j = 
 				\begin{cases}
 					0, & \mbox{if } \alpha_j = 0, \mbox{ and } \nabla_j f(\alphav) \geq 0\\
 					0, & \mbox{if } \alpha_j = 1, \mbox{ and } \nabla_j f(\alphav) \leq 0\\ 					
 					v_j, & \mbox{otherwise}\,.
 				\end{cases}
 		\]
 		Note that we can instead rewrite $\tilde\vv$ as $\tilde v_j = v_j$ if $j\in \A$, or else is $0$. Also we have $\norm{\vv}_1 \geq \norm{\tilde\vv}_1$.
 		
 		For a coordinate $j \in [n]$ such that $\alpha_j = 0$, it holds that $v_j \in [0,1]$. Then $\nabla_j f(\alphav) \geq 0$ implies that $\nabla_j f(\alphav) v_j \geq 0$. Similarly if $\alpha_j = 1$ and $\nabla_j f(\alphav) \leq 0$, then it holds that $\nabla_j f(\alphav) v_j \geq 0$. Thus we have
	\[
		\P(\tilde\vv) \leq \P(\vv)\,.
	\]
	This means that if we want to minimize $\P$, we can restrict our search space to coordinates in $([n]\setminus\A)$. We will use $\deltav \in [0,1]^n - \alphav, \deltav[\A] = 0$ to mean the set $\deltav \in \R^n$ such that $\deltav + \alphav \in [0,1]^n$ and for all $j \in \A$, $\deltav_j = 0$. We then have that 
	\begin{align*}
		\min_{\deltav \in [0,1]^n - \alphav}\P(\deltav) &= \min_{\deltav \in [0,1]^n - \alphav, \deltav[\A] = 0}\P(\deltav)\\
		&\geq \min_{\deltav \in \R^n, \deltav[\A] = 0}\P(\deltav)\\
		&=\min_{\deltav \in \R^n, \deltav[\A] = 0}\encasecurly{\scal{\nabla f(\alphav)}{\deltav} + \frac{L}{2}\norm{\deltav}_1^2}\\
		&= -\frac{1}{2L}\max_{j \in \A}\abs{\nabla_j f(\alphav)}\\
		&\geq -\frac{\Theta^2}{2L}\abs{\nabla_i f(\alphav)}\,.
	\end{align*}
	The last step is because $i$ was defined to be a $\Theta$-approximate {\steepsub} direction. Now we also know that the update was a {\good} step. This means that $(\alpha_i - \nabla_i f(\alphav)/L) \in [0,1]$ which means that
	\begin{align*}
		\chi_i(\alphav) &= \min_{\gamma + \alpha_i \in [0,1]} \encasecurly{\gamma\nabla_{i} f(\alphav) + \frac{L\gamma^2}{2}}\\
 		  &\leq (\alpha_i^+ - \alpha_i)\nabla_{i} f(\alphav) + \frac{L}{2}(\alpha_i^+ - \alpha_i)^2\\
 		  &= -\frac{1}{2L}\abs{\nabla_i f(\alphav)}\\
 		  &\leq \frac{1}{\Theta^2}\min_{\deltav \in [0,1]^n - \alphav}\P(\deltav)\,.
	\end{align*}
	This finishes the proof of the lemma.
 	\end{proof}
 	
 	We now turn our attention to the {\cross} step which crosses from one end to the other.
 	\begin{lemma}[{\cross} steps also make a lot of progress]\label{lem:cross-step-progress-box} 	
 	Suppose that the $\Theta$-approximate {\steepsub} rule chose to update the $i$th coordinate, and further suppose that it was a {\cross} step. Then, 
 	\[
 		\chi_i(\alphav) \leq \frac{\Theta}{2n} \min_{\vv \in [0,1]^n}\encasecurly{\scal{\nabla f(\alphav)}{\vv - \alphav}}\,.
 	\]
 	\end{lemma}
 	\begin{proof}
 	Since the update was a {\cross} step, this means that $(\alpha_i - \nabla_i f(\alphav)) \neq [0,1]$ and that $\alpha_i^+ \in \{0,1\}$. In particular this imlpies that when solving for the optimal $\gamma$ in the below problem, it is greater than 1:
 	\begin{align*}
 		 		1 &< \argmin_{\gamma \geq 0}\encasecurly{\gamma (\nabla_i f(\alphav))(\alpha_i^+ - \alpha_i) + \frac{L \gamma^2}{2}(\alpha_j^+ - \alpha_j)^2} \\
 		 		&= \frac{-(\nabla_i f(\alphav))(\alpha_j^+ - \alpha_j)}{L(\alpha_i^+ - \alpha_i)^2}\,.
 	\end{align*}
 	Now using this inequality in $\chi_i(\alphav)$ we get
 	\begin{align*}
 		\chi_i(\alphav) &= (\nabla_i f(\alphav))(\alpha_i^+ - \alpha_i) + \frac{L}{2}(\alpha_i^+ - \alpha_i)^2\\
 		&\leq  (\nabla_i f(\alphav))(\alpha_i^+ - \alpha_i) - \frac{1}{2}(\nabla_i f(\alphav))(\alpha_i^+ - \alpha_i)\\
 		&= \frac{1}{2}(\nabla_i f(\alphav))(\alpha_i^+ - \alpha_i)\\
 		&= -\frac{1}{2}\abs{\nabla_i f(\alphav)}\,.
 	\end{align*}
 	The last step is because in a {\cross} step, $\abs{\alpha_i^+ - \alpha_i} = 1$ and is in the opposite direction of the gradient coordinate. Now since $i$ was a  $\Theta$-approximate {\steepsub} direction,
 	\begin{align*}
 	\chi_i(\alphav) &\leq -\frac{1}{2}\abs{\nabla_i f(\alphav)}\\
 	&\leq -\frac{\Theta}{2}\max_{j \in \A}\abs{\nabla_j f(\alphav)}\\
 	&\leq -\frac{\Theta}{2\abs{\A}}\max_{\vv \in [0,1]^n}\scal{\nabla f(\alphav)}{\alphav - \vv}\\
 	&\leq \frac{\Theta}{2n}\min_{\vv \in [0,1]^n}\scal{\nabla f(\alphav)}{\vv - \alphav}\,. \qedhere
 	\end{align*}
 	\end{proof}

 	Finally let us check how bad the {\bad} steps really are.
 	 \begin{lemma}[{\bad} steps are not too bad]\label{lem:bad-progress-box}
 		In any step of Algorithm \ref{alg:box-steepest} including the {\bad} steps, the objective value never increases.
 	\end{lemma}
 	\begin{proof}
 		This directly follows from the fact that we always minimize an upper bound on the function $f(\alphav)$ at every iteration.
 	\end{proof}
 	
 \subsection{Convergence in the Strongly Convex Case}
 \begin{theorem}\label{thm:strong-convex-box-appendix}
 	After $t$ steps of Algorithm \ref{alg:box-steepest} where in each step the coordinate was selected using a $\Theta$-approximate {\steepsub} rule, let $\B_t \subseteq [t]$ indicate the {\bad} steps. Assume that the function $f$ is $L$-coordinate-wise smooth and $\mu_1$ strongly convex with respect to the $L1$ norm. Then the size of $\abs{\B_t} \leq \floor{t/2}$ and
 	\[
 		f(\alphav^{(t+1)}) - f(\alphav^\star) \leq \encaser{1 - \min\encaser{\frac{\Theta}{2n},\frac{\Theta^2 \mu_1}{L}}}^{t - \abs{\B_t}}\encaser{f(\alphav^{(0)}) - f(\alphav^\star)}\,.
 	\]
 	As is standard, $f(\alphav^\star) = \min_{\alphav \in [0,1]^n}f(\alphav)$.
 \end{theorem}
 	
 We have almost everything we need in place to prove our convergence rates for the strongly convex, and the general convex case.
Recall that the function $f$ is $\mu_1$ strongly convex with respect to the $L1$ norm. This implies that
\[
	f(\alphav^\star) \geq f(\alphav) + \scal{\nabla f(\alphav)}{\alphav^\star - \alphav} + \frac{\mu_1}{2}\norm{\alphav^\star - \alphav}_1^2\,.
\]	
 We will need one additional Lemma from \citep{karimireddy2018adaptive} to relate the upper bound we minimize in Lemma \ref{lem:good-step-progress-box}.
 \begin{lemma}[Relating different regularizing constants \citep{karimireddy2018adaptive}]\label{lem:change-mu-box}
 For any vectors $\gv, \wv \in \R^n$, and constants $L \geq \mu > 0$, 
 	\begin{equation*}
 		\min_{\wv \in [0,1]^n}\encasecurly{\scal{\gv}{\wv - \alphav} + \frac{L}{2}\norm{\wv - \alphav}_1^2} \leq \frac{\mu}{L} \min_{\wv \in [0,1]^n}\encasecurly{\scal{\gv}{\wv - \alphav} + \frac{\mu}{2}\norm{\wv - \alphav}_1^2}\,.
 	\end{equation*}
 \end{lemma}	
 	\begin{proof}
 		The function $\scal{\gv}{\wv - \alphav}$ is convex and is 0 when $\wv = \alphav$, the set $[0,1]^n$ is also convex, and $\norm{\wv - \alphav}_1$ is a convex positive function. Thus we can apply Lemma 9 from \citep{karimireddy2018adaptive}.
 	\end{proof}
 
 The proof of the theorem now easily follows. 
\paragraph*{Proof of Theorem \ref{thm:strong-convex-box-appendix}.}
First by Lemma \ref{lem:count-bad-steps-box}, we know that $\abs{\B_t} \leq \floor{t/2}$. Now suppose that $t$ was a {\good} step and updated the $i$th coordinate. The progress made in a proximal update using \eqref{eq:coordinate-progress} and Lemma \ref{lem:good-step-progress-box},
 \begin{align*}
  f(\alphav^{(t+1)}) &\leq f(\alphav^{(t)}) +  \chi_i(\alphav^{(t)})\\
  &\leq f(\alphav^{(t)}) + \Theta^2\min_{\wv \in [0,1]^n}\encasecurly{\scal{\nabla f(\alphav^{(t)})}{\wv - \alphav^{(t)}} + \frac{L}{2}\norm{\wv - \alphav^{(t)}}_1^2}\\
  &\leq f(\alphav^{(t)}) + \frac{\Theta^2\mu_1}{L}\min_{\wv \in [0,1]^n}\encasecurly{\scal{\nabla f(\alphav^{(t)})}{\wv - \alphav^{(t)}} + \frac{\mu_1}{2}\norm{\wv - \alphav^{(t)}}_1^2 }\,.
\end{align*} 	
In the last step we used Lemma \ref{lem:change-mu-box}. We will now use our definition of strong convexity. We have shown that
\begin{align*}
	f(\alphav^{(t+1)}) - f(\alphav^\star) &\leq f(\alphav^{(t)}) - f(\alphav^\star) + \frac{\Theta^2\mu_1}{L}\min_{\wv \in [0,1]^n}\encasecurly{\scal{\nabla f(\alphav^{(t)})}{\wv - \alphav^{(t)}} + \frac{\Theta^2\mu_1}{2}\norm{\wv - \alphav^{(t)}}_1^2}\\
	&\leq  f(\alphav^{(t)}) - f(\alphav^\star) + \frac{\Theta^2\mu_1}{L}\encaser{\scal{\nabla f(\alphav^{(t)})}{\alphav^\star - \alphav^{(t)}} + \frac{\Theta^2\mu_1}{2}\norm{\alphav^\star - \alphav^{(t)}}_1^2}\\
	&= (1 - \frac{\Theta^2\mu_1}{L})\encaser{f(\alphav^{(t)}) - f(\alphav^\star)}\,.
\end{align*}
Now instead suppose that the step $t$ was a {\cross} step. In that case Lemma \ref{lem:cross-step-progress-box} tells us that
\begin{align*}
	f(\alphav^{(t+1)}) &\leq f(\alphav^{(t)}) +  \chi_i(\alphav^{(t)})\\
	&\leq f(\alphav^{(t)}) + \frac{\Theta}{2n}\min_{\wv \in [0,1]^n}\scal{\nabla f(\alphav)}{\wv - \alphav}\\
	&\leq f(\alphav^{(t)}) + \frac{\Theta}{2n} \scal{\nabla f(\alphav)}{\alphav^\star - \alphav}\\
	&\leq f(\alphav^{(t)}) +  \frac{\Theta}{2n}\encaser{f(\alphav^\star) - f(\alphav)}\,.
\end{align*}
In the last step, we used the convexity of $f$. Subtracting $f(\alphav^\star)$ and rearranging gives that
\[
	f(\alphav^{(t+1)}) - f(\alphav^\star) \leq \encaser{1 - \frac{\Theta}{2n}}\encaser{f(\alphav^{(t)}) - f(\alphav^\star)}\,.
\]
Finally if the step was {\bad}, we know form Lemma \ref{lem:bad-progress-box} that the function value does not decrease. Putting these three results about the different cases together gives us the theorem.
 \qed

 \begin{remark}\label{rem:box-is-n-there}
 	In the box-constrained case, we do not quite get the straightforward linear rate involving the condition number $\frac{L}{\mu_1}$ which we were expecting. We instead get a rate which depends on $\max\encaser{\frac{L}{\mu_1}, n}$. The new $n$ term is because of the {\cross} steps. In the case of separable quadratics which a diagonal Hessian:
 	\[
 		\nabla^2 f(\alpha) = \text{diag}(\lambda_1,\dots,\lambda_n)\,.
 	\]	
 	In such a case, \cite{nutini_coordinate_2015}[Section 4.1] show that the constant $\mu_1$ is
 	\[
 		\mu_1 = \encaser{\sum_{j=1}^n \frac{1}{\lambda_j}}^{-1}\,.
 	\]
 	Since $L$ in our case was defined to be the max diagonal element of the Hessian, we have that
 	\[
 		\frac{L}{\mu_1} = \encaser{\max_{j \in [n]} \lambda_j}\encaser{\sum_{j=1}^n \frac{1}{\lambda_j}} \geq \encaser{\sum_{j=1}^n \frac{\lambda_j}{\lambda_j}} = n\,.
 	\]
 	This means that, at least in the separable quadratic case, $\frac{L}{\mu_1} \geq n$ and the rate of convergence depends only on the condition number.
 \end{remark}

 \subsection{Convergence in the General Convex Case}
 \begin{theorem}\label{thm:general-convex-box-appendix}
 			After $t$ steps of Algorithm \ref{alg:box-steepest} where in each step the coordinate was selected using the $\Theta$-approximate {\steepsub} rule, let $\B_t \subseteq [t]$ indicate the {\bad} steps. Assume that the function was $L$-coordinate-wise smooth. Then the size $\abs{\B_t} \leq \floor{t/2}$ and
	\[
		f(\alphav^{(t+1)}) - f(\alphav^\star) \leq \frac{8 L D^2}{\Theta^2 (t - \abs{\B_t})}\,,
	\]
	where $D$ is the $L1$ diameter of the level set. For the set of minima $\Q^\star$,
	\[
		D = \max_{\wv \in [0,1]^n}\min_{\alphav^\star \in \Q^\star}\encasecurly{\norm{\wv - \alphav^\star}_1 | f(\wv) \leq f(\alphav^{(0)})}\,.
	\]
\end{theorem}
\begin{proof}
	This essentially follows from the proof of \citep{DhillonNearestNeighborbased2011}[Lemma 8]. As we saw before, in {\good} steps all three rules {\steepsub}, {\steeplookr} and {\steeplook} coincide. For {\cross} steps, by definition we take the largest step possible and so is also a valid {\steeplookr} step. Thus we can fall back onto to the analysis of {\steeplookr} even for the {\steepsub} rule. One could also derive a rate directly from Lemma \ref{lem:good-step-progress-box} and Lemma \ref{lem:cross-step-progress-box}, but we would not achieve a significantly better rate. 
\end{proof}

\section{Proofs of Mapping Between $\steepsub$ and $\MIPS$}\label{sec:proofs-mapping}
In this section we will discuss the proof of our claims that our formulation of casting the problem of finding the steepest $\steepsub$ direction as an instance of the $\SMIPS$ problem. 
 \paragraph{Proof of Lemma \ref{lem:box-steepest} [Box case]:}
 We want to show that
 $$
 \max_i \encase{\min_{s\in \partial g_i} \abs{\nabla_i f(\alphav^{(t)}) + \sv }} = \max_{i \in \A_t}\abs{\nabla_i f(\alphav^{(t)})}\,,
 $$
 where the active set $\A_t$ is defined as
 $$
 	\A_t := \encasecurly{i \in [n]\ \text{ s.t. } \
 	\begin{split}
 		\alpha_i^{(t)} &\in (0,1), \text{ or}\\
 		\alpha_i^{(t)} &= 0 \text{ and } \nabla_i f(\alphav^{(t)}) < 0, \text{ or}\\
 		\alpha_i^{(t)} &= 1 \text{ and } \nabla_i f(\alphav^{(t)}) > 0 \,.
 	\end{split}}
 $$
 Let us examine the subgradient of the indicator function $\ind{[0,1]}$. If $\alpha_i^{(t)} \in (0,1)$, the subgradient of the indicator function is 0. If $\alpha_i^{(t)} = 0$, the subgradient is $\partial g_i = (-\infty,0]$ and if $\alpha_i^{(t)} = 1$, the subgradient is $\partial g_i = [0,\infty)$. Thus $\min_{s\in \partial g_i} \abs{\nabla_i f(\alphav^{(t)}) + \sv }$ equals $\abs{\nabla_i f(\alphav^{(t)})}$ if $\alpha_i \in (0,1)$, or if $\alpha_i^{(t)} = 0 \text{ and } \nabla_i f(\alphav^{(t)}) < 0, \text{ or } \alpha_i^{(t)} = 1 \text{ and } \nabla_i f(\alphav^{(t)}) > 0$. In all other cases, it is 0. This proves the lemma.
 \paragraph{Proof of Lemma \ref{lem:l1-steepest} [$L1$ case]:}
 We want to show that 
 $$
 	\max_i \encase{\min_{s\in \partial g_i} \abs{\nabla_i f(\alphav^{(t)}) + \sv }} = \max_{i}\abs{s(\alphav)_i}
 $$
 Here $g_i(\alpha_i^{(t)}) = \lambda \abs{\alpha_i^{(t)}}$. If $\alpha_i^{(t)} \neq 0$, $\partial g_i = \lambda\sign(\alpha_i^{(t)})$. If $\alpha_i^{(t)} = 0$, $\partial g_i = [-\lambda, \lambda]$. Thus the value of $\min_{s\in \partial g_i} \abs{\nabla_i f(\alphav^{(t)}) + \sv }$ is $\abs{\nabla_i f(\alphav^{(t)}) + \lambda\sign(\alpha_i^{(t)})}$ if $\alpha_i^{(t)} \neq 0$. If $\alpha_i^{(t)} = 0$, then  $\min_{s\in \partial g_i} \abs{\nabla_i f(\alphav^{(t)}) + \sv }$ evaluates to $\abs{\shrinkage_{\lambda}(\nabla_i f(\alphav^{(t)}))}$. In short, $\min_{s\in \partial g_i} \abs{\nabla_i f(\alphav^{(t)}) + \sv }$ exactly evaluates to $\sv(\alphav^{(t)})$. This finishes the proof of the lemma.
 \paragraph{Proof of Lemma \ref{lem:box-steepest-efficient} [Box case]:}
 We want to show that
 $$
 	\max_{\tilde{\Av}_j \in \B_t}\scal{\tilde{\Av}_j}{\qv_t} = \argmax_{j \in \A_t}\abs{\nabla_j f(\alphav^{(t)})}\,.
 $$
 Given the structure of $f(\alphav)$, we can rewrite $\nabla_j f(\alphav^{(t)}) = \scal{\Av_j}{\nabla l(A \alphav^{(t)})} + c_j$. Further we can write
 \begin{equation}\label{eq:box-steep-rewrite}
 	\argmax_{j \in \A_t}\abs{\nabla_j f(\alphav^{(t)})} = \argmax_{j \in \A_t}\ \max \encaser{ \ \encasecurly{\scal{\Av_j}{\nabla l(A \alphav^{(t)})} + c_i} \ ,\ \encasecurly{\scal{-\Av_j}{\nabla l(A \alphav^{(t)})} - c_i}}\,.
 \end{equation}
 Let us examine at the cases we need to ignore i.e. the points not in $\A_t$: i) $\alpha_i^{(t)} = 0$ and $\nabla_i f(\alphav^{(t)}) > 0$ or ii) $\alpha_i^{(t)} = 1$ and $\nabla_i f(\alphav^{(t)}) < 0$. Suppose some coordinate $j'$ falls into case (i). Instead of excluding it from $\A_t$, we could instead keep only the $\scal{-\Av_{j'}}{\nabla l(A \alphav^{(t)})} - c_i$ in the right hand side of \eqref{eq:box-steep-rewrite} in place of $\max\encaser{\scal{\Av_{j'}}{\nabla l(A \alphav^{(t)})} +c_i, \scal{-\Av_{j'}}{\nabla l(A \alphav^{(t)})}} -c_i$. This automatically excludes $j'$ from the $\argmax$ if $\nabla l(A \alphav^{(t)}) + c_i > 0$ or equivalently if $\nabla_{j'} f(\alphav^{(t)}) > 0$. Similarly if $j'$ falls in case (ii), we will keep only the $\scal{\Av_{j'}}{\nabla l(A \alphav^{(t)}) + c_i}$ term which excludes $j'$ from the $\argmax$ if $\nabla l(A \alphav^{(t)}) + c_i < 0$ or equivalently if $\nabla_{j'} f(\alphav^{(t)}) < 0$. This is exactly what the term $\max_{\tilde{\Av}_j \in \B_t}\scal{\tilde{\Av}_j}{\qv_t}$ does.
 \paragraph{Proof of Lemma \ref{lem:L1-steepest-efficient} [$L1$ case]:}
 We want to show that  
 $$
 	\max_{\tilde{\Av}_j \in \B_t}\scal{\tilde{\Av}_j}{\qv_t} = \argmax_{i}\abs{s(\alphav)_i}\,.
 $$
 Given the structure of $f(\alphav)$, we can rewrite $\nabla_j f(\alphav^{(t)}) = \scal{\Av_j}{\nabla l(A \alphav^{(t)})} + c_j$. Further we can use the equivalence in Lemma \ref{lem:l1-steepest} to reformulate the required statement as
 \begin{equation}\label{eq:L1-steepest-reformulation}
 	\max_{\tilde{\Av}_j \in \B_t}\scal{\tilde{\Av}_j}{\qv_t} = \argmax_{j}\abs{\min_{s \in \partial \abs{\alpha^{(t)}_j}} \scal{\Av_j}{\nabla l(A \alphav^{(t)})} + c_j + \lambda s}\,.
 \end{equation}
 Drawing from the proof of Lemma \ref{lem:l1-steepest}, we look at the following cases:
 \begin{enumerate}
 	\item{When $\alpha_j > 0$:} In this case $\partial \abs{\alpha^{(t)}_j}$ evaluates to 1 and the right side of \eqref{eq:L1-steepest-reformulation} becomes $$\max\encaser{ \scal{\Av_j}{\nabla l(A \alphav^{(t)})} + c_j + \lambda,  \scal{-\Av_j}{\nabla l(A \alphav^{(t)})} - c_j -  \lambda}\,.$$
 	\item{When $\alpha_j < 0$:} In this case $\partial \abs{\alpha^{(t)}_j}$ evaluates to -1 and the right side of \eqref{eq:L1-steepest-reformulation} becomes $$\max\encaser{ \scal{\Av_j}{\nabla l(A \alphav^{(t)})} + c_j - \lambda,  \scal{-\Av_j}{\nabla l(A \alphav^{(t)})} - c_j + \lambda}\,.$$
 	\item{When $\alpha_j = 0$ and $\scal{\Av_j}{\nabla l(A \alphav^{(t)})} + c_j \geq 0$:} The right side of \eqref{eq:L1-steepest-reformulation} becomes $$\max\encaser{ \scal{\Av_j}{\nabla l(A \alphav^{(t)})} + c_j - \lambda,  0}\,.$$
 	\item{When $\alpha_j = 0$ and $\scal{\Av_j}{\nabla l(A \alphav^{(t)})} + c_j \leq 0$:} The right side of \eqref{eq:L1-steepest-reformulation} becomes $$\max\encaser{ \scal{-\Av_j}{\nabla l(A \alphav^{(t)})} - c_j - \lambda,  0}\,.$$
 \end{enumerate}
 Comparing the four cases above to the definition of $\B_t$ shows that the lemma holds.

 \section{Further Extensions of the Steepest and $\MIPS$ Framework}\label{sec:extending}
 In this section we show that the steepest rules $\steepsub$ can be extended to take into account different smoothness constants along different coordinate directions and that the new rules are also instances of our $\MIPS$ framework. Further we show that the steepest rule for $L2$-regularized problems can also be cast as an $\MIPS$ instance.

 \subsection{Coordinate Specific Smoothness Constants}
 The {\steepsub} and the {\steeplook} strategies as described in this work reduce to the {\gs} rule when the function is smooth and $g =0$. However if the smoothness constants (Definition \ref{def:smoothness}) highly varies along different coordinate directions, then the {\gs}-L rule which accounts for this variation enjoys better theoretical convergence properties. We can similarly modify the {\steepsub} and the {\steeplook} strategies to account for coordinate specific smoothness constants.

 Suppose that we modify Definition \ref{def:smoothness} as in \cite{nutini_coordinate_2015} and let $L_i$ be the smoothness constant along coordinate direction $i$. Then we rescale the $i$th coordinate by $1/L_i$. 
 This step is equivalent to normalizing the columns of the data matrix $A$ when $f(\alphav)$ is of the form $l(A\alphav)$. The rescaling will make the smoothness constants of all the coordinates equal (in fact equal to 1). Now we can apply the techniques in this work to obtain steepest rules which implicitly take advantage of the coordinate specific smoothness constants $L_i$. Of course this would change the regularization parameter $\lambda$ along the coordinates, and we will have a coordinate-wise regularization constants $\lambda_i$. Our algorithms and proofs can easily be extended to this setting.

\subsection{Block Coordinate Setting}
Just as in the smooth case, it might be possible to also extend the theoretical results here for the block coordinate case using new norms in place of the $L1$ norm in which the analysis is currently done. Define $\xv_{[:i]}$ to mean a the $i$ largest absolute values of $\xv$ and $\xv_{[i:]}$ to mean the $n-i$ smallest absolute values. If we want to pick the top $\kappa$ coordinates for e.g., we could define a new norm using which the analysis could extend
\[
	\norm{\xv}_{[\kappa]}^2 \defeq  \norm{\xv_{[:n+1 -\kappa]}}_1^2 + \norm{\xv_{[n+1 -\kappa:]}}_2^2\,.
\]

 \subsection{Solving $\ell_2$ Regularized Problems using $\SMIPS$}
 Suppose we are given a problem of the form
 $$
   l(A\alphav) + \frac{\lambda}{2}\norm{\alphav}^2 \,.
 $$
 In this case, the {\steepsub} rule simplifies to
 $$
   \argmax_{j \in [n]}\abs{\scal{\Av_j}{\nabla l(A \alphav)} + \lambda \alpha_j}\,.
 $$
 We can cast this in the $\MIPS$ framework. We need to add $n$ new components to the $d$ dimensional vector $\Av_j$ as follows:
 \begin{equation}\label{eq:A-tilde-L2}
   \tilde{\Av}_i := \begin{pmatrix} \beta\unit_i \\ \Av_i \end{pmatrix}\,.
 \end{equation}
 The constant $\beta$ is tuned to ensure good performance of the underlying MIPS algorithm, and is typically chosen to be $O(1/\sqrt{n})$.
 Then, define
 \begin{align}\label{eq:P-l2-def}
   \begin{split}
     \P &:= \P^\plus \cup \P^\minus \,\text{where,}\\
     \P^\pm &:= \encasecurly{\pm\tilde{\Av}_1,\dots,\pm\tilde{\Av}_n} \,.
   \end{split}
 \end{align}
 Finally, construct the query vector $\qv_t$ as
 \begin{equation}\label{eq:query-l2-steepest}
   \qv_t := \begin{pmatrix} \frac{\lambda}{\beta}\alphav^{(t)} \\ \nabla l(A \alphav^{(t)})\end{pmatrix}\,. 
 \end{equation}
 Then it is easy to verify the following equivalence -
 $$
   \MIPS_\P(\qv_t) = \argmax_{j \in [n]}\abs{\scal{\Av_j}{\nabla l(A \alphav)} + \lambda \alpha_j}\,.
 $$

 However we are now dealing with vectors of size $d + n$ instead of just $d$ which is usually too expensive ($n \gg d$). Handling $\tilde{\Av}_i$ is easy since it has at most $d+1$ non-zero entires. To get around the computational cost of hashing $\qv_t$, we note that the hyperplane hashing only requires computing $\scal{\wv}{\qv_t}$ for some vector $\wv$. This can be broken down as
 $$
   \scal{\wv}{\qv_t} = \frac{\lambda}{\beta}\scal{\wv_1}{\alphav^{(t)}} + \scal{\wv_2}{\nabla l(A \alphav^{(t)})}\,.
 $$
 At iteration $t+1$, we will need to recompute the second part in time $O(d)$. However since $\alphav^{(t+1)}$ and $\alphav^{(t)}$ differ in only 1 component (say $i_t$), updating the first part of the hash computation can be done efficiently as
 $$
   \scal{\wv}{\qv_{t+1}} = \frac{\lambda}{\beta}(\scal{\wv_1}{\alphav^{(t)}} + w_{i_t}(\alpha_{i_{t+1}} - \alpha_{i_{t}})) + \scal{\wv_2}{\nabla l(A \alphav^{(t +1)})}\,.
 $$
 This can also be combined with ideas from Section \ref{sec:efficient} to solve problems which have both an $L2$-regularizer and are box-constrained or $L1$-regularized. This is important for example for solving elastic net regression.

\section{Additional Experiments}\label{sec:figures}
\subsection{Consistent poor performance of {\dhillon}}
Here we run {\dhillon} on the {\sector} datasets. Figure \ref{fig:dhillon} illustrates consistent poor performance of {\dhillon} rule.
\begin{figure}
	\hfill
	\includegraphics[width=0.3\linewidth]{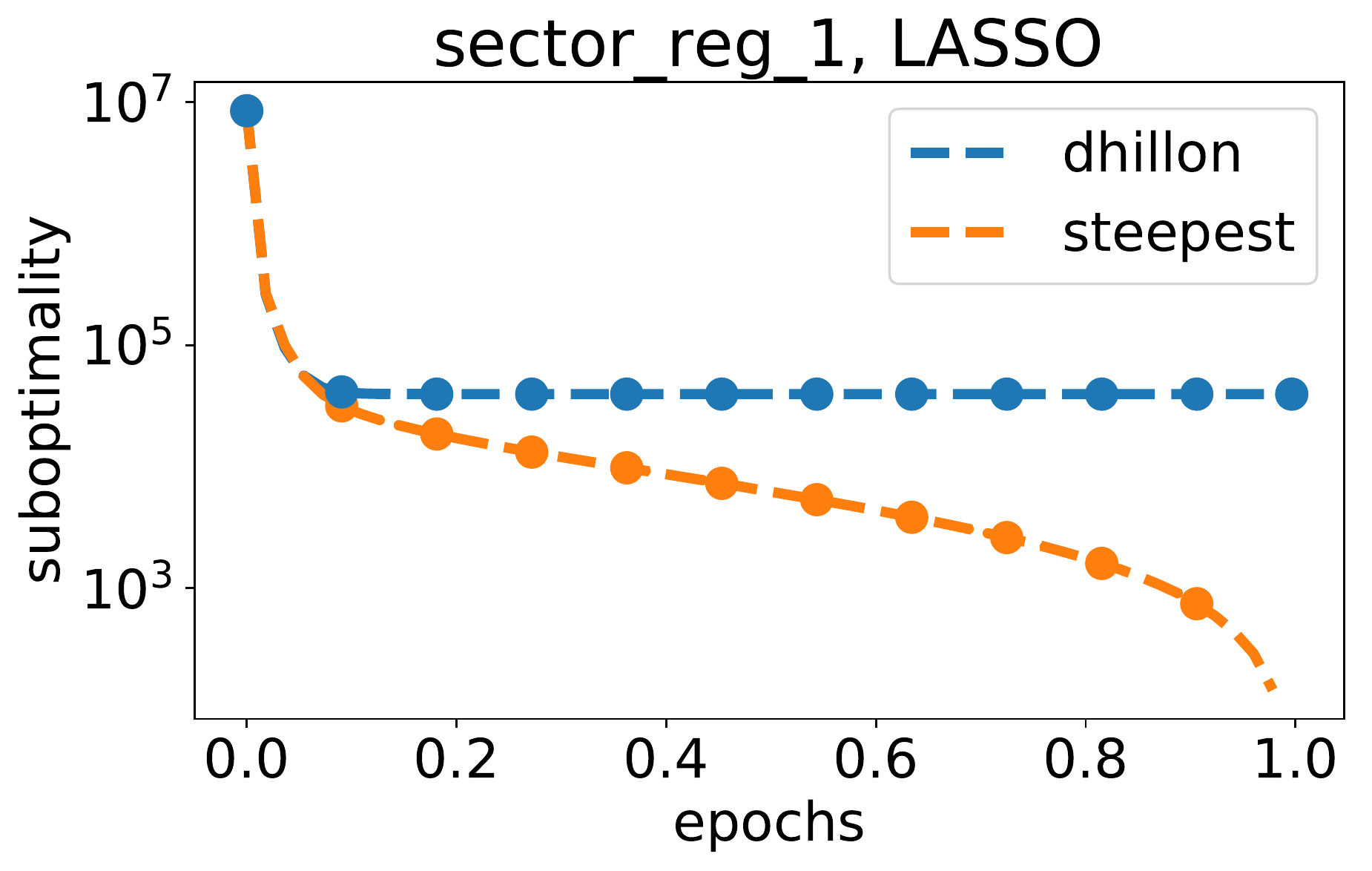}\hfill
	\includegraphics[width=0.3\linewidth]{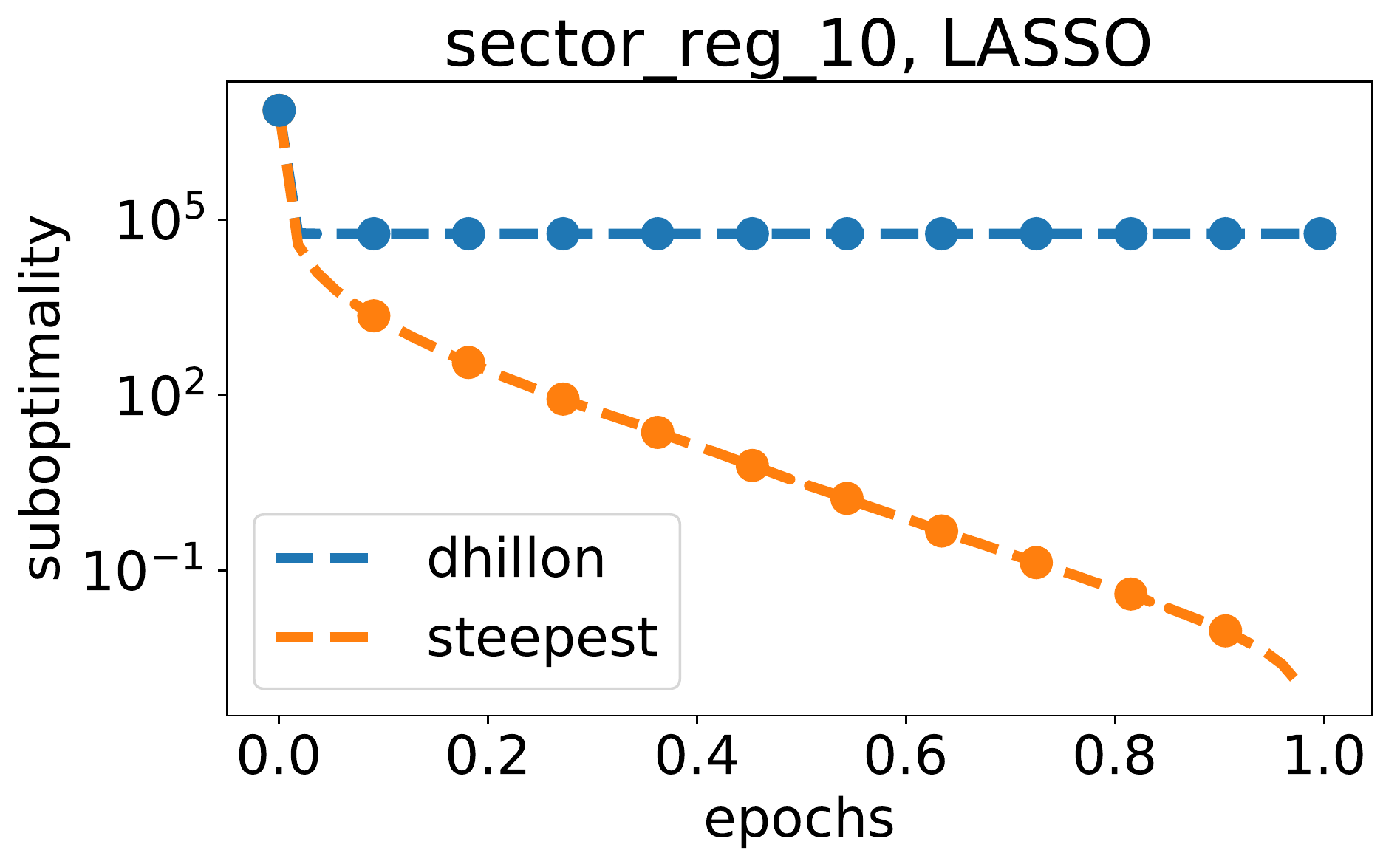}\hfill
	\caption{ Evaluating {\dhillon}: {\steepest} which is based on the {\steepsub} rule consistently outperforms {\dhillon} which quickly stagnates.}
	\label{fig:dhillon}
\end{figure}

\subsection{{\nmslib} on {\sector} dataset}
We perform additional evaluation of {\nms} algorithm on {\sector} dataset for the LASSO. {\nmslib} hyper-parameters and sparsity levels could be found in Table \ref{tab:datasets-sector}. Figure \ref{fig:sector-nmslib} shows that {\nms} for {\sector} dataset has the same strengths and weaknesses.

\begin{table}[]
	\centering
	\caption{Parameters for {\sector} dataset. ($\mathbf{d}$, $\mathbf{n}$) is dataset size, $\mathbf{\beta}$ is tunable constant from \eqref{eq:A-tilde-L1}, \eqref{eq:query-L1-steepest}, $\mathbb{\rho}$ is sparsity of the solution.}
	\label{tab:datasets-sector}
	{\small
		\setlength{\tabcolsep}{2pt}
		\begin{tabular}{lrrrrrrr}
			\hline
			\textbf{Dataset} & {$\mathbf{n}$} & {$\mathbf{d}$} & {$\sqrt{n}\mathbf{\beta}$} &  {efC}  & {efS} & {post} & {$\mathbb{\rho}$}\\ \hline
			{\sector}, $\lambda = 1$      &  55,197          & 6,412         & 10    & 100 & 100 &  2 & 10\\
			{\sector}, $\lambda = 10$      &  55,197          & 6,412         & 10    & 400 & 200  & 0& 3\\
		\end{tabular}
	}
\end{table}

\begin{figure}[!h]
	\hfill
	\includegraphics[width=0.24\linewidth]{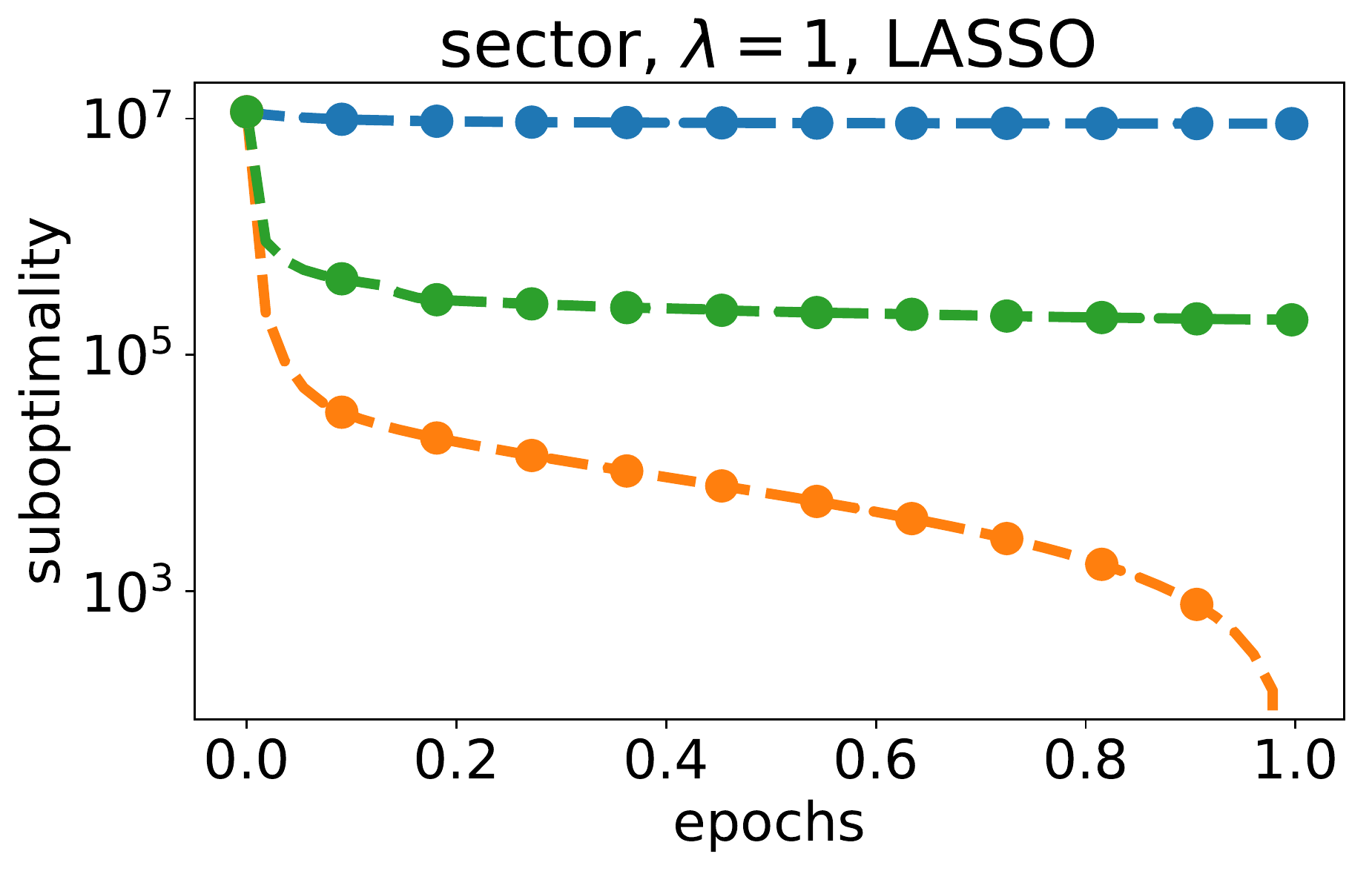}\hfill
	\includegraphics[width=0.24\linewidth]{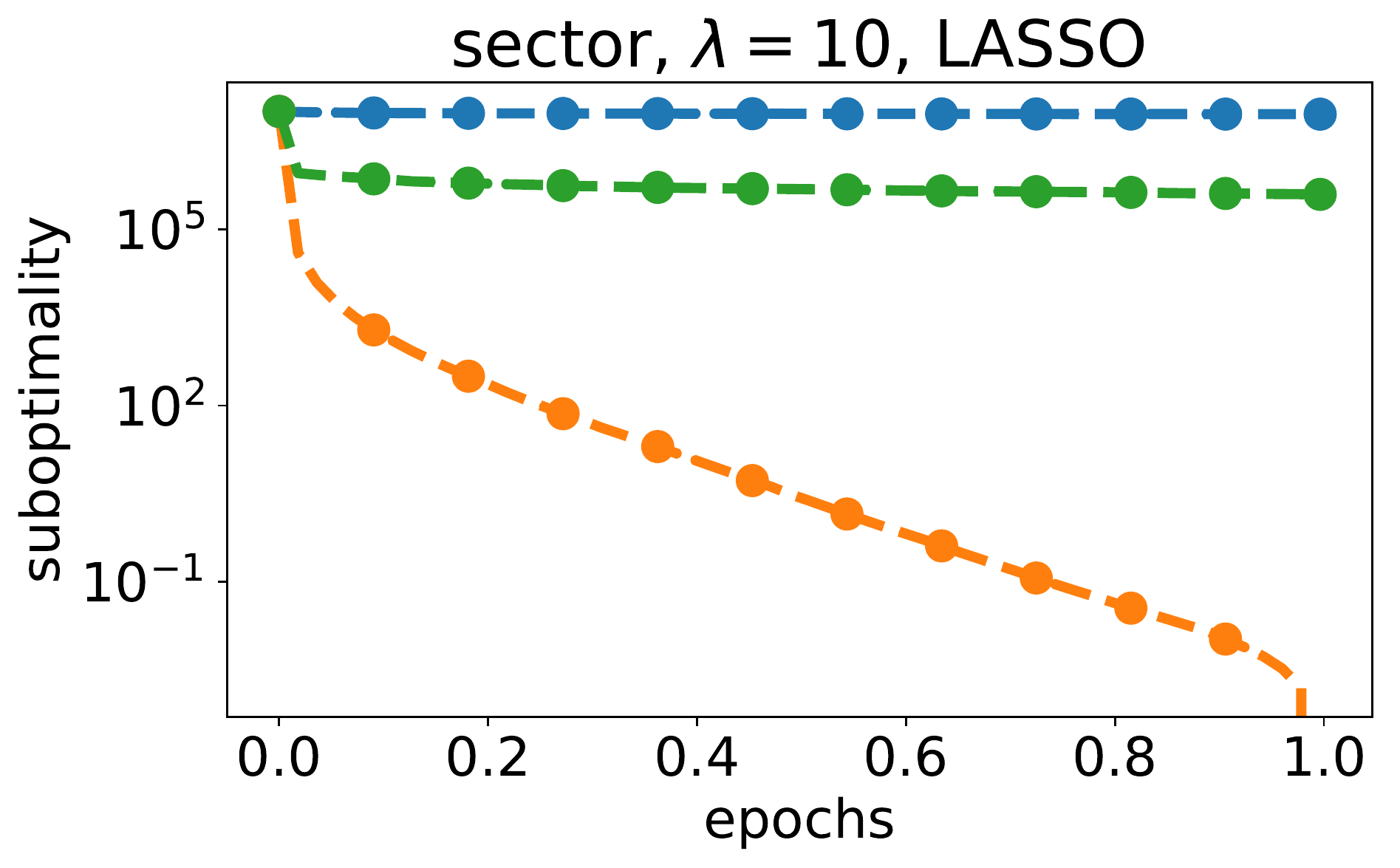}\hfill
	\includegraphics[width=0.24\linewidth]{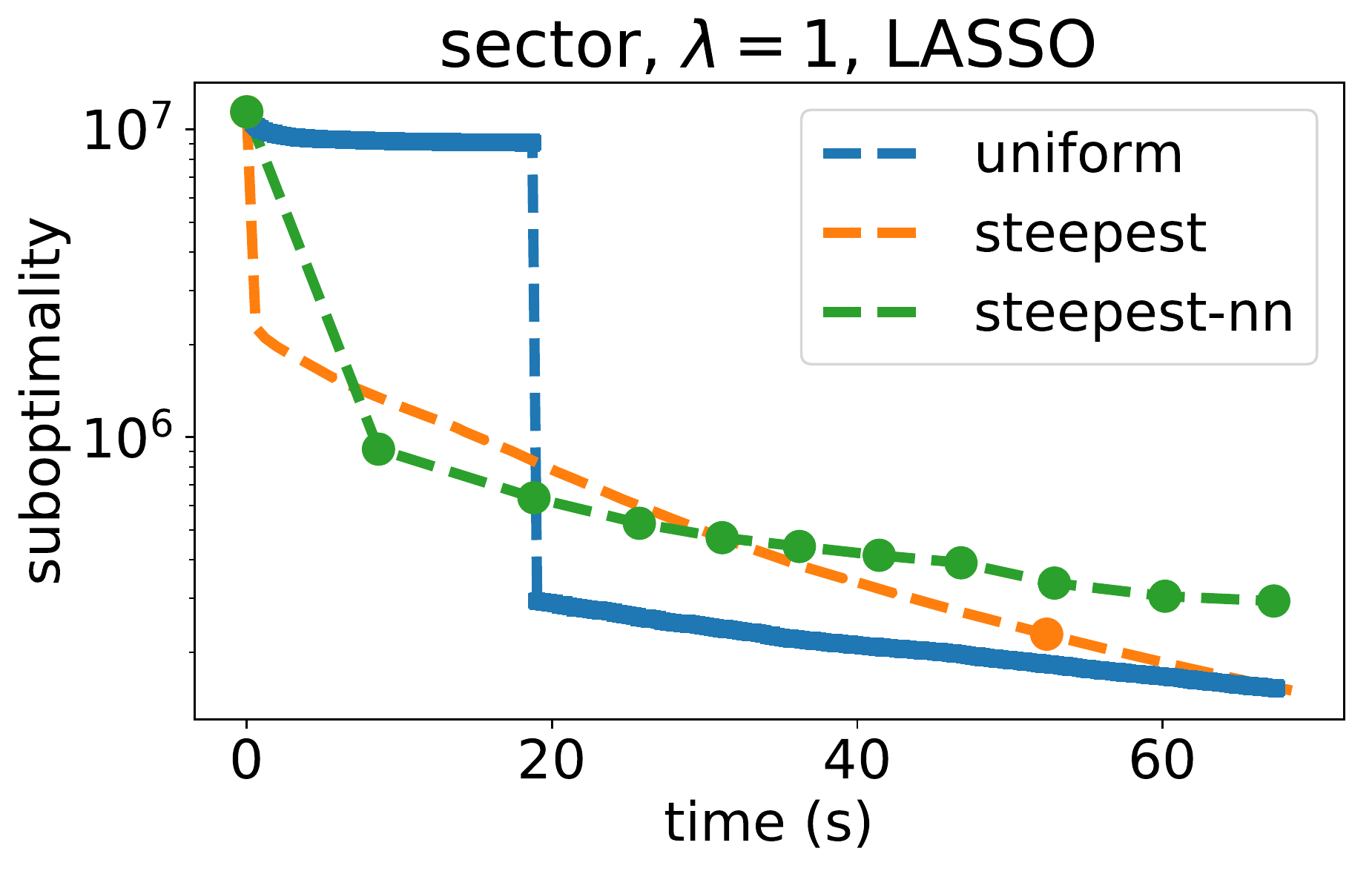}\hfill
	\includegraphics[width=0.24\linewidth]{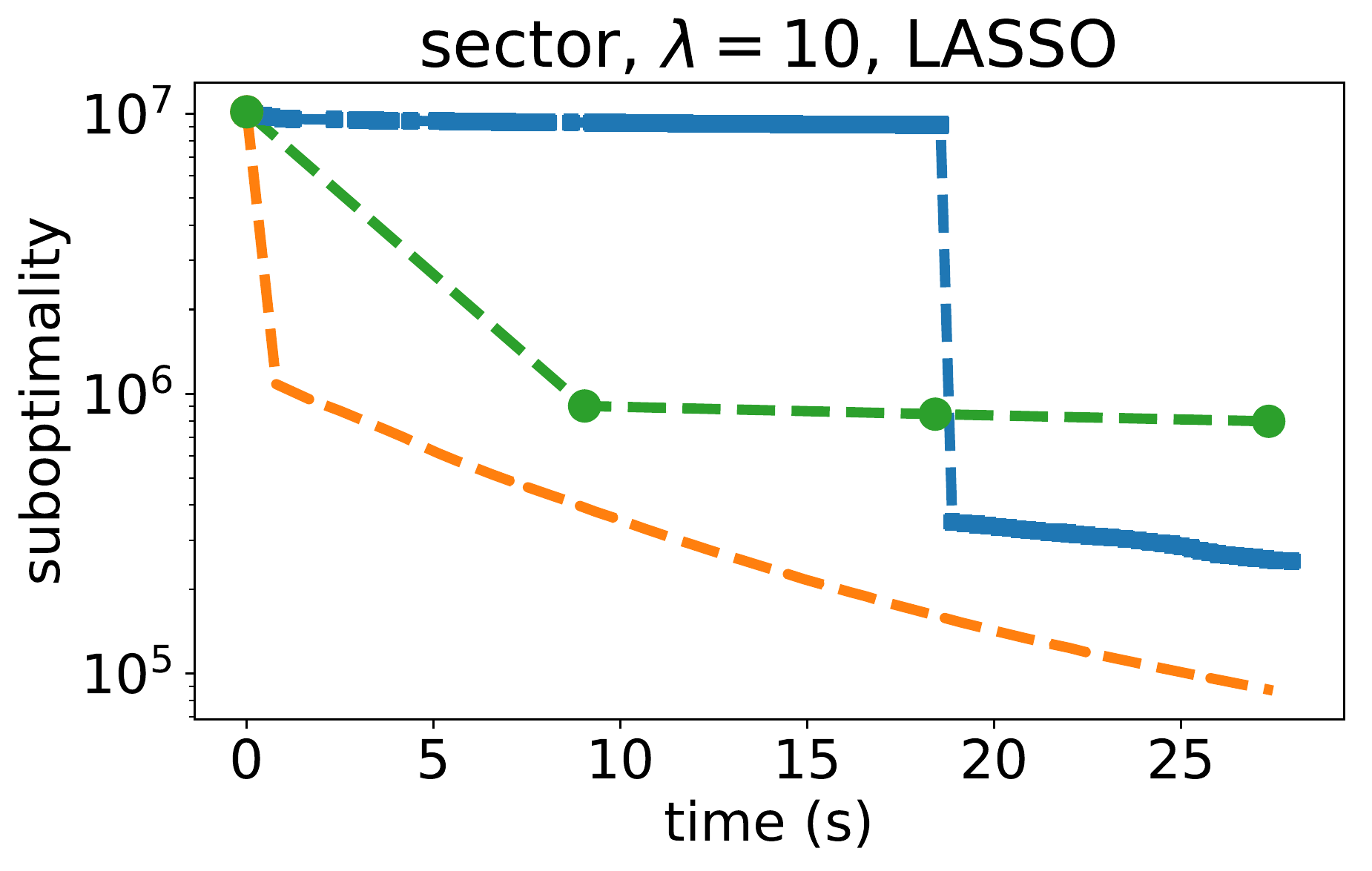}\hfill
	\caption{The performance of {\nms} for {\sector} dataset is similar to that on {\rcv}. }
	\label{fig:sector-nmslib}
\end{figure}

\subsection{Make regression}

We also use the {\makeregression} function from Sk-learn \citep{pedregosa2011scikit}\footnote{\url{http://scikit-learn.org/stable/modules/generated/sklearn.datasets.make_regression.html}} to create tunable instances for Lasso. Here we keep the number of informative features fixed at 100, fix $d = 1000$, and vary $n$ in $\encasecurly{10^4, 10^5, 10^6}$. From Fig.~\ref{fig:make-regression}, we can see that as $n$ increases, {\nms} starts significantly outperforming {\uniform} in terms of wall clock time. This is to be expected since the cost per iteration of {\nms} remains roughly fixed, whereas the gain in progress over {\uniform} grows with $n$.

\begin{figure}[!h]
    \hfill
	\includegraphics[width=0.3\linewidth]{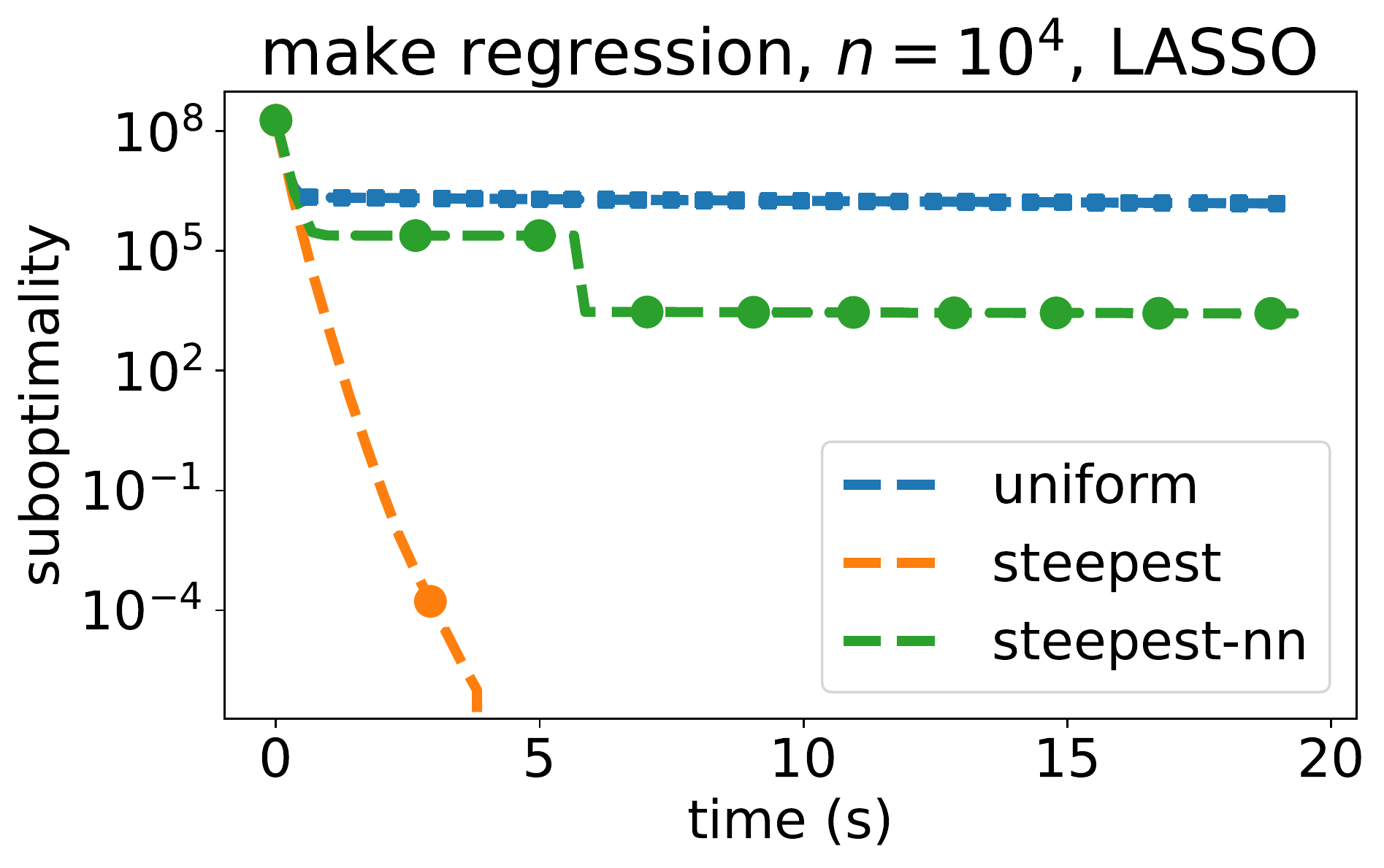}\hfill
	\includegraphics[width=0.3\linewidth]{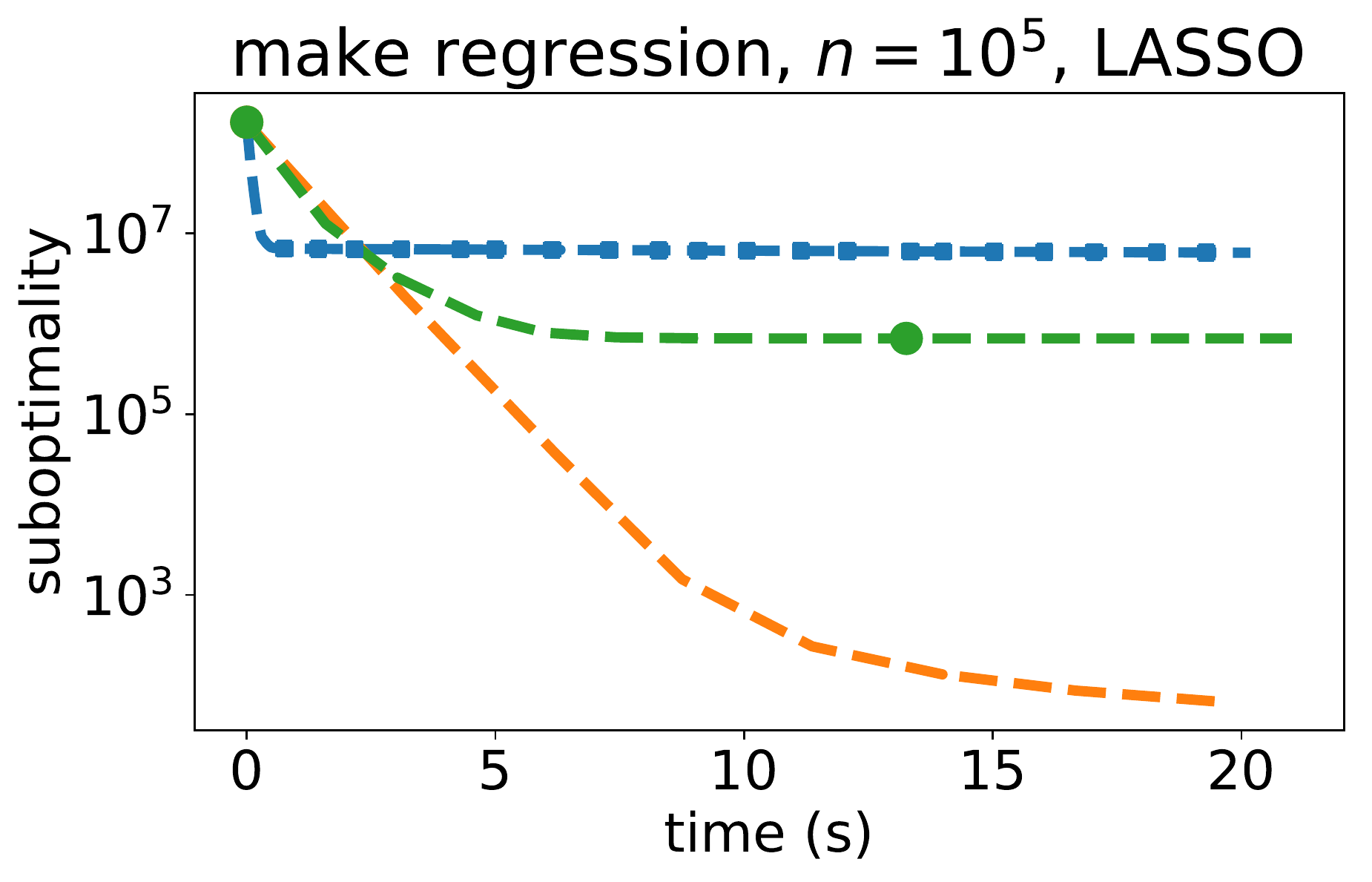}\hfill
	\includegraphics[width=0.3\linewidth]{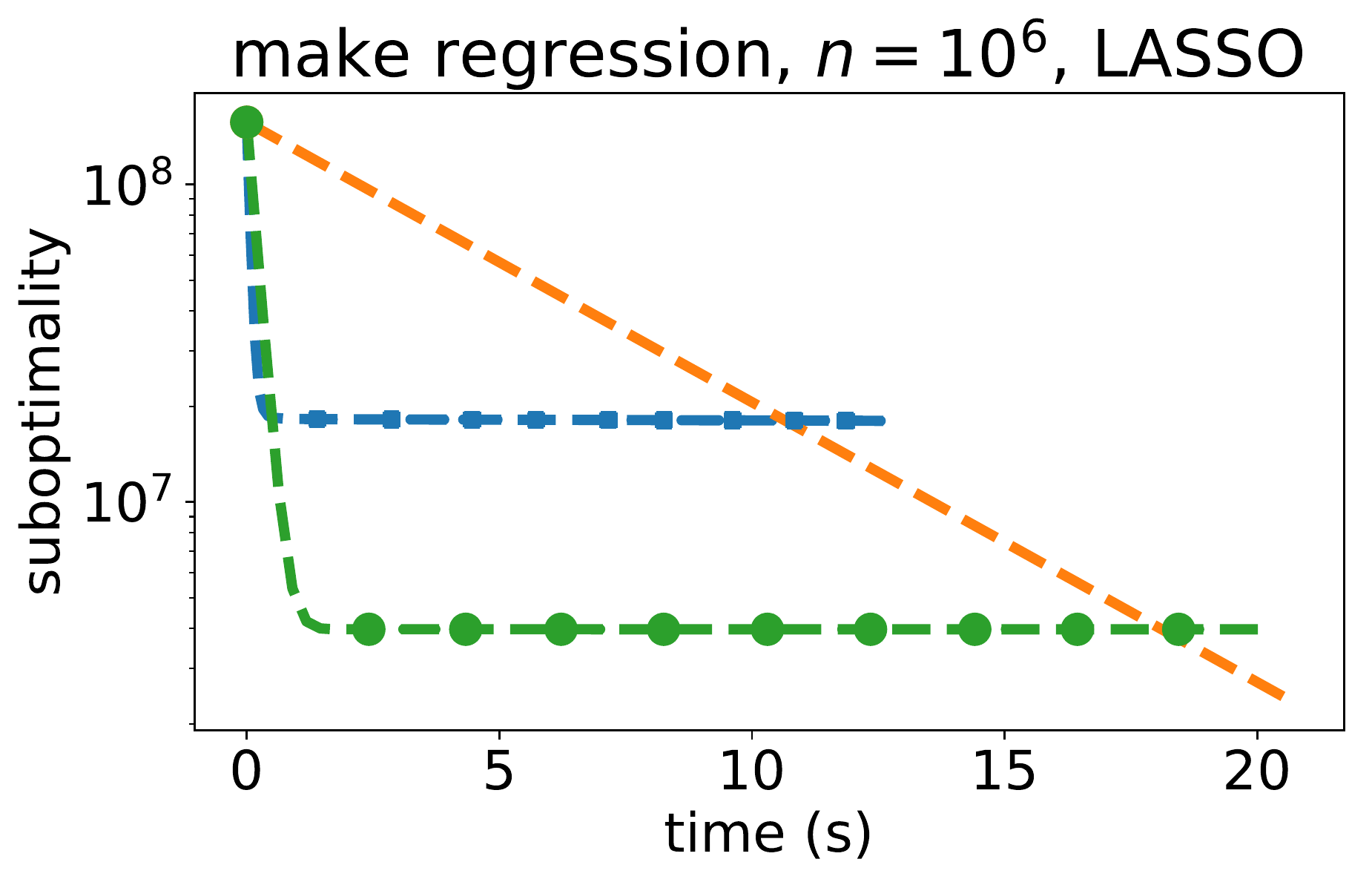}\hfill\null
	\caption{ Performance on {\makeregression} with $n \in \{10^4,10^5,10^6\}$. The advantage of {\nms} over {\uniform} and {\steepest} increases with increasing problem sizes.}
	\label{fig:make-regression}
\end{figure}

\subsection{Extreme Sparse Solution}

Here we experiment with LASSO on {\sector} dataset in extreme case, with $\lambda = 100$. The solution has only 5 non-zero coordinates. Figure \ref{fig:sector-sparse} shows comparison of {\nms} (implemented with {\nmslib}), {\uniform} and {\steepest} rules. {\steepest} converges to the true optimum in less than 10 iterations while {\nms} and {\uniform} strugling to find a good coordinate. As in other experiments, {\nms} shows its usefulness in early stage of optimization.

\begin{figure}[!h]
	\hfill
	\includegraphics[width=0.49\linewidth]{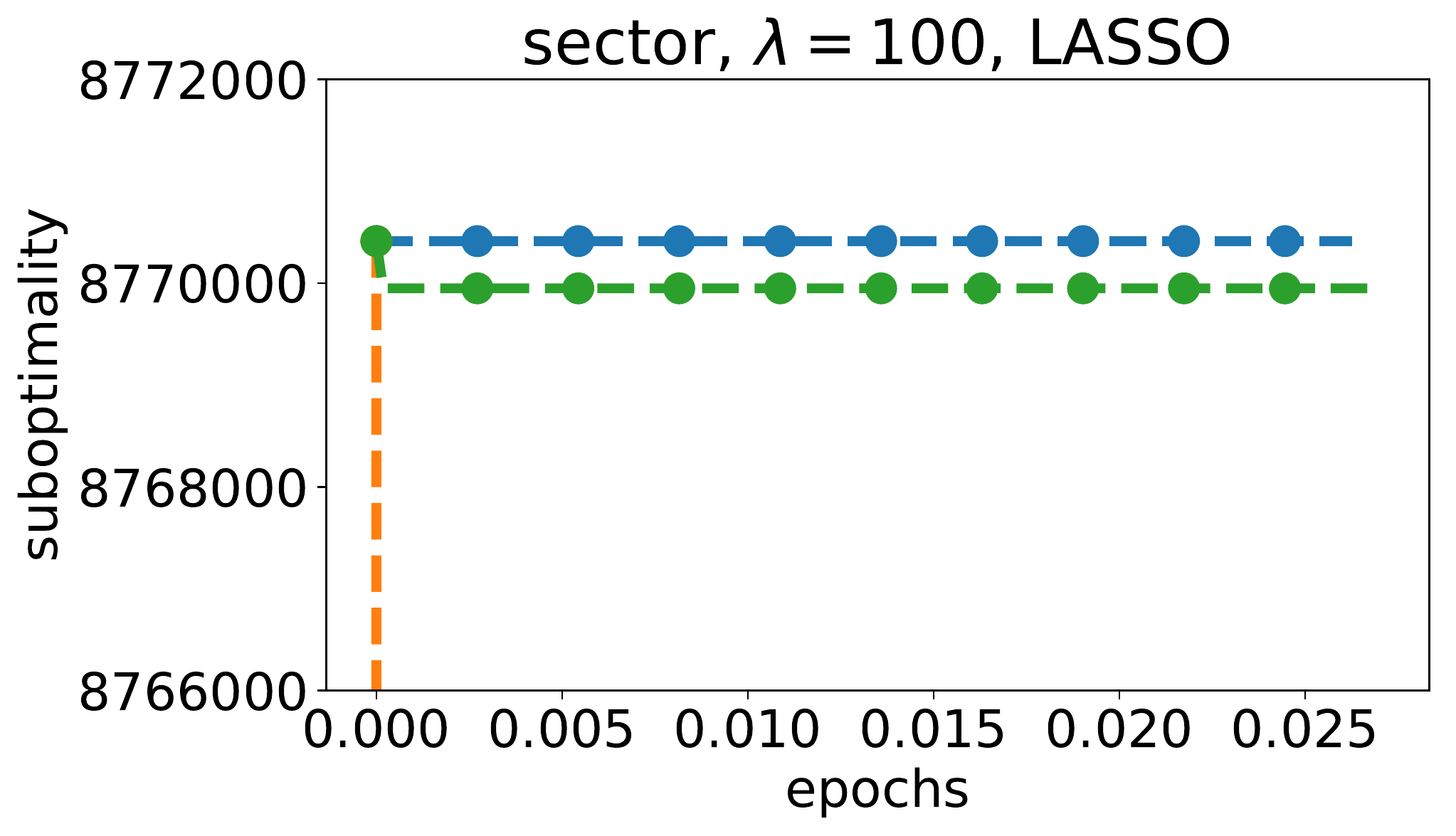}\hfill
	\includegraphics[width=0.49\linewidth]{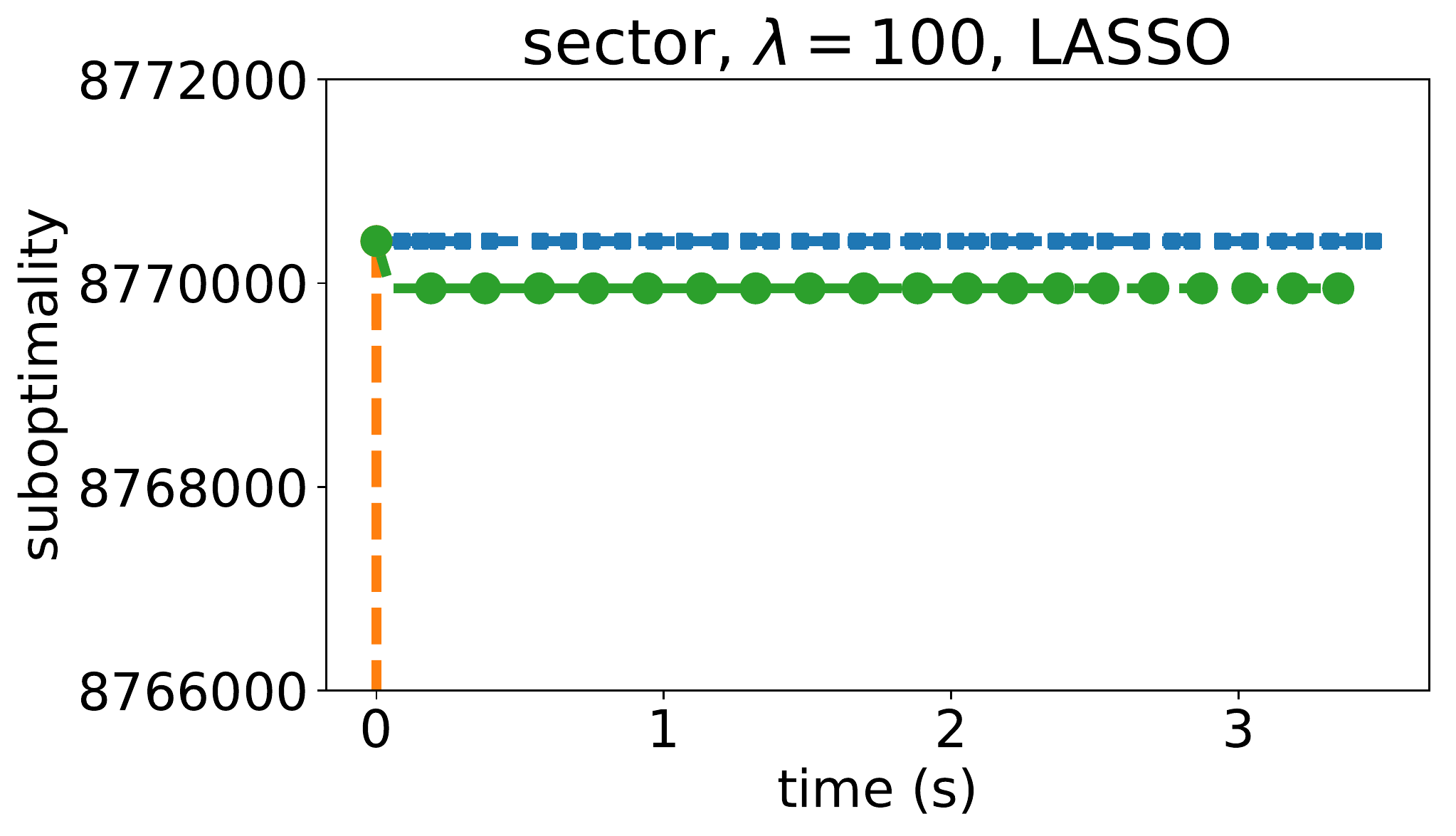}\hfill
	\caption{ The solution is extremely sparse. {\steepest} converges in less than 10 iterations. {\nms} has a small gain over {\uniform}, but both of them extremely bad compared to {\steepest} even in wall-time.}
	\label{fig:sector-sparse}
\end{figure}

\subsection{Test Accuracy}
For SVM experiments we randomly split the datasets to create train (75\%) and test (25\%) subsets.
Figure \ref{fig:accuracy} shows the primal function value, dual function value, duality gap, and accuracy on the test set for {\wa} and {\ijcnn} datasets. Even measured against wall time, {\nms} is very competitive with {\uniform} in all the metrics.
\begin{figure}[!h]
	\includegraphics[width=\linewidth]{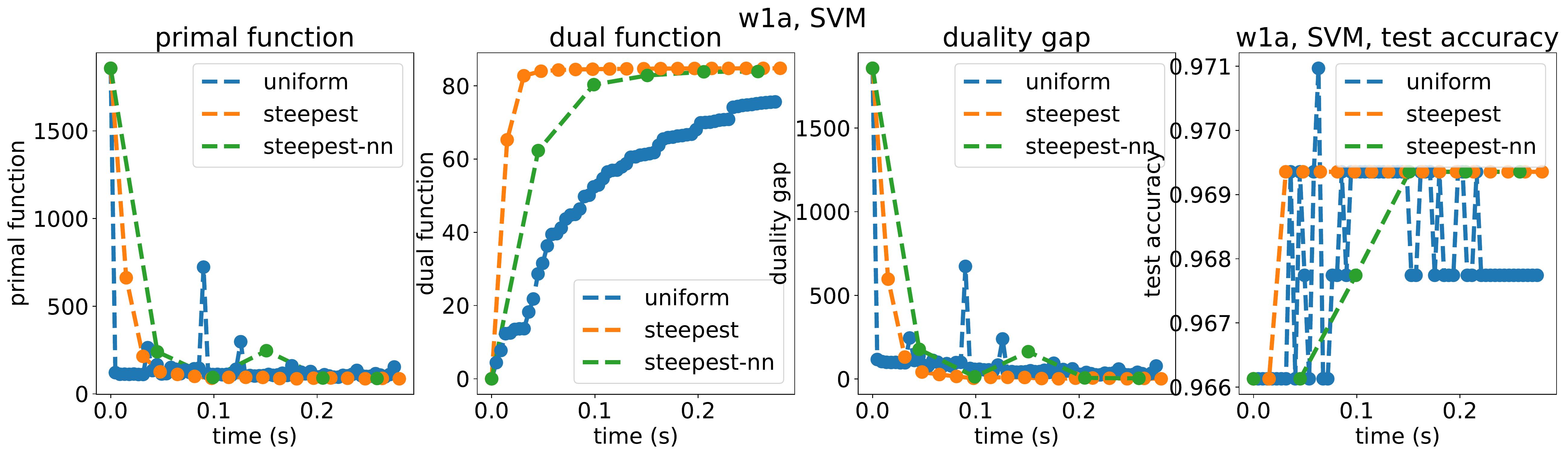}
	\includegraphics[width=\linewidth]{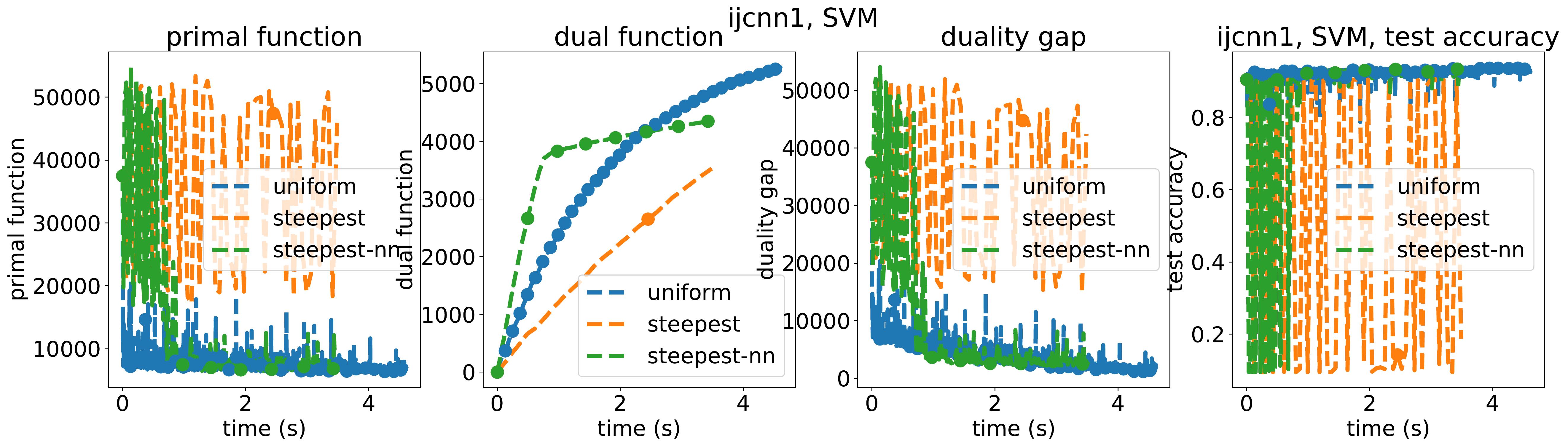}
	\caption{Primal function, dual function, duality gap and accuracy: Even measured against wall time, {\nms} is very competitive with {\uniform} in all the metrics.}\label{fig:accuracy}
\end{figure}

\subsection{Density of the solution}

In Figure \ref{fig:density}, we plot the number of non-zero coordinates of current solution as a function of time for datasets {\sector}, {\rcv} for Lasso and {\wa}, {\ijcnn} for SVM. Recall from Figures \ref{fig:time} and \ref{fig:accuracy} that {\nms} is competitive (and sometimes much better) than {\uniform} in terms of optimization error or accuracy especially at the start. Figure \ref{fig:density} shows that the solutions obtained by {\nms} are much sparser at the start. So if we want a \emph{quick}, sparse and accurate solution, then {\nms} could be the algorithm of choice. For SVMs sparsity translates to a small set of support vectors.

\begin{figure}[!h]
    \hfill
	\includegraphics[width=0.24\linewidth]{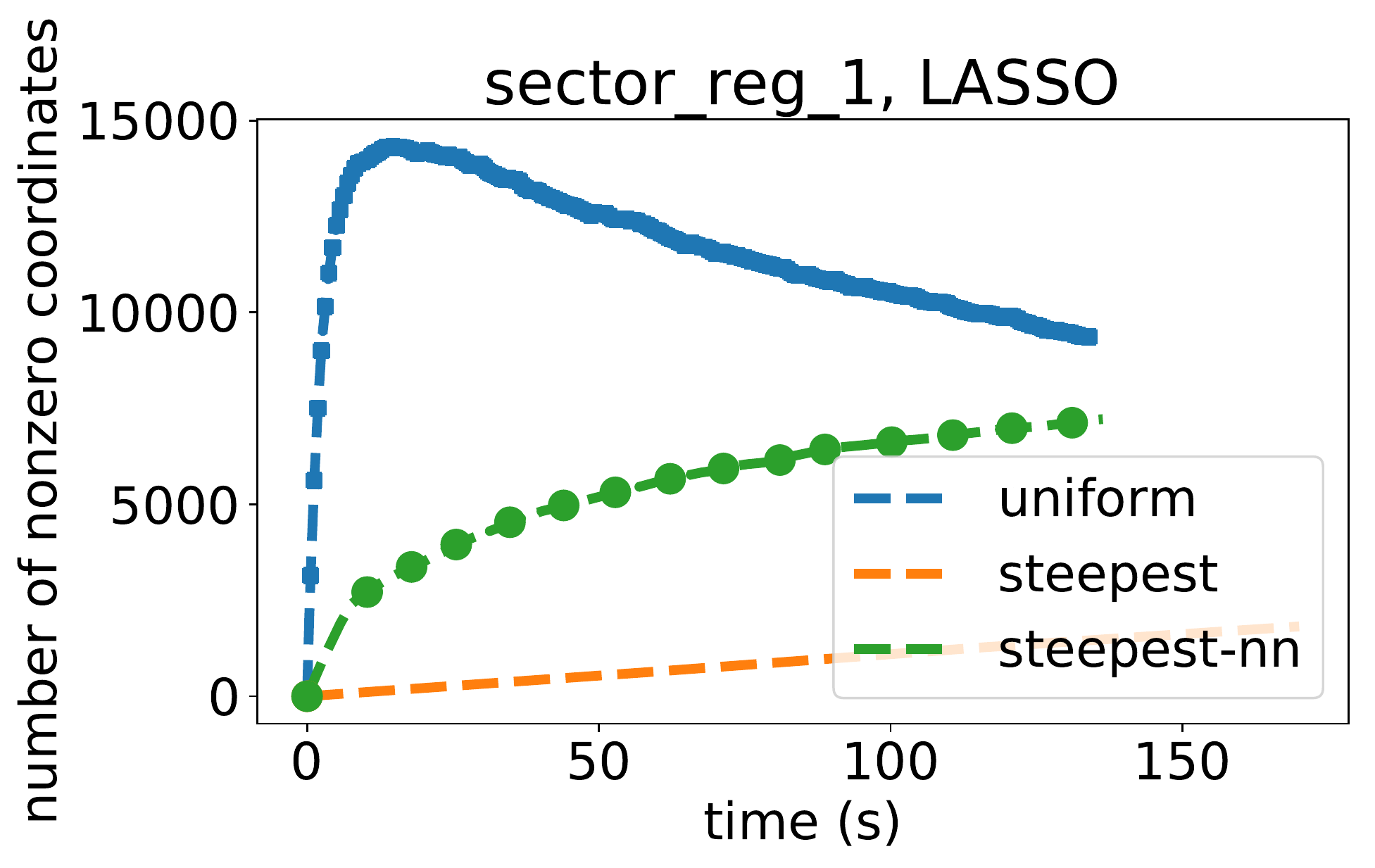}
	\includegraphics[width=0.24\linewidth]{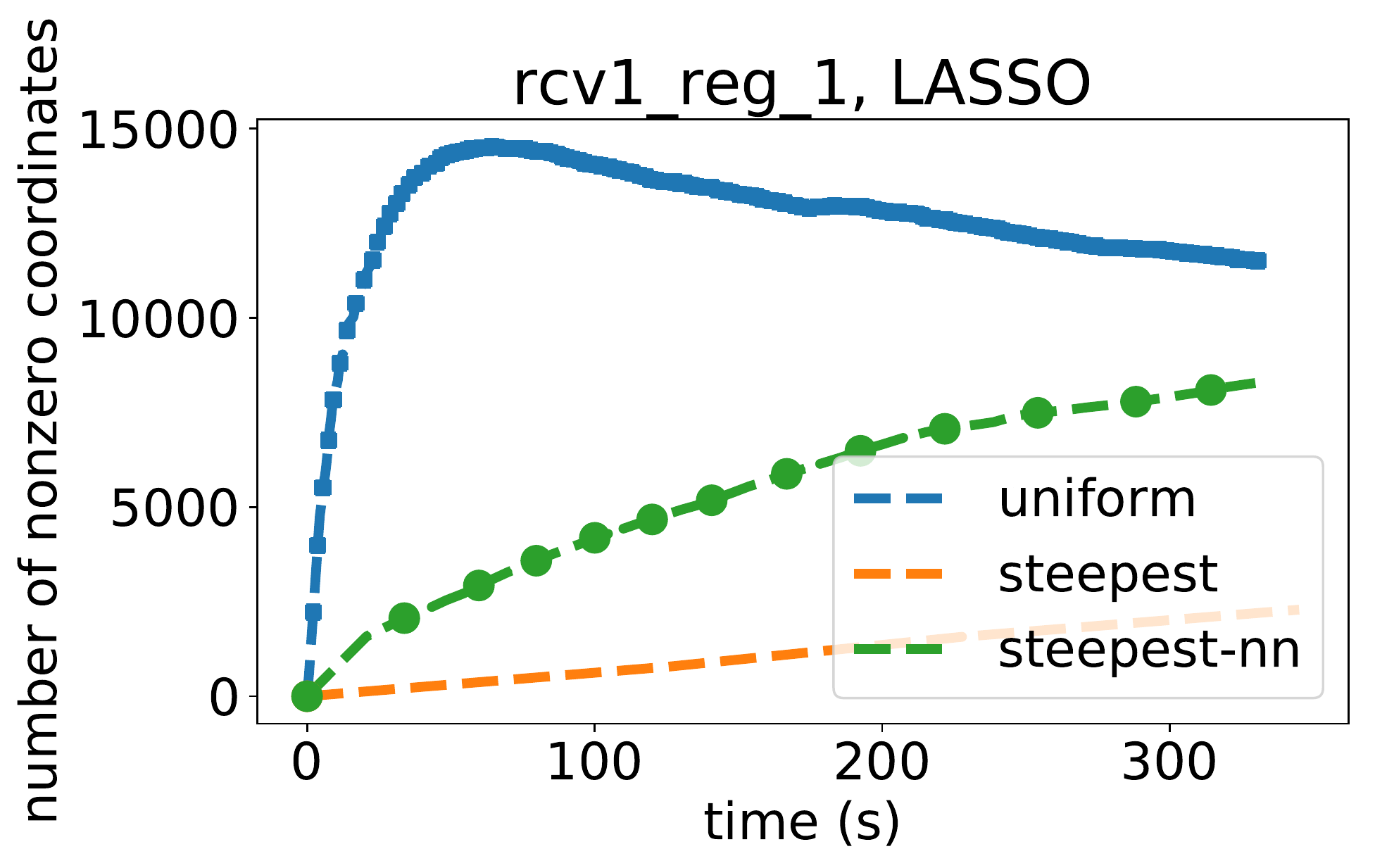}
	\includegraphics[width=0.24\linewidth]{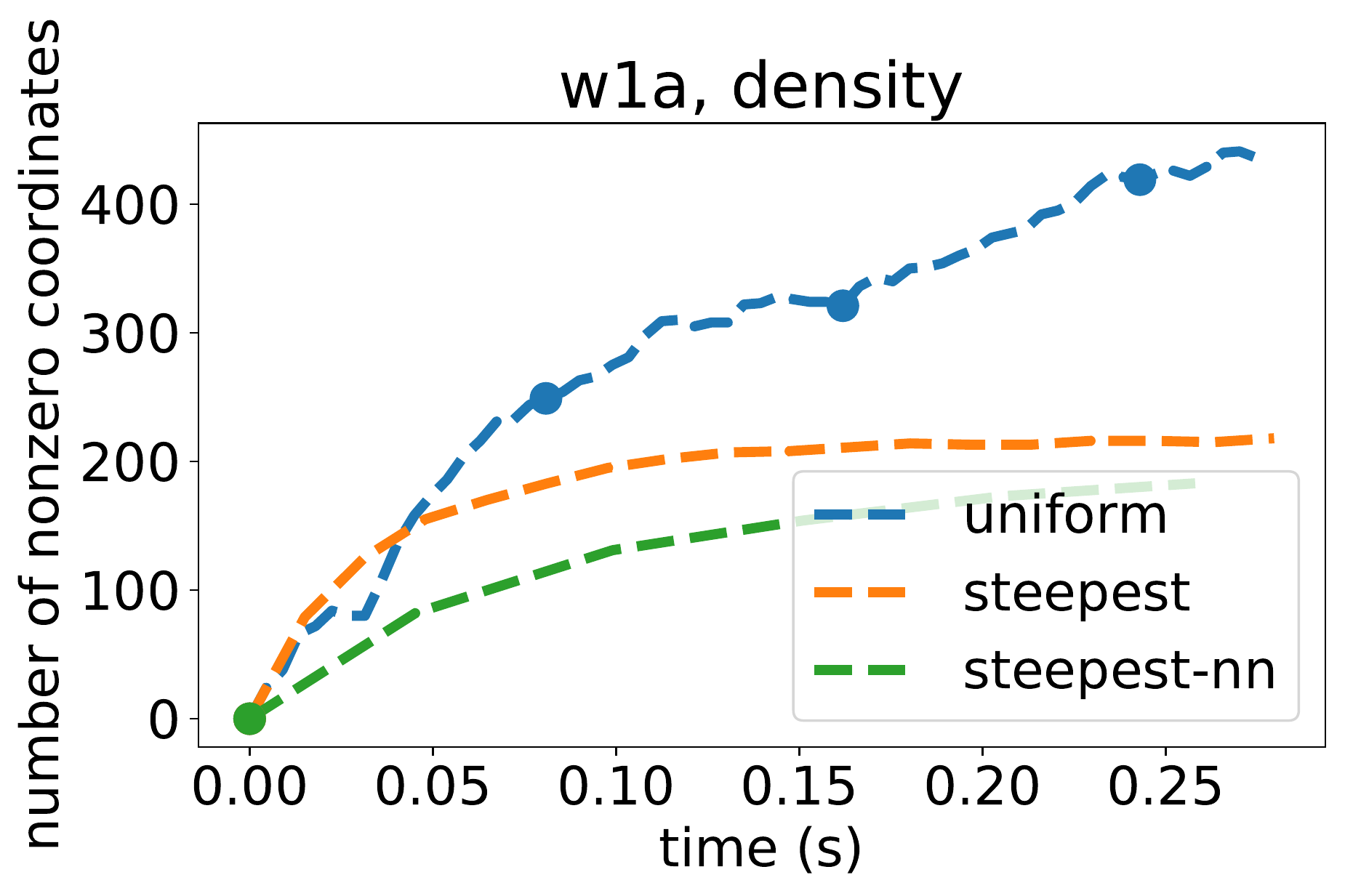}
	\includegraphics[width=0.24\linewidth]{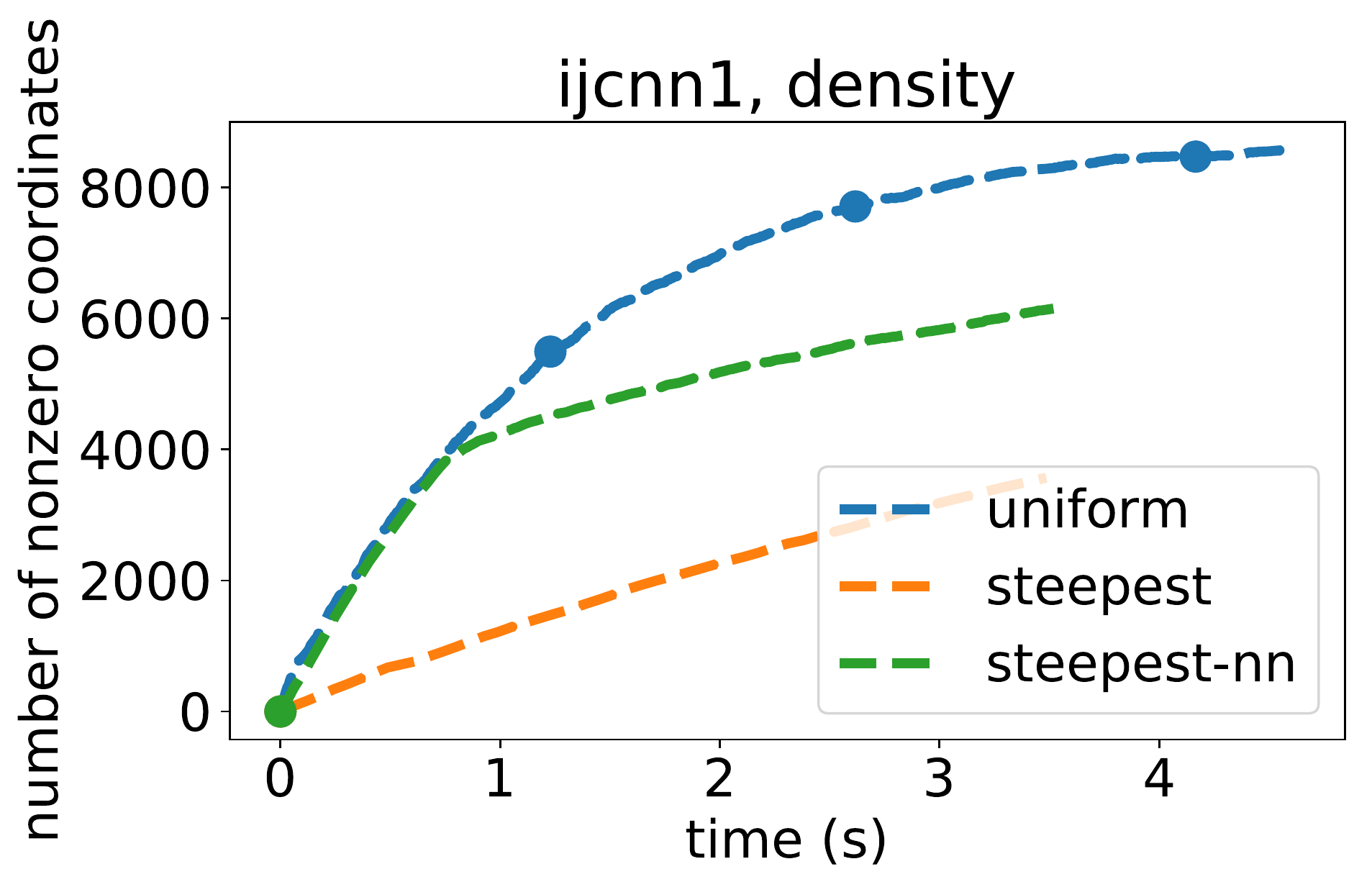}
	\caption{Density: The solutions obtained by {\nms} are much sparser than {\uniform}, especially towards the beginning.}\label{fig:density}
\end{figure}

\subsection{Effect of Adaptivity}

Here we investigate whether our adaptive choosing of the subsets has any adverse effect on the performance of {\nmslib} to solve the $\MIPS$ problem. At each iteration during the course of our algorithm, we compute the value of dot product of the query vector with i) the optimal point over all candidates ($\MIPS$), ii) the optimal point over only the subset of the candidates ($\SMIPS$), iii) result of running {\nmslib} on all candidates ($\MIPS$-nn), and vi) the result of running LSH on the masked subset ($\SMIPS$-nn). Comparing (i) and (iii) shows the performance of the {\nmslib} algorithm, and comparing (ii) and (iv) shows how well the {\nmslib} algorithm handles our adaptive queries. The results in Figure \ref{fig:adaptivity} show that indeed the influence of adaptivity is negligible - both ($\MIPS$-nn) and ($\MIPS$) are close to ($\SMIPS$-nn) and ($\SMIPS$) respectively. What is surprising though is the overall poor performance of the {\nmslib} algorithm even after spending significant effort tuning the hyper-parameters as evidenced by the gap between ($\MIPS$-nn) and ($\MIPS$). This strongly suggests that improving the underlying algorithm for solving our $\SMIPS$ instances could lead to significant gains.

\begin{figure}[!h]
	\includegraphics[width=0.24\linewidth]{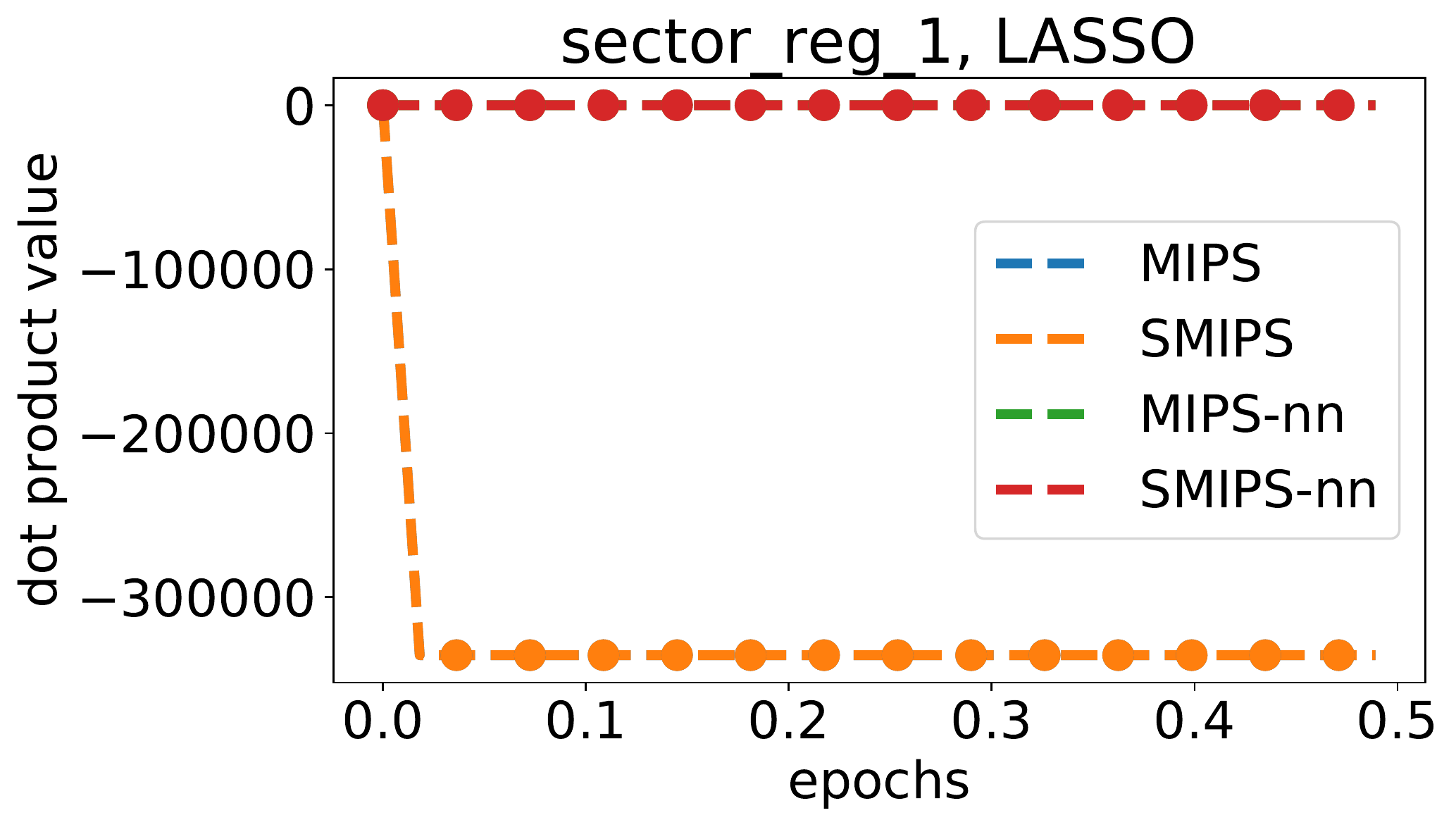}
	\includegraphics[width=0.24\linewidth]{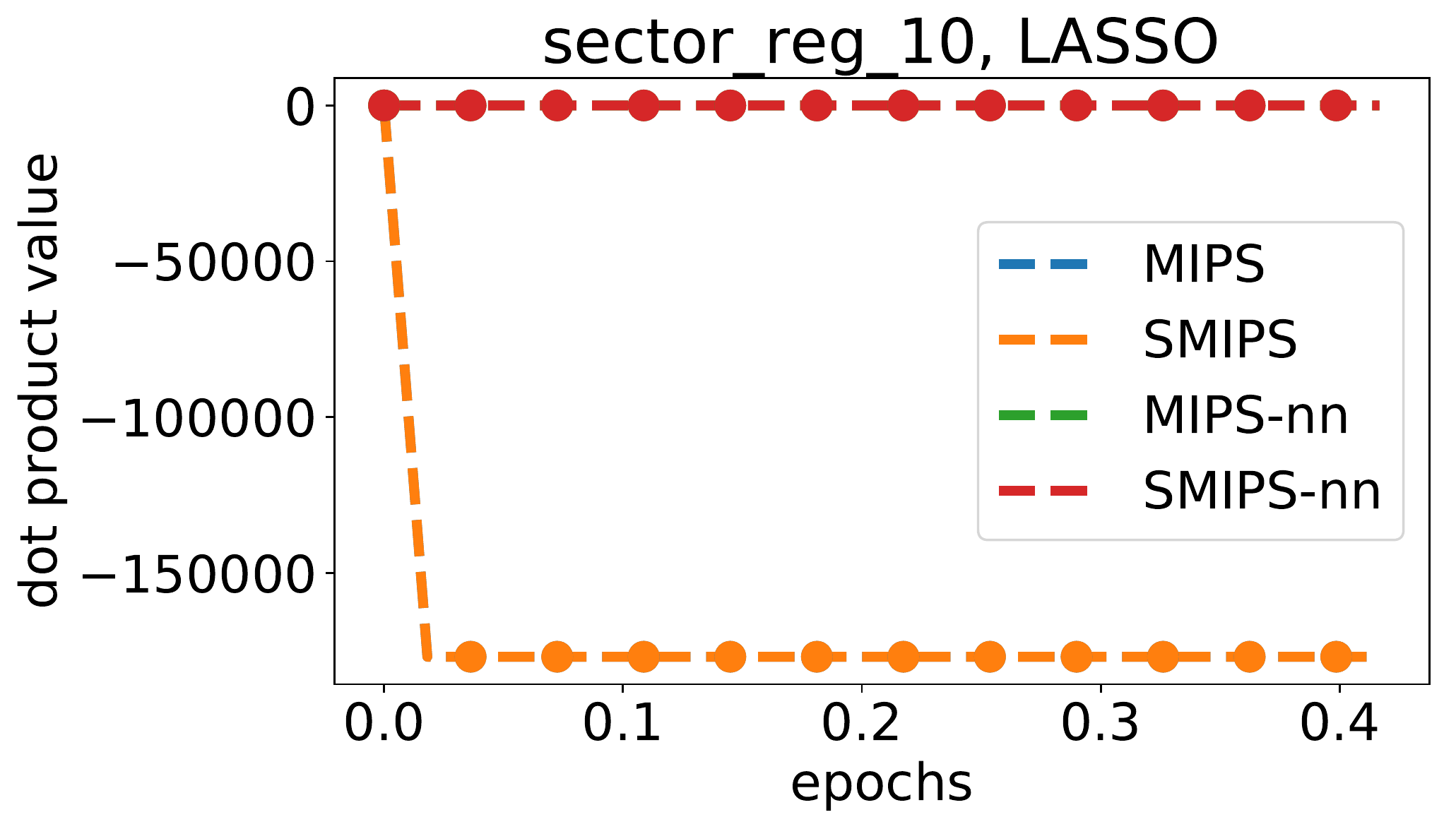}
	\includegraphics[width=0.24\linewidth]{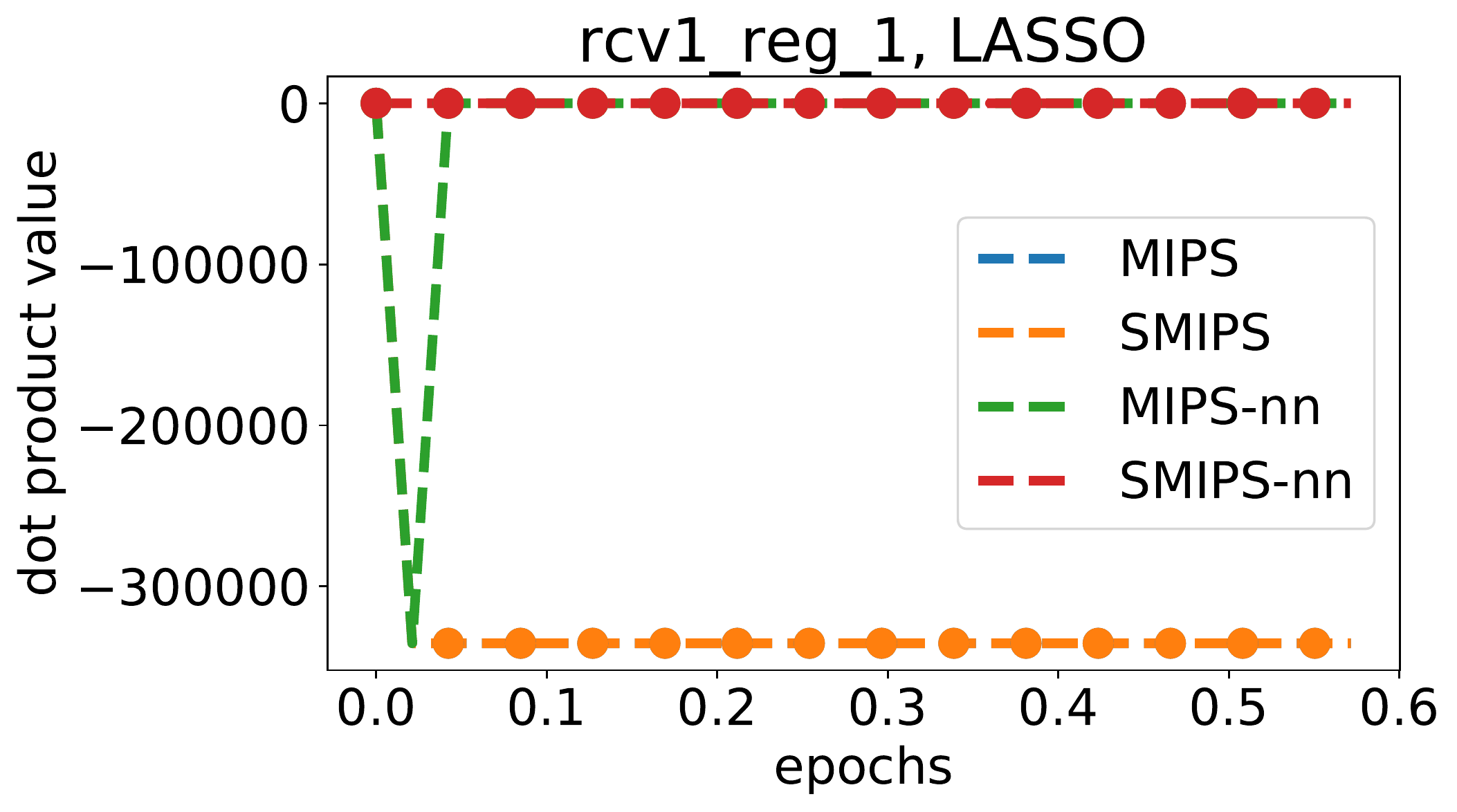}
	\includegraphics[width=0.24\linewidth]{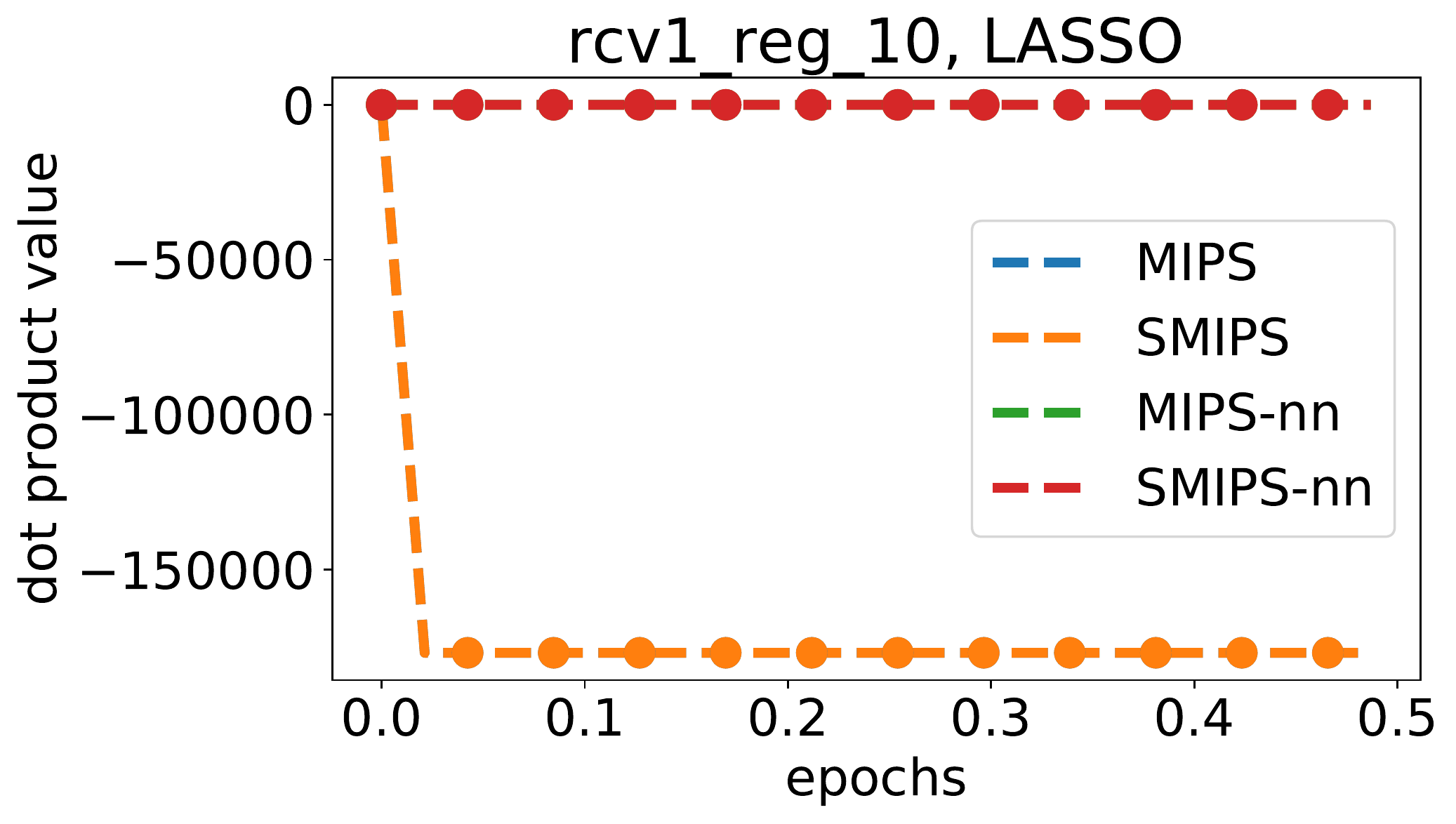}
	\caption{Adaptivity: ($\MIPS$-nn) and ($\MIPS$) are close to ($\SMIPS$-nn) and ($\SMIPS$) respectively indicating that choosing subsets adaptively does not affect {\nmslib} algorithm substantially. However the gap between ($\MIPS$-nn) and ($\MIPS$) indicates the general poor quality of the solutions returned by {\nmslib}.}\label{fig:adaptivity}
\end{figure}

\subsection{GS-s implementation using {\falconn} library}

\begin{figure}[!h]
	\includegraphics[width=0.24\linewidth]{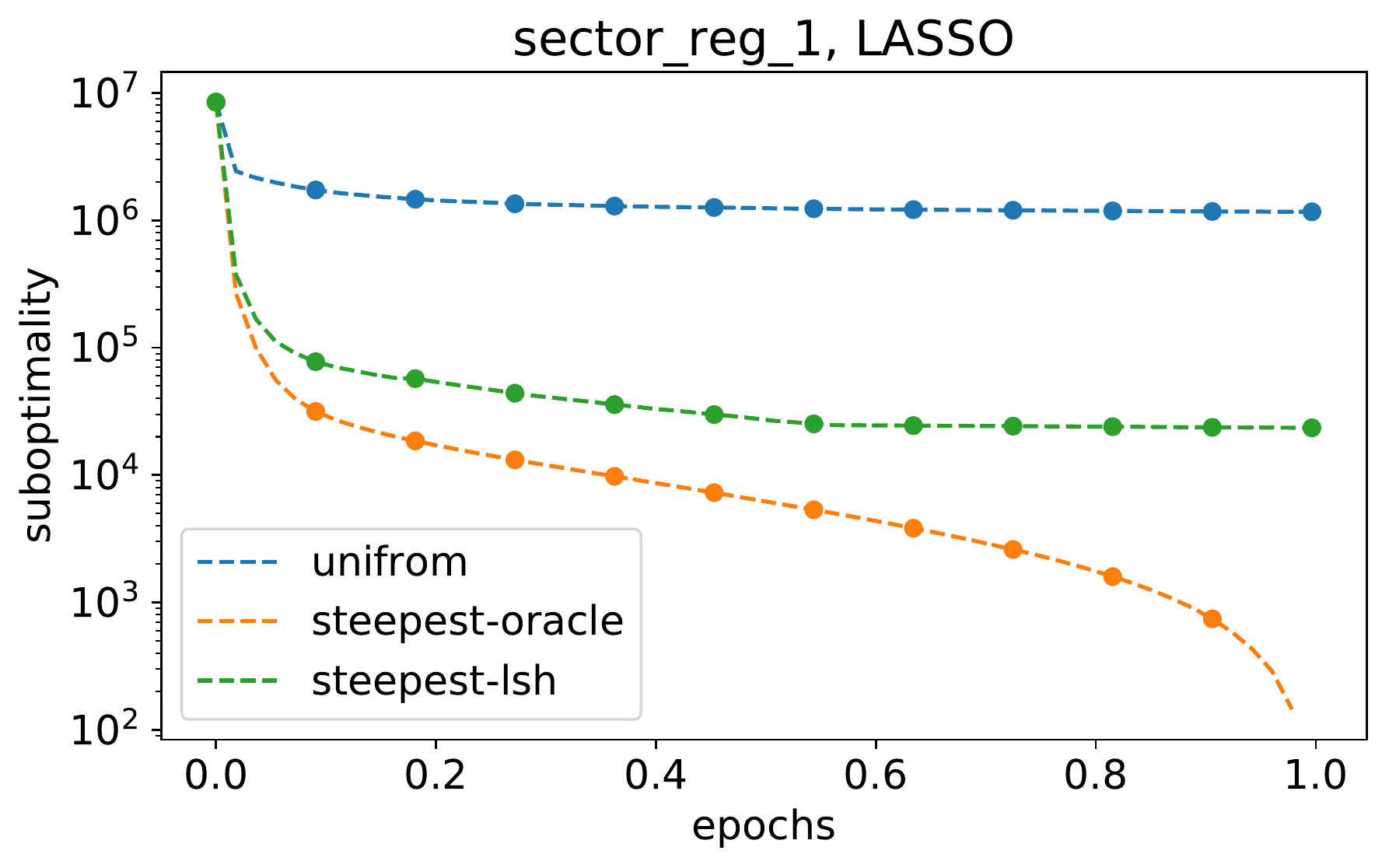}
	\includegraphics[width=0.24\linewidth]{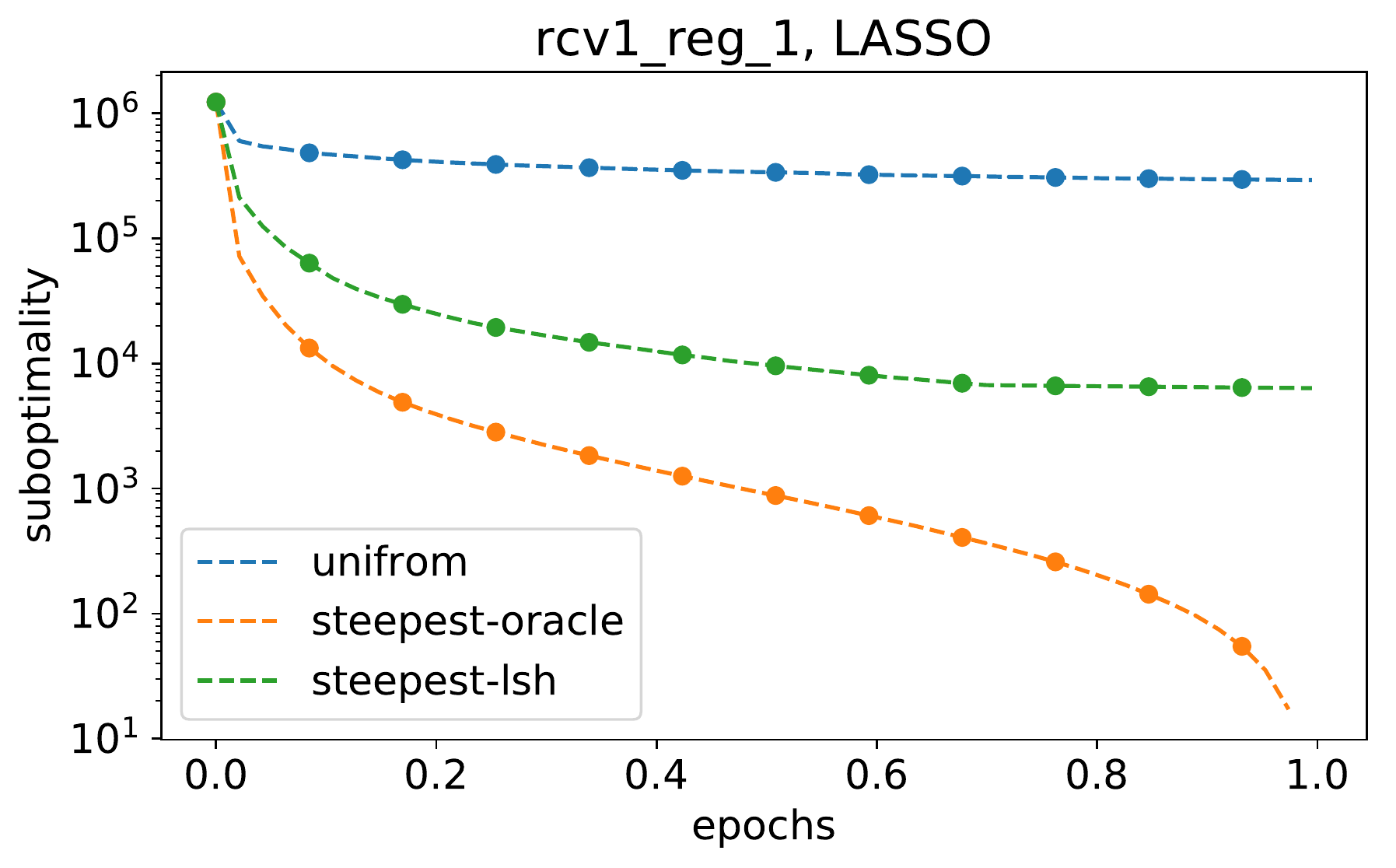}
	\includegraphics[width=0.24\linewidth]{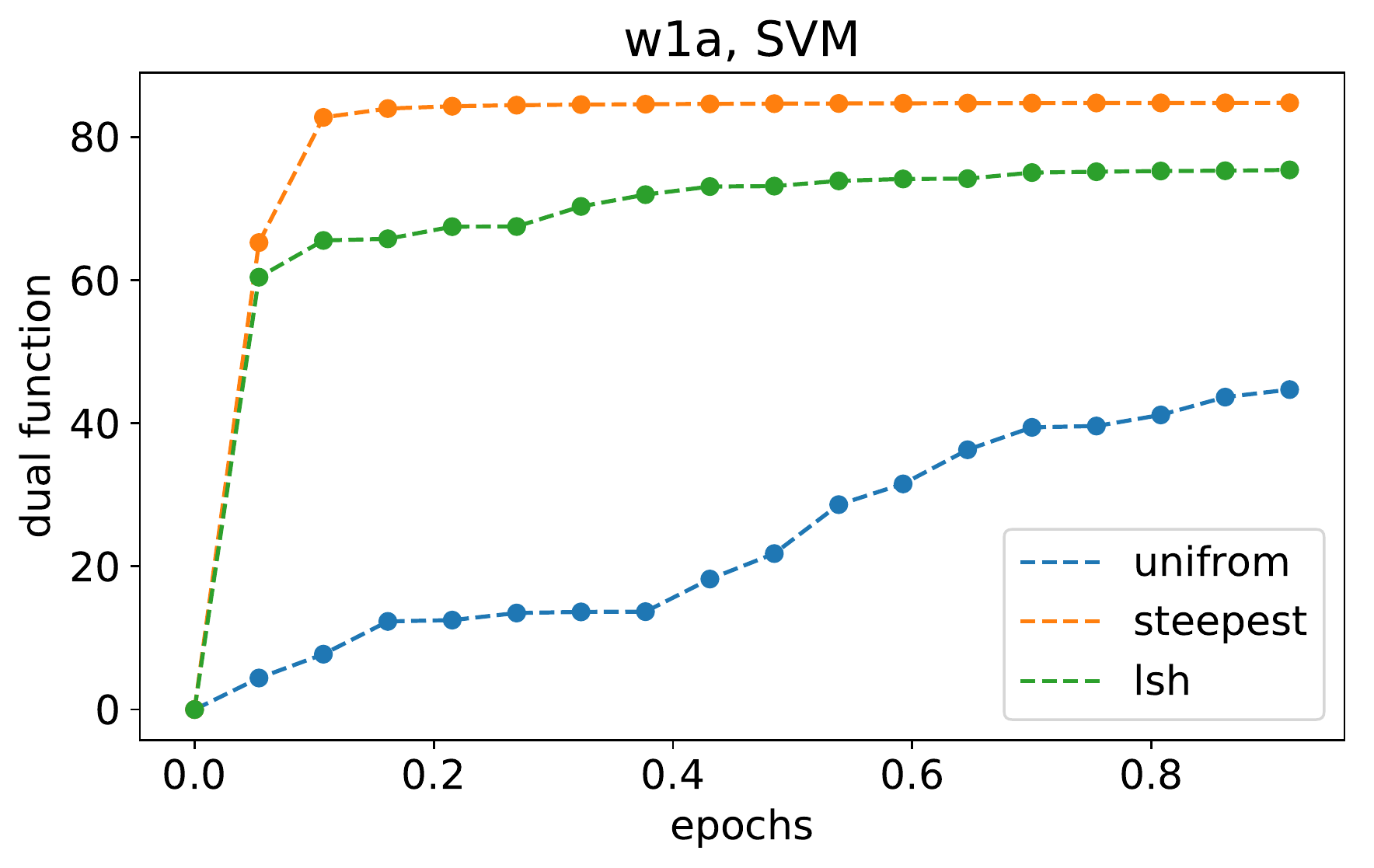}
	\includegraphics[width=0.24\linewidth]{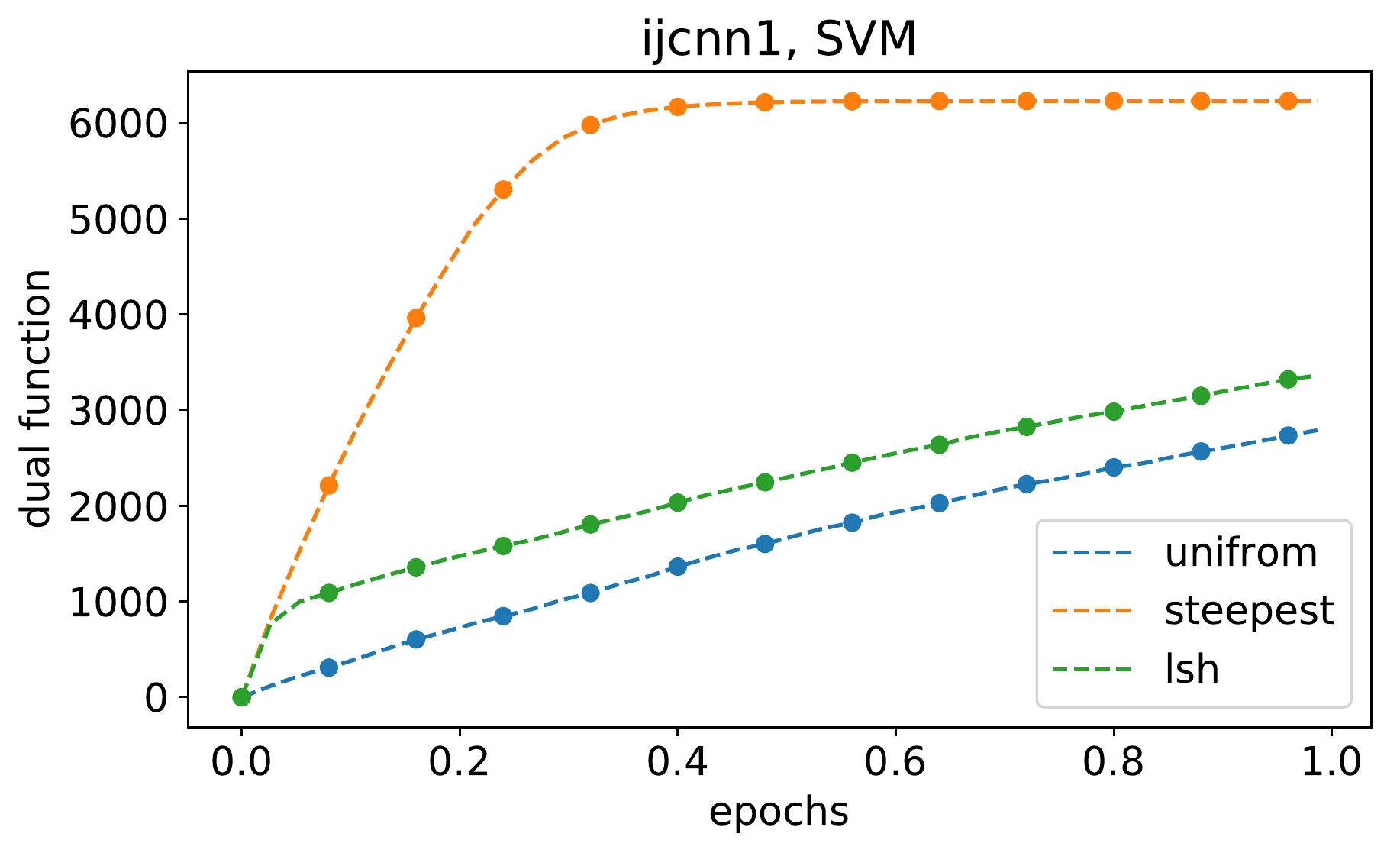}
	\caption{{\steepest} as well {\lsh} significantly outperform {\uniform} in number of iterations.}\label{fig:lsh-iterations}
\end{figure}

\begin{figure}[!h]
	\includegraphics[width=0.24\linewidth]{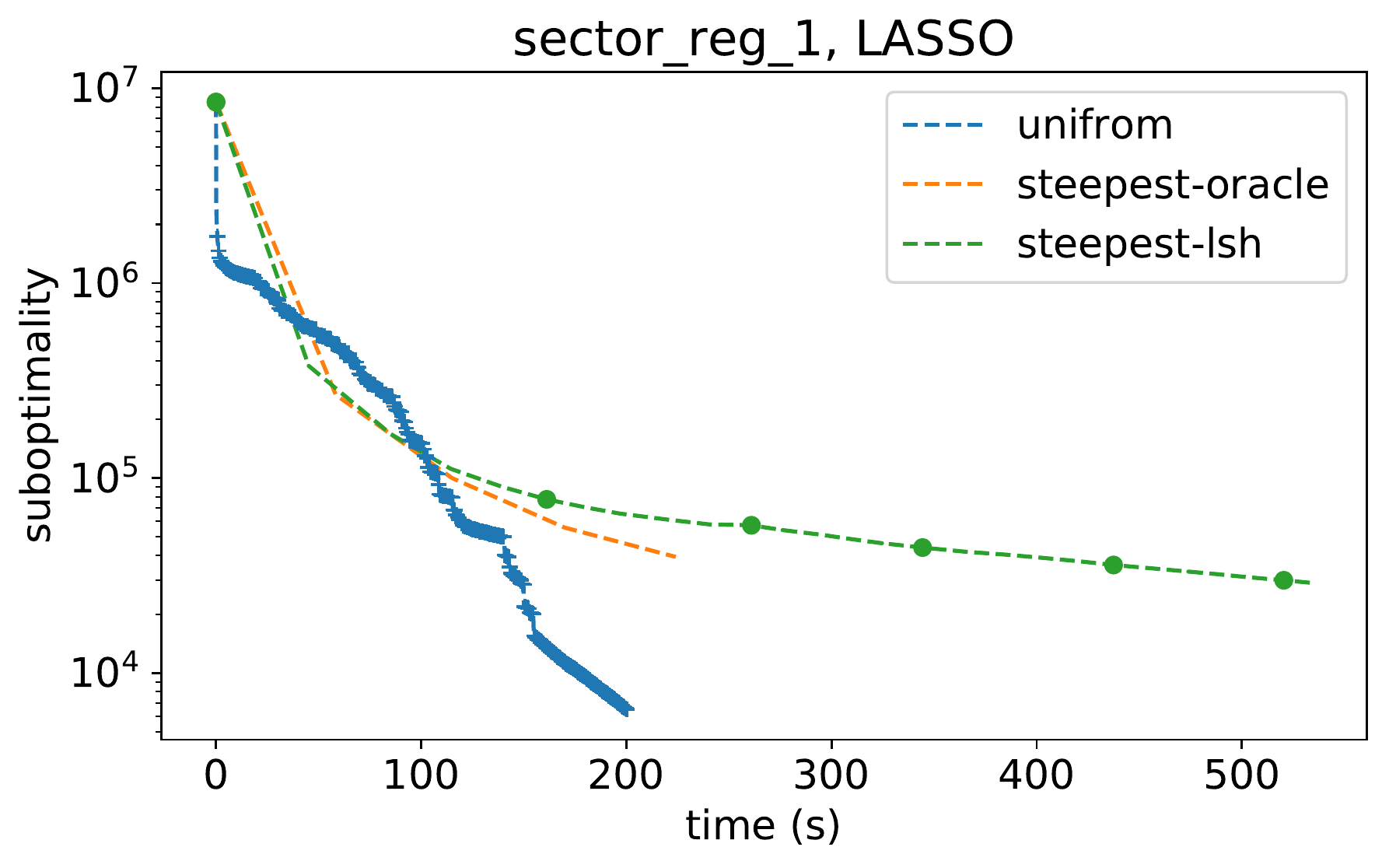}
	\includegraphics[width=0.24\linewidth]{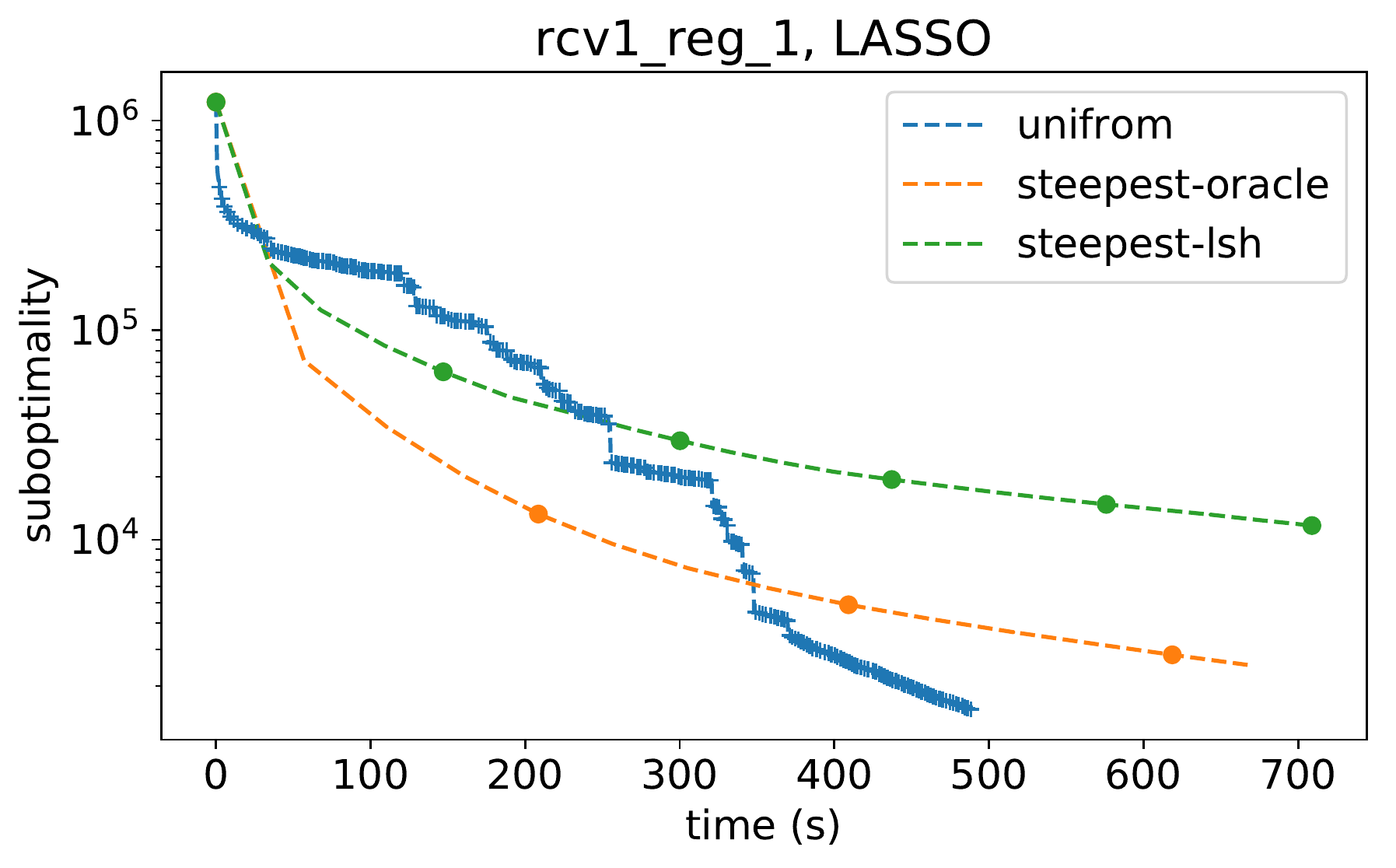}
	\includegraphics[width=0.24\linewidth]{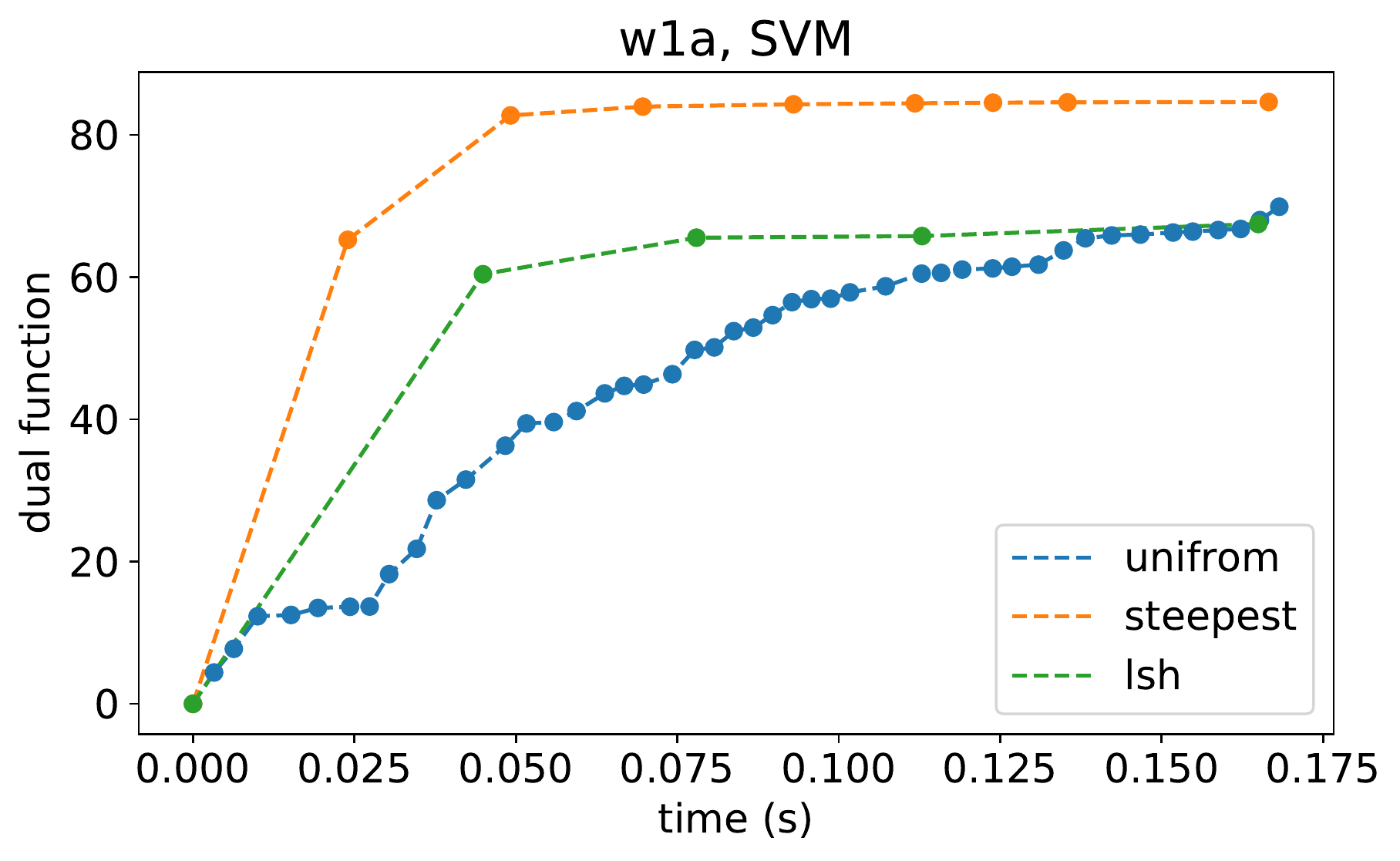}
	\includegraphics[width=0.24\linewidth]{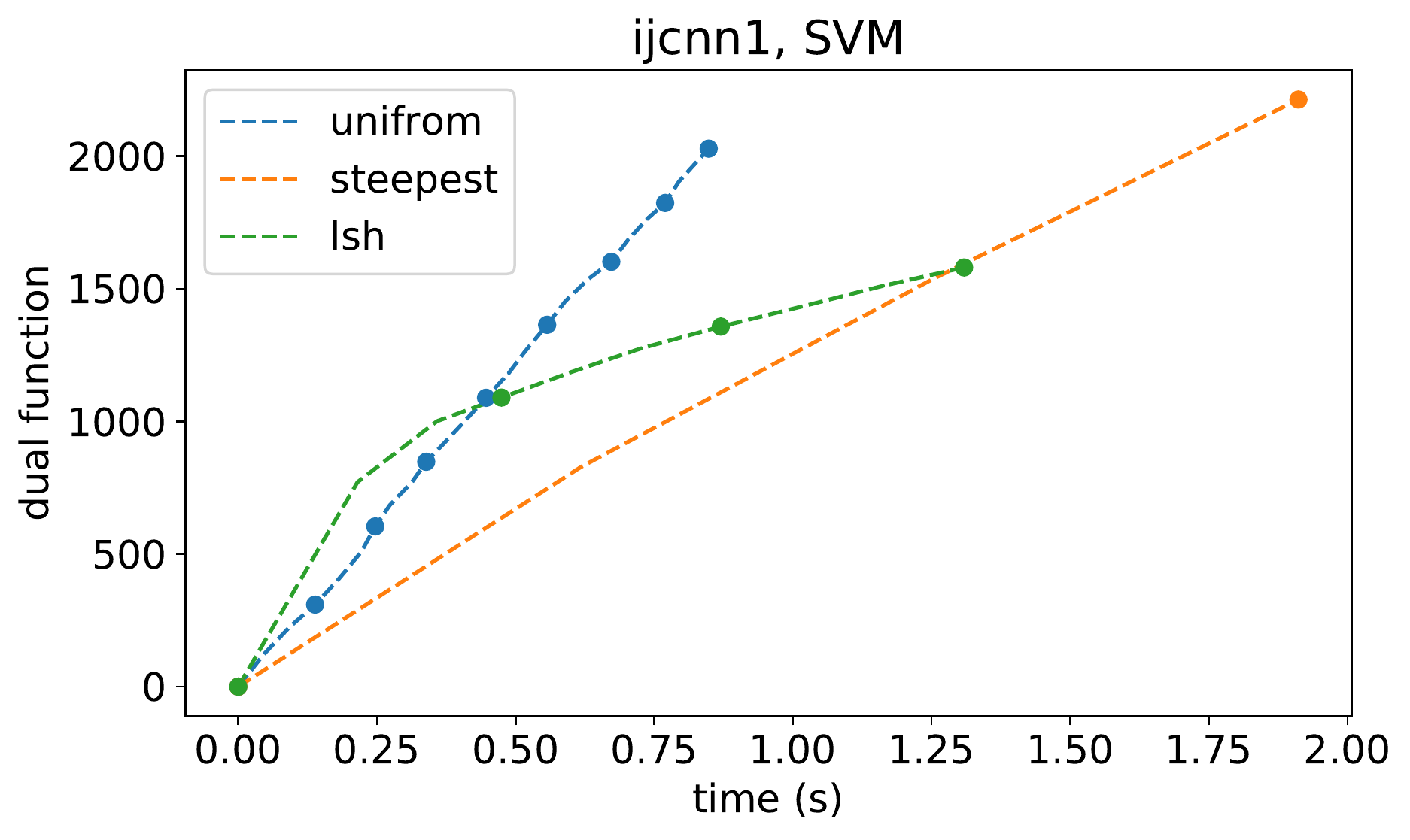}
	\caption{{\lsh} is very competitive and sometimes outperforms {\uniform} even in terms of wall time especially towards the beginning. However eventually the performance of {\uniform} is better than {\lsh}. This is because as the norm of the gradient becomes small, the \emph{hyperplane} LSH algorithm used performs poorly.}\label{fig:lsh-time}
\end{figure}

\begin{figure}[!h]
	\includegraphics[width=0.24\linewidth]{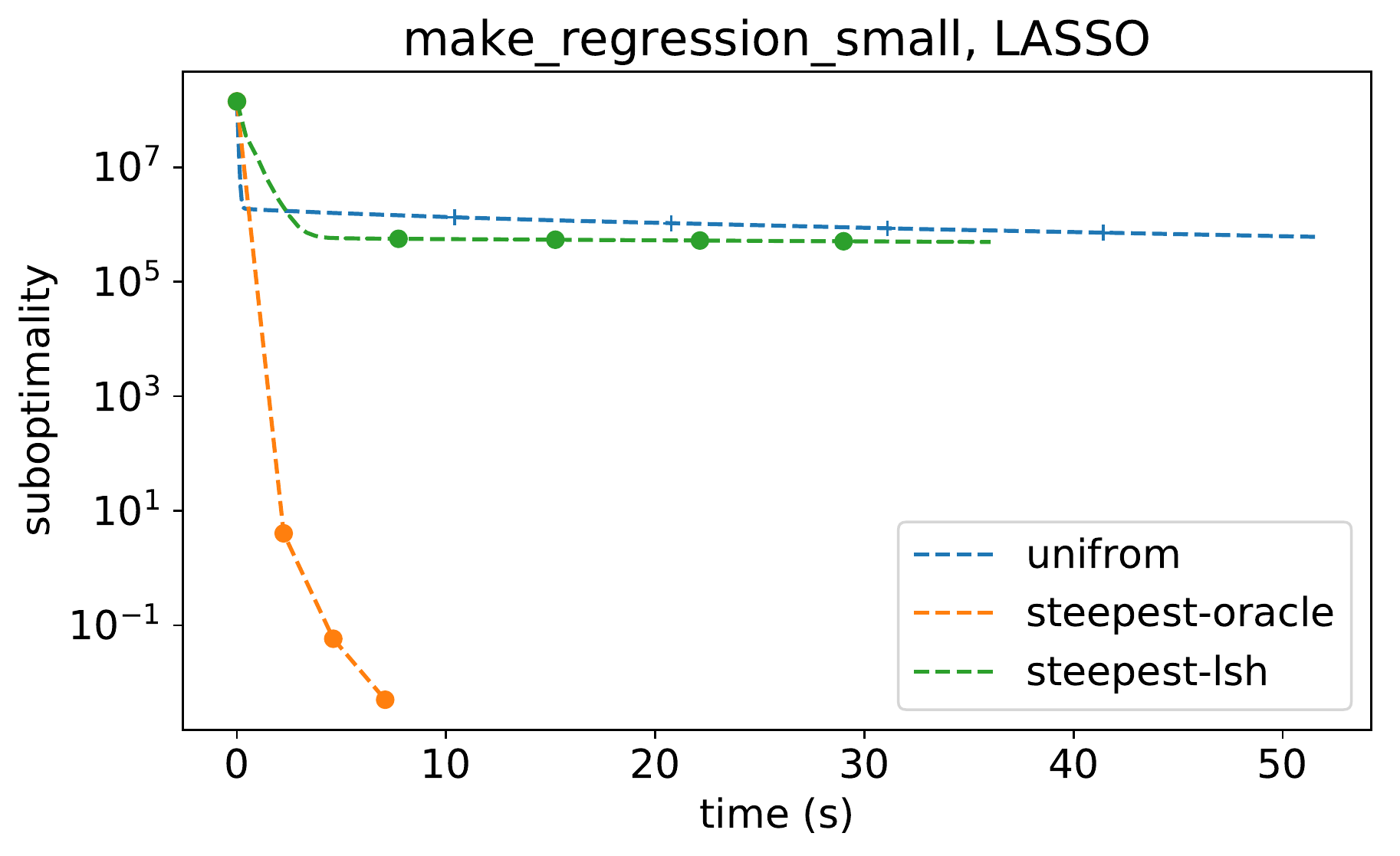}
	\includegraphics[width=0.24\linewidth]{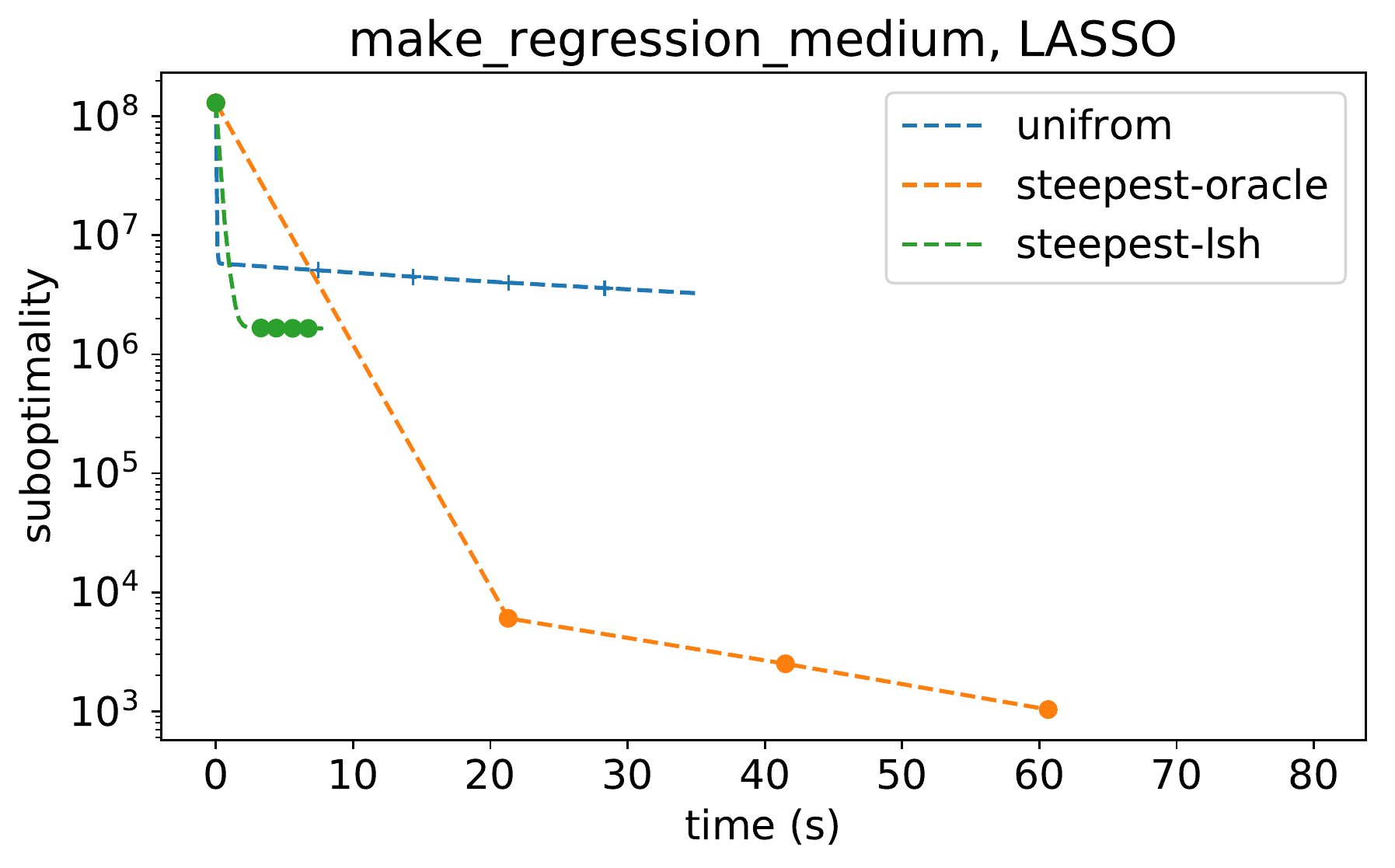}
	\includegraphics[width=0.24\linewidth]{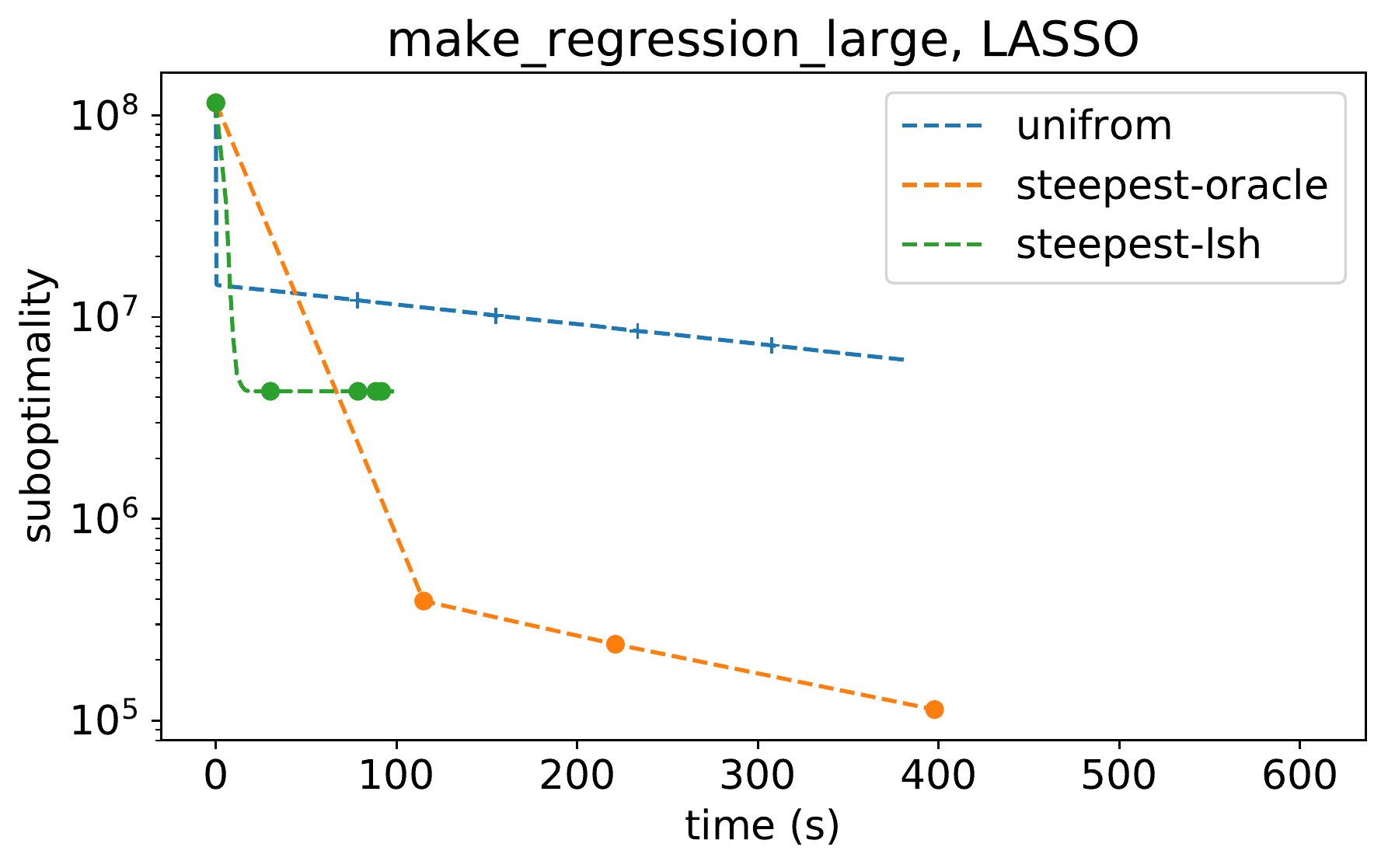}
	\caption{Performance on {\makeregression} with $n \in \{10^4,10^5,10^6\}$. The advantage of {\lsh} over {\uniform} consistently increases with increasing problem sizes, eventually significantly outperforming it.}\label{fig:lsh-make-regression}
\end{figure}

\begin{table}[]
	\centering
	\caption{Datasets and {\falconn} hyper-parameters: Lasso is run on {\rcv} and {\sector}, and SVM on {\wa} and {\ijcnn}. $L$ is fixed at 10, and $k = \floor{\log(n) - h}$.\vspace{1mm}}
	\label{tab:datasets-lsh}
	\begin{tabular}{llllll}
		\hline
		\textbf{Dataset} & {$\mathbf{n}$} & {$\mathbf{d}$} & {$\sqrt{n}\mathbf{\beta}$} &  {$\mathbf{m}$}  & {$\mathbf{h}$}\\ \hline
		{\sector}      &  55,197          & 6,412         & 10    & 90 & 1\\
		{\rcv}         &  47,236          & 15,564        & 1.5  & 90 & 0\\
		{\wa}          &  2,477           & 300           & 10    & 50 & 1\\
		{\ijcnn}       & 49,990         &  22           & 3    & 10 & 0\\ \hline
	\end{tabular}
\end{table}

\begin{figure}[!h]
	\includegraphics[width=0.24\linewidth]{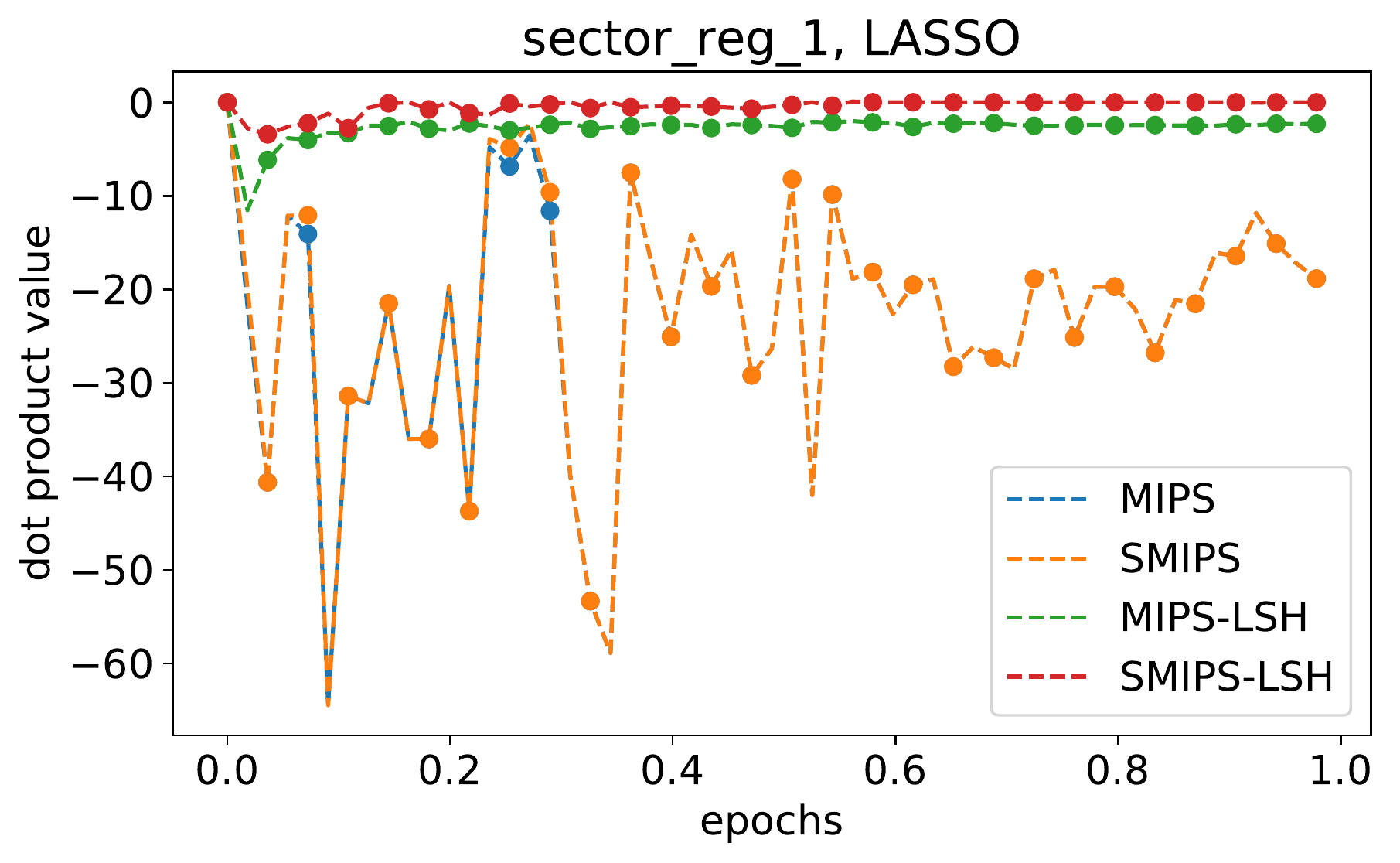}
	\includegraphics[width=0.24\linewidth]{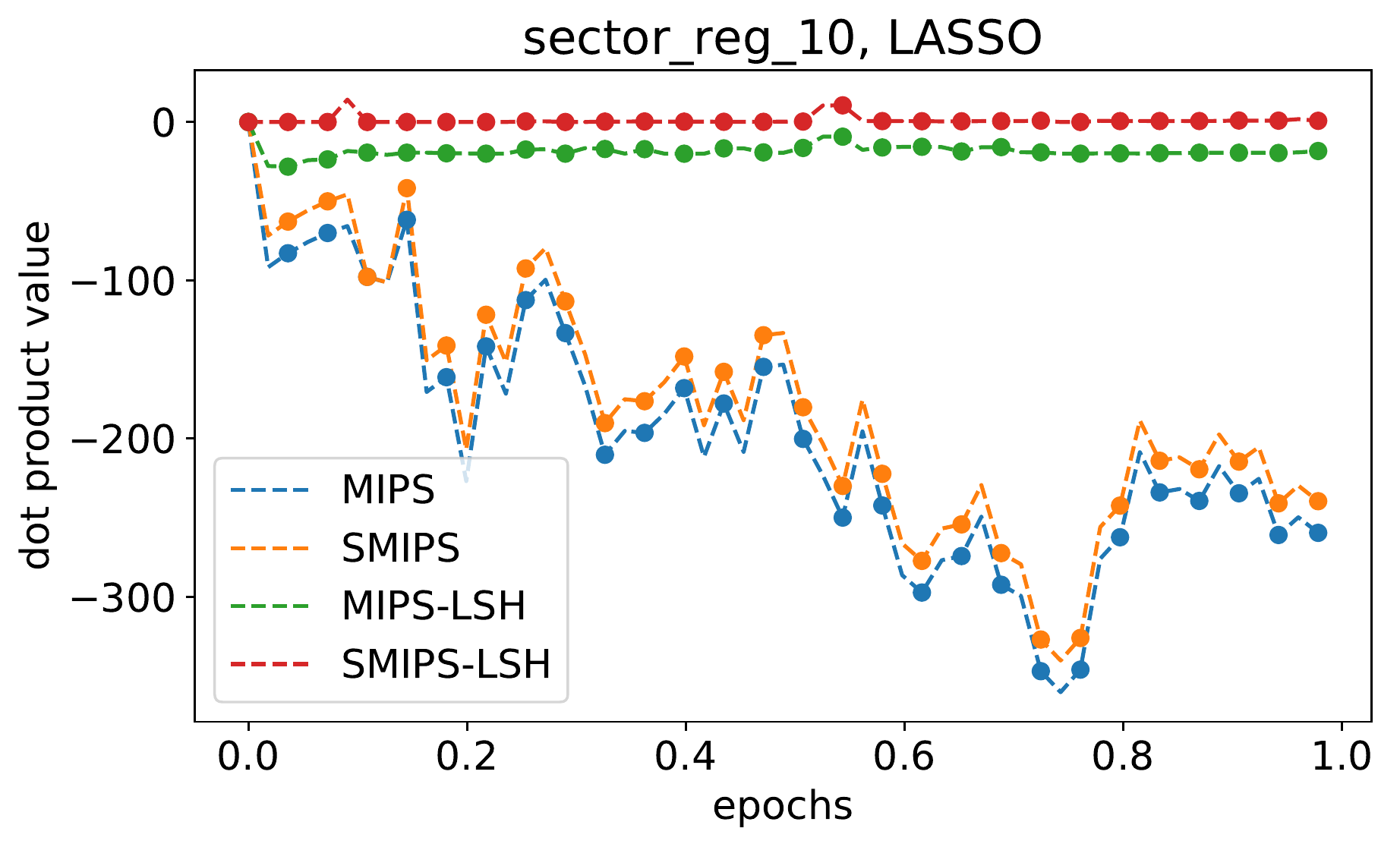}
	\includegraphics[width=0.24\linewidth]{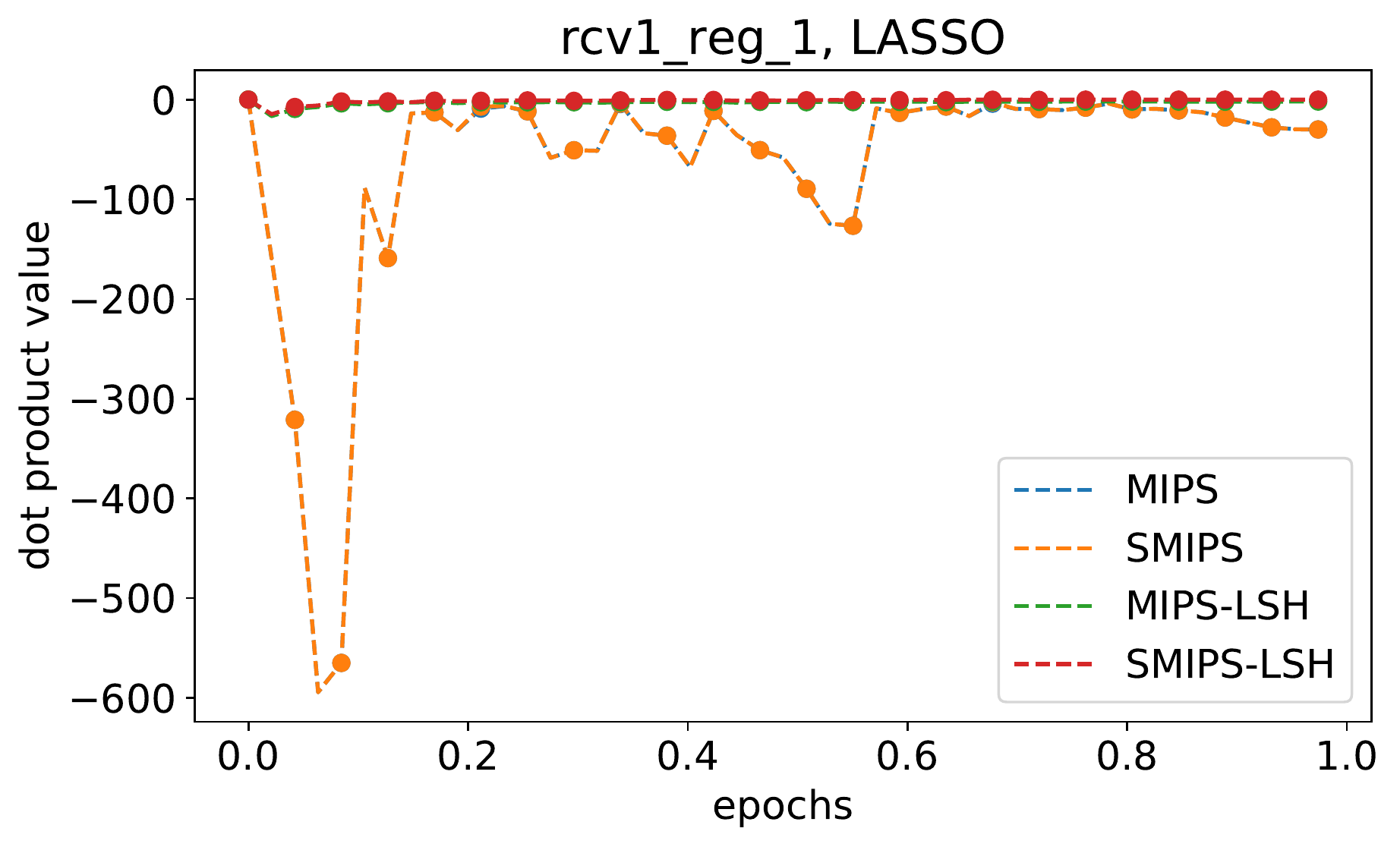}
	\includegraphics[width=0.24\linewidth]{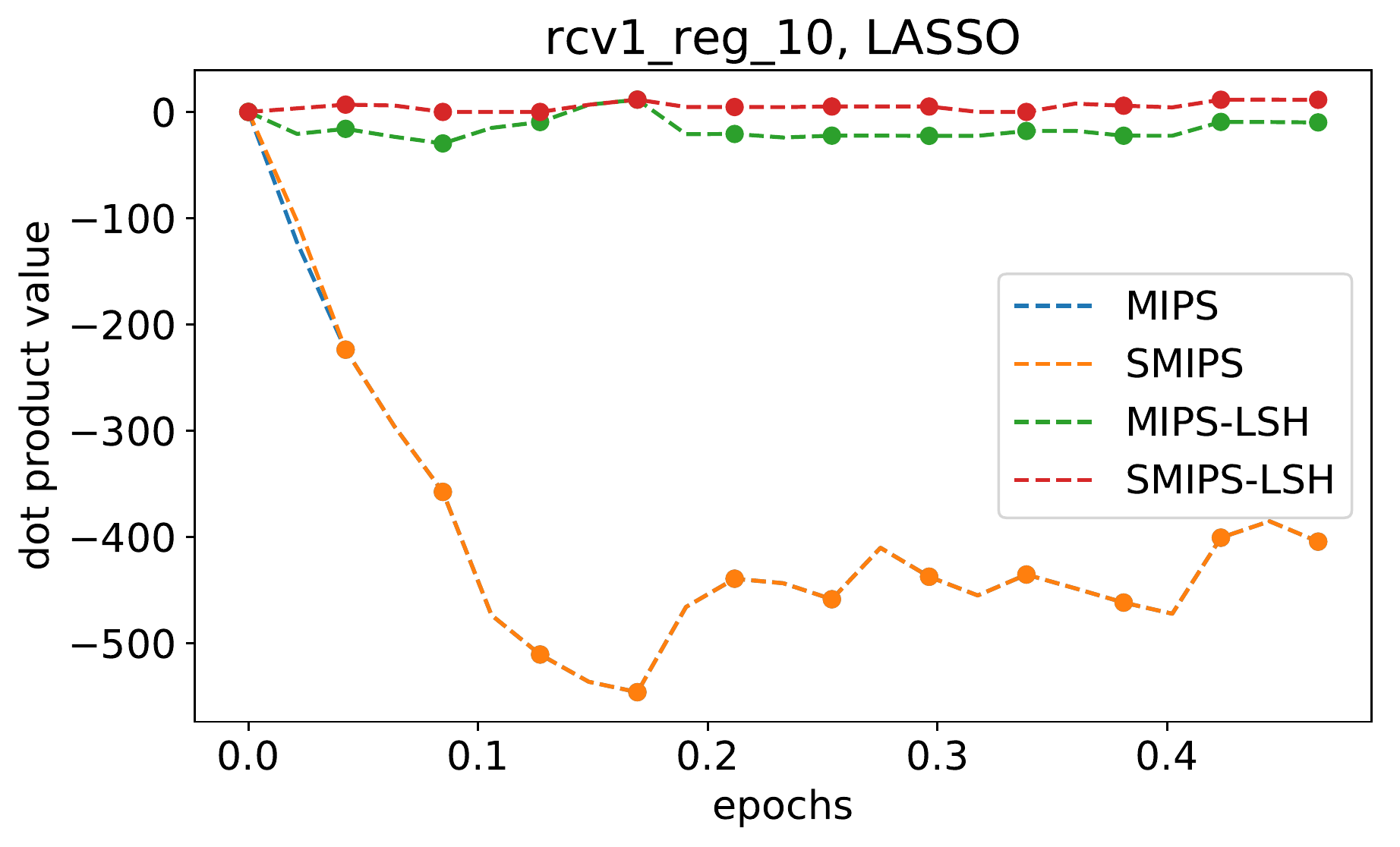}
	\caption{Adaptivity: ($\MIPS$-LSH) and ($\MIPS$) are close to ($\SMIPS$-LSH) and ($\SMIPS$) respectively indicating that choosing subsets adaptively does not affect LSH algorithm substantially. However the gap between ($\MIPS$-LSH) and ($\MIPS$) indicates the general poor quality of the solutions returned by LSH.}\label{fig:adaptivity-lsh}
\end{figure}

We repeat experiments from section \ref{sec:experiments} using {\falconn} library, which is an efficient implemenation of LSH. Used {\falconn} hyper-parameters could be found in the Table \ref{tab:datasets-lsh}. Figure \ref{fig:lsh-iterations} shows the superiority of {\lsh} and {\steepest} over {\uniform} in terms of iterations. Figure \ref{fig:lsh-time} shows their wall-time comparison. Qualitatively, the behaviour of {\lsh} is similar to the {\nmslib} results, whereas quantitive results are better using {\nmslib} library for the SVM and a bit worse for the LASSO.

\end{document}

%% file: img/good-bad-cross-box.pdf_tex
\begingroup%
  \makeatletter%
  \providecommand\color[2][]{%
    \errmessage{(Inkscape) Color is used for the text in Inkscape, but the package 'color.sty' is not loaded}%
    \renewcommand\color[2][]{}%
  }%
  \providecommand\transparent[1]{%
    \errmessage{(Inkscape) Transparency is used (non-zero) for the text in Inkscape, but the package 'transparent.sty' is not loaded}%
    \renewcommand\transparent[1]{}%
  }%
  \providecommand\rotatebox[2]{#2}%
  \ifx\svgwidth\undefined%
    \setlength{\unitlength}{307.32159213bp}%
    \ifx\svgscale\undefined%
      \relax%
    \else%
      \setlength{\unitlength}{\unitlength * \real{\svgscale}}%
    \fi%
  \else%
    \setlength{\unitlength}{\svgwidth}%
  \fi%
  \global\let\svgwidth\undefined%
  \global\let\svgscale\undefined%
  \makeatother%
  \begin{picture}(1,0.40092419)%
    \put(0,0){\includegraphics[width=\unitlength,page=1]{good-bad-cross-box.pdf}}%
    \put(0.08373228,0.17932025){\color[rgb]{0,0,0}\makebox(0,0)[lb]{\smash{$b_1$}}}%
    \put(0.1986253,0.27959053){\color[rgb]{0,0,0}\makebox(0,0)[lb]{\smash{$b_2$}}}%
    \put(0.4242334,0.28167952){\color[rgb]{0,0,0}\makebox(0,0)[lb]{\smash{$g_1$}}}%
    \put(0.61224009,0.20229892){\color[rgb]{0,0,0}\makebox(0,0)[lb]{\smash{$g_2$}}}%
    \put(0.83990489,0.20647682){\color[rgb]{0,0,0}\makebox(0,0)[lb]{\smash{$c$}}}%
    \put(0.10674316,0.01120712){\color[rgb]{0,0,0}\makebox(0,0)[lb]{\smash{Bad Steps}}}%
    \put(0.41600661,0.00812209){\color[rgb]{0,0,0}\makebox(0,0)[lb]{\smash{Good Steps}}}%
    \put(0.74387714,0.01017875){\color[rgb]{0,0,0}\makebox(0,0)[lb]{\smash{Cross steps}}}%
    \put(0.59757187,-0.14288413){\color[rgb]{0,0,0}\makebox(0,0)[lb]{\smash{}}}%
  \end{picture}%
\endgroup%